\documentclass[10pt]{article}
\usepackage[]{amsmath, amsthm, amsfonts, amssymb}
\usepackage[all]{xy}

\usepackage[colorlinks=true,linkcolor=blue]{hyperref}

 \newcommand {\N} {{\mathbb N}}
 \newcommand {\C} {{\mathbb C}}
 \newcommand {\R} {{\mathbb R}}
 \newcommand {\Z} {{\mathbb Z}}
 \newcommand {\Q} {{\mathbb Q}}

 \newcommand {\PP} {{\mathbb P}}
 \newcommand {\calH} {{\mathcal H}}
 \newcommand {\F} {{\mathcal F}}

 \newcommand {\E} {{\mathcal E}}
 \newcommand {\dt} {{\bullet}}
 \newcommand {\calC} {{\mathcal C}}
\newcommand {\calK} {{\mathcal K}}
 \newcommand {\G} {{\mathcal G}}

\newcommand {\T} {{\mathcal T}}

\newcommand {\A} {{\mathcal A}}

\newcommand {\V} {{\mathcal V}}

\newcommand {\cR} {{\mathcal R}}
\newcommand {\D} {{\mathcal D}}

\newcommand {\M} {{\mathcal M}}
\newcommand {\PM} {\mathcal{P}\M}
\newcommand {\ph} {h}

\DeclareMathOperator*{\2lim} {\text{2-}\varinjlim}

\DeclareMathOperator{\im}{im}

\newtheorem{thm}{Theorem}[subsection]
\newtheorem{prop}[thm]{Proposition}
\newtheorem{lemma}[thm]{Lemma}
\newtheorem{cor}[thm]{Corollary}

\newtheorem{ex}[thm]{Example}
\newtheorem{defn}[thm]{Definition}
\newtheorem{remark}[thm]{Remark}
\newtheorem{conj}[thm]{Conjecture}
\newtheorem{constr}[thm]{Construction}

\begin{document}
\title {An abelian category of motivic sheaves }

 \author{Donu Arapura\footnote{Partially supported by the NSF}}
\date{}
% \address{Department of Mathematics\\
%   Purdue University\\
%   West Lafayette IN 47907\\
%   U.S.A.}

\maketitle

\begin{abstract}
  A category of motivic ``sheaves'' is constructed over a variety in characteristic $0$ using
Nori's method. Although the relationship with many alternative constructions 
remains to be clarified, it does have many of the properties one expects. For example,
 it is abelian and  has Betti, Hodge and $\ell$-adic realizations, and it has a Tannakian subcategory
of motivic local systems.
\end{abstract}

\tableofcontents
\bigskip

\section{ Introduction}
The basic homological invariants  of a fibration  of topological spaces $f:X\to S$, 
are the local systems $R^if_*\Q$. When this is a family of complex algebraic varieties
defined over a subfield $k$ of $\C$, there are many  related invariants, such as the
Gauss-Manin connection, the associated variation of mixed Hodge structure, and
the action of the algebraic fundamental group on \'etale cohomology of the fibres.
According to Grothendieck's philosophy, all of  these structures should come from
the motive of the family. My goal here is to make this idea precise in the following
way. Given a field $F$, and a  variety $S$,  as above,  I will construct
an abelian category $\M(S; F)$ of  motivic ``sheaves'' of $F$-modules. The above local systems
can be promoted to objects in $\M(S;\Q)$, and the
associated structures can be obtained by applying appropriate realizations functors.

Before explaining what I will do, let me say a few words about what I won't.
The usual approach to building a category of motives is to start with 
a category of varieties and algebraic correspondences and modify and
complete this in some way. This stays very close to the underlying
geometry which is  good. On the other hand, it is usually very hard  to prove for example
that what  one gets is (derived from) an abelian category. 
A more pragmatic
approach is to take a system of compatible realizations. This usually has
good categorical properties, but is somewhat ad hoc in nature; and in the relative
setting, it would be appear that any such approach would be necessarily
very technical (e.g. \cite{saito2}). 

Here I want to take a middle path
first blazed by Nori while building a category of motives over a field
\cite{nori-mot,levine}. The approach appeals
to a particular realization at the outset, but is essentially
geometric in its character. Since the construction is not that widely
known, I will indicate the basic idea starting with a toy model and
then refining it. In fact, one of the goals of this paper is to give an
exposition of some, although not all, aspects of Nori's construction.  
 Consider the category of  $k$-algebraic
varieties $Var_k$. Since we assume that $k\subseteq \C$, we
may apply singular (or Betti) cohomology $H^i$ to obtain a
contravariant functor from $Var_k$
 to  the category $\Q$-mod  of finite dimensional $\Q$-vector
spaces.  The key point is that $H^i(X)$ is not just a vector space, but a module over
the ring of natural transformations $End(H^i)$, or a comodule over the ``dual'' coalgebra
$End^\vee(H^i)$ (section~\ref{section:Flin} ) which is technically better behaved.
The category $\M_i'(k)$ of finite dimensional comodules over this
coalgebra forms an  approximation to Nori's category.
It is  abelian, and the  objects $M\in \M'_i(k)$ are
not too wild, in that they admit  presentations
of the form
$$\bigoplus_j H^i(X_j) \to \bigoplus_k H^i(Y_k)\to M\to 0$$
Furthermore, each object $M$ also carries a 
canonical mixed Hodge structure   and (after tensoring with $\Q_\ell$)
an action of  $Gal(\bar k/k)$ as one would hope. So far so
good, but  it would be better to include the
various $\M_i'$ into a single category $\M$, so that standard exact
sequences respect the $\M$-structure. Toward this end, it is necessary to modify
the basic construction by incorporating boundary operators into the foundations.
Thus instead of starting with $Var_k$, the source category
 $\Delta$ consists of triples $(X,Y\subset X ,i\in \N)$ and the
appropriate notion of  morphism, which includes abstract boundary maps
$(X,Y,i)\to (Y,\emptyset, i-1)$. This is really  a partial category  in the sense that the
composition law is only partially  defined; nevertheless the basic constructions go though.
The category of comodules over $End^\vee(H)$, where $H:\Delta\to \Q\text{-mod}$ is the functor
sending  $(X,Y,i) \mapsto H^i(X,Y)$,
yields a rational version of Nori's category of effective
cohomological  motives.
Following the usual pattern, the category $\M(k)$ is obtained by
inverting the Tate motive. This step can be built in from the
beginning, and we find it convenient to do so. 

Now turning to the general case,
the building blocks for $\M(S;F)$ are quadruples
consisting of a quasiprojective family $f:X\to S$, a closed subvariety $Y\subset X$ and
 indices $i\in \N, w\in\Z$. 
This is subject to   a further technical admissibility condition (definition~\ref{defn:controlledpair})
which will be  satisfied if $f$ is projective.
When $Y=\emptyset$, this data represents the motivic version of $R^if_*F(w)$
denoted here by $h_S^i(X)(w)$. The parameter $w$  keeps track of Tate
twists, which although extraneous for ordinary sheaves are nontrivial
in the Hodge and \'etale realizations.  For nonempty $Y$, the associated
motive $h_S^i(X,Y)(w)$ roughly corresponds to the fiberwise cohomology of the pair.
In essence, $\M(S;F)$ is set up as  the universal theory for which:

\begin{enumerate}
\item[(M1)] $\M(S;F)$ is an $F$-linear abelian category with  a faithful exact functor $R_B$ to
the category of sheaves of $F$-modules on $S$ with its classical topology.
\item[(M2)] A morphism  $X'\to X$ over $S$, taking $Y'$ to $Y$ would give rise to a morphism
of $h_S^i(X,Y)(w)\to h_S^i(X',Y')(w)$ compatible with the usual
pullback  map under $R_B$.
\item[(M3)] Whenever $Z\subseteq Y\subseteq X$, there are  connecting morphisms 
$h_S^i(X,Y)(w)\to h^{i+1}_S(Y,Z)(w)$ compatible with the usual
pullback maps.
\item[(M4)] $h_S^{i+2}(X\times \PP^1, X\times \{0\}\cup Y\times
  \PP^1)(w)\cong h_S^i(X,Y)(w-1)$.
\item[(M5)] Objects and morphisms of $\M(S;F)$ can be patched on 
 a  Zariski open cover.
\end{enumerate}
The actual construction is obtained by modifying the framework discussed
in the previous paragraph.  Given
stratification $\Sigma$ and a collection of base points on the strata,
let $\Delta(\Sigma)$ be the collection of quadruples  such that
the cohomology sheaf  is constructible with respect to $\Sigma$. We can make this
into a partial category by adding morphisms corresponding to items (M2), (M3)
and (M4).
The functor $H_\Sigma:\Delta(\Sigma)\to \Q\text{-mod}$ is defined by
sending $(X,Y,i,w)$ to the product $H^i(X_s,Y_s)$ at the  given set of
base points.
The category $\PM(S,\Sigma;F)$ of $\Sigma$-constructible premotivic sheaves
 is  constructed explicitly as the category of comodules over 
$End^\vee(H_\Sigma)$. Note that contrary to initial appearances, this is
not simply a product of $\M(k)$ over the base points because
$\Delta(\Sigma)$ does not decompose this way (see example~\ref{ex:Mnotproduct}). The trivial exception is
when $S$ is a finite set of points.
The category $\PM(S;F)$ of premotivic sheaves  is given as the direct limit of these categories
as  $\Sigma$ gets finer. This will satisfy (M1) to (M4). The category
$\M(S;F)$  is obtained from $\PM(S;F)$ by forcing (M5) by passing to
the  associated stack. This last step can be made explicit. In fact a weak form of (M5)
holds for $\PM$. So it is not quite  clear to me whether this axiom is
redundant, nevertheless it is included for completeness.

Here are the precise properties:

\begin{thm}
  To every  $k$-variety, there is  an $F$-linear abelian category
$\M(S;F)$ such that:
\begin{enumerate}
\item These are defined over the prime field $F_0$, i.e.
$\M(S,F) \cong  \M(S, F_0)\otimes_{F_0} F$.

\item There is an exact Betti realization functor 
$$R_B:\M(S;F)\to Cons(S_{an},F)$$ to the category of constructible sheaves of
 $F$-modules for the classical topology. 

\item There is an exact Hodge realization functor
$$R_H:\M(S;\Q)\to Cons\text{-}MHM(S)\subset D^bMHM(S)$$
to the heart of the classical $t$-structure of the derived category mixed
Hodge modules (see appendix C).

\item There is an exact \'etale realization functor
$$R_{et}:\M(S;F)\to Cons(S_{et},F)$$  to
the category of  constructible
sheaves of $F$-modules for the \'etale topology. (In this case, $F$ should be finite or $\Q_\ell$.)

\item When $f$ satisfies a suitable admissibility condition (of being controlled in the
  sense of  definition~\ref{defn:controlledpair}), there exist motives in $\M(S;F)$ corresponding to $R^if_*F(n)$ under realization.

\item There are inverse images compatible with realizations.

\item There are higher direct images for projective or constant maps compatible with realizations.

\end{enumerate}
\end{thm}

Many of the above items are formal consequences of the definitions,
but the last is not. The construction of direct images  is technically the most difficult
part of this paper. General arguments give the existence of an adjoint to inverse
image which ought to play the role of the direct image. Proving that
this has reasonable properties for projective maps  requires work,
which uses  a refinement of the method of \cite{arapuraL}.
This earlier paper was really the starting point for this entire project. This
ultimately hinges on Nori's insight that Beilinson's ``basic lemma''
can be used to construct cell decompositions which reduce
the homological complexities. In the relative setting, there are
few additional complications. For instance, these decompositions are
only obtained locally over the base, so patching issues of the sort
given in (M5) comes into play.
But  modulo these technicalities, the basic strategy  of using cell
decompositions does work.

The objects of $\M(S)$ play the role of constructible sheaves. 
Inside this, we have a subcategory of  ``local systems'' arising
from particularly nice families $(X\to S, Y,i,w)$. The precise
condition  is that $X$ can be completed to a smooth projective map so
that $Y$ together with the boundary is a divisor with relative normal crossings.
These enjoy the  following good properties.

\begin{thm}
 There is  an abelian full subcategory
$\M_{ls}(S;\Q)\subset \M(S,\Q)$ of  motivic local systems  such that:
  \begin{enumerate}
\item The images of $\M_{ls}(S;\Q)$ ( respectively
  $\M_{ls}(S;\Q_\ell)$) under  $R_B$ (respectively $R_{et}$) is
  contained in the category of locally constant (respectively lisse
  sheaves). The image under $R_H$ is contained in the category $VMHS(S_{an})$
of admissible variations of mixed Hodge structures.

   \item There are tensor products on $\M_{ls}(S;F)$ compatible with
 realizations. With this structure it is  a Tannakian category.
 
\item The subcategory $\M_{pure}(S,\Q)\subset \M_{ls}(S,\Q)$
  generated by smooth projective families is a semisimple Tannakian
  category.

\item Objects in   $\M_{ls}(S,\Q)$ carry a  weight filtration such
  that the associated graded objects are pure.
  \end{enumerate}
\end{thm}

A number of the arguments again rely on the existence of cellular
decompositions. Regarding item 4,  I do not have a good notion of weight in $\M(S)$ at
present. I expect that it would require the development of an
analogue of perverse sheaf in $D^b\M(S)$, since
the pure objects are almost surely of this form.

 A natural question, that is only partially
solved here, is the relation of this approach to motives to the others.
Andr\'e \cite{andre} defines the class of motivated
cycles on smooth projective variety to be the cycles which would be
expected to be algebraic assuming Grothendieck's standard conjectures.
Andr\'e showed that the category of pure motives over a
field constructed with such correspondences has all the expected
properties. This construction can be extended to
more general bases without much difficulty \cite{ad}. We show that
this category is precisely $M_{pure}(S)$. In his unpublished work,
Nori has constructed a functor from Voevodsky's category $D_{gm}(k)$
to $D^b\M(k)$. It seems reasonable to expect that this generalizes
over a base,
but such matters will be postponed for the future.
In the final section, I discuss  Nori's Hodge conjecture  which says
that $\M(\C)$ embeds fully faithfully into the category of mixed Hodge
structures.  This would imply that the canonical mixed Hodge structure
on cohomology is ``Galois invariant''.
The relative
case can be reduced to this by rather formal argument involving direct image and
restriction functors.
\bigskip

{\em Acknowledgement:}
I would like to thank Madhav Nori for giving me his permission to
include some of his beautiful constructions (of course, I take responsibility
for errors). Also thanks to M. de Cataldo, F. D\'eglise, M. Levine,
J. Lipman,  D. Patel, M. Saito and J. Sch\"urmann for some useful
conversations. This research was started at the Max Planck Institute in the
fall of 2007 and continued at the Tata Institute in the winter of 2008;  I am grateful
to these institutes as well as the NSF for their support.
\bigskip

{\em Notation:} Since the notation will tend to get rather heavy, {\em  I will routinely 
 suppress subscripts, superscripts and others
 symbols whenever they can be understood from context.}
Given a  ring $R$, let $R$-Mod (respectively $R$-mod)  stand for the
category of (finitely presented) left $R$-modules. 
  Fix a  field $k$ embeddable into $\C$ and another field $F$. For most of the paper, I
  will work with a fixed embedding $\iota:k\hookrightarrow \C$.
   A $k$-variety is simply a reduced  separated $k$-scheme of finite type. Let
$Var_k$ be the category of these. 
Given a $k$-variety $X$, the word point $x\in X$ generally refers to
a $\bar k$-rational point.   I will   denote the analytic
space $(X\times_\iota Spec\, \C)_{an}$ by $X_{\iota, an}$ or $X_{an}$ or 
sometimes just $X$, in keeping with the previous comment regarding
notation.
 A quasi-projective morphism is a morphism which
can be factored as a composition of an open immersion followed by a
projective morphism.
 I will usually write  $H^i(X;F)$ for $H^i(X_{\iota,an};F)$.
 Given a map $f:X\to S$ of spaces and a sheaf $\F$ on $X$,
I will often denote the higher direct image $R^if_*\F$ by
 $H^i_{S}(X,\F)$. Since this will {\em never} be used to denote
 cohomology with support in this paper, there should be no danger of
 confusion. 

\section{Representations of graphs}

\subsection{Endomorphism coalgebras}\label{section:Flin}

In the next couple of sections, we set up the basic foundation for the
rest of the paper. Let $F$ be a field.  Suppose we are given an
$F$-linear abelian category $\A$ with an exact faithful embedding $H$ into the
category of finite dimensional vector spaces $F\text{-mod}$. Then the
ring  $End(H)=End_F(H)$, of $F$-linear natural transformations of $H$
to itself, will act naturally on $H(A)$ for any $A\in Ob\A$.
This suggests that one might be able to reconstruct $\A$ as the category
of finite dimensional modules over this ring.  However, this does not generally
 work (example \ref{ex:gradedmodule}).
The right thing to consider is the category of 
comodules over the dual object $End^\vee(H)$ whose construction we
learned from \cite{js}.  Before getting into the construction, we
should explain how to characterize it. Given a commutative
$F$-algebra $R$, we can form new category $\A\otimes R$ with the same
objects as $\A$, but $Hom_{\A\otimes R}(-,-)=Hom_\A(-,-)\otimes R$.
The functor $H$ extends to an $R$-linear functor $H\otimes R:\A\otimes
R\to R\text{-mod}$. In this way, we have an algebra valued functor  $R\mapsto
End_R(H\otimes R)$. 

\begin{lemma}\label{lemma:CharofEndvee}
  This functor is represented by a colagebra $End^\vee(H)$, i.e.
$$Hom_F(End^\vee(H), R) \cong End_R(H\otimes R)$$
\end{lemma}
This implies that 
$End^\vee(H)^*=End(H)$, but usually $End(H)^*\not=End^\vee(H)$.
Nevertheless, most of the
statements become easier to follow if one  formulates them for $End(H)$
and dualizes. The lemma tells us how to make sense of this.
Note that we can express $End^\vee(H)$ or any
coalgebra as a directed union of finite dimensional subcoalgebras $\cup
E_i$. Thus the correct dual object to $End^\vee(H)$ is not $End(H)$
but the pro-algebra
$\text{``}\varprojlim\text{"} E_i^*$. Moreover, $\A$ can described as
$2$-colimit of the categories of $E_i$-modules.
We will find this viewpoint convenient later on, but  for the moment, it
seems simpler to work with the coalgebra.

 Given pair of  functors $G,H:C\to D$, with $D$ 
 $F$-linear, $Hom(G,H)$ is an $F$-vector space. More explicitly, we
 can identify $Hom(G,H)$ with
 \begin{equation}
   \label{eq:HomGH}
 \ker[\prod_{M\in ObC} Hom(G(M),H(M)) \longrightarrow \prod_{f:N\to
   P\in Mor C} Hom(G(N), H(P))]
\end{equation}
where the map takes  the collection $(\eta_M)_M$ to $(H(f)\circ \eta_N-\eta_P\circ G(f))_f$.
Composition makes $End(H)= Hom(H,H)$ into an algebra as noted above.
Following \cite{js}, it is convenient to introduce a  smaller ``predual'' object, which
means that $Hom^\vee(G,H)^*=Hom(G,H)$.

Let $F\text{-Lin}$ be the collection of $F$-linear abelian 
categories with finite dimensional $Hom$'s. Suppose that we now have a pair of functors $G,H:C\to D$,
with $D\in F\text{-Lin}$. Define
$Hom^\vee(G,H)$ to be the cokernel of
\begin{equation}
  \label{eq:Endv}
  \bigoplus_{f:N\to P\in C}  Hom(G(N),H(P))^*\stackrel{S}{\longrightarrow} \bigoplus_{M\in
    Ob C } Hom(G(M), H(M))^*
\end{equation} 
where the map $S$ is defined  so that $Hom^\vee(G,H)^*=Hom(G,H)$. More explicitly,
$S$ sends $\eta_f^*\in Hom(G(N),H(P))^*$ to $\eta^*_N\in Hom(G(N),H(N))^* $
plus $\eta^*_P\in Hom(G(P), H(P))^*$ where
$$\langle \eta_N^*,\eta_N\rangle = \langle\eta^*_f, H(f)\circ \eta_N\rangle$$
$$\langle \eta_P^*,\eta_P\rangle =-\langle \eta_f^*, \eta_P\circ
G(f)\rangle $$

Upon setting $End^\vee(H):= Hom^\vee(H,H)$, we see that this
satisfies lemma~\ref{lemma:CharofEndvee} as a vector space, and we can
use this formula to define the colagebra structure. However, it will
be useful to describe this more explicitly.
The sum of the maps
$$End(H(M))^*\to F,$$
dual to the identity, is easily seen to factor through $End^\vee(H)$
and this defines the  counit
$$ 1_H^\vee: End^\vee(H)\to F$$
Given functors $G,H,L:C\to D$ with $D\in F\text{-Lin}$ we have  a comultiplication
$$\circ^\vee:Hom^\vee(G,L)\to Hom^\vee(G,H)\otimes Hom^\vee(H,L)$$
dual to composition $\circ$.
More precisely, $\circ$ is given by product of compositons
$$c_M:Hom(G(M), H(M))\otimes Hom(H(M), L(M))\to Hom(G(M), L(M))$$
Then $\circ^\vee$ is given by the sum of the dual maps $c_M^*$ 
\begin{equation}
  \label{eq:cMvee}
 Hom(G(M), L(M))^*\to Hom(G(M), H(M))^*\otimes Hom(H(M), L(M))^*  
\end{equation}
Given $G,G':C\to D$ and $H,H':D\to E$ with $D,E\in F\text{-Lin}$, there is compositon
$$\diamond^\vee:Hom^\vee(G'\circ G, H'\circ H)\to Hom^\vee(G,H)\otimes Hom^\vee(G',H')$$
dual to the  operation $\diamond$ defined in appendix A.
The operation $\diamond$ is a product of maps
$$d_M: Hom(G(M),H(M))\times Hom(G'(H(M)), H'(H(M))) \to Hom(G'\circ G(M), H'\circ H(M))$$
and $\diamond^\vee =\sum d_M^*$.
To simplify arguments with these operations, we use the following duality principle:

\begin{lemma}\label{lemma:dualityPrinc}
Suppose we are given an identity   in $+,\circ,\diamond,1_G$,
which amounts to the commutivity of a finite diagram with arrows  labelled by these operations.
Then  the dual identity, obtained by reversing arrows and relabelling by
 $+,\circ^\vee,\diamond^\vee,1_G^\vee$, also holds.
\end{lemma}

\begin{proof}
  Suppose we have a finite diagram with vertices given as finite tensor
products of spaces $Hom^\vee(-,-)$, and edges labelled by $+,\circ^\vee,\diamond^\vee,1_G^\vee$.
Then commutivity
can be established by chasing elements. Given an element of
one of the vertices, we can find a subdiagram  of finite dimensional vector spaces which
contains it.  Duality for finite dimensional vector spaces implies that  the 
commutivity of the subdiagram would then follow from commutivity of the dual diagram.
\end{proof}
 
Using this  principle, we can see that:

\begin{lemma}\label{lemma:EndveeColag} Given composable functors $H$ and $G$
  \begin{enumerate}
  \item $End^\vee(G)$ is a colagebra over $F$ with respect to $\circ^\vee, 1^\vee$.
\item The map $p$ given by
$$
\xymatrix{
 End^\vee(H\circ G)\ar[r]^{p}\ar[d]^{\diamond^\vee} & End^\vee(G) \\ 
 End^\vee(H)\otimes End^\vee(G)\ar[ru]^{1^\vee\otimes 1} & 
}
$$
is a colagebra homomorphism.
  \end{enumerate}
   
\end{lemma}

\begin{proof}
 The first part is clear, since the dual statement is that $End(G)$ is an algebra. 
For the second, we have to establish
that $p$ preserves comultiplication. Dually, by  identities given in the appendix, 
  \begin{eqnarray*}
    1\diamond (\alpha \circ \beta) &=& (1\circ 1)\diamond (\alpha\circ \beta)\\
&=&(1\diamond \alpha)\circ  (1 \diamond  \beta)
  \end{eqnarray*}
\end{proof}

One can now readily verify lemma~\ref{lemma:CharofEndvee}.
We also leave the formulation and  proof of the corresponding statement
for $Hom^\vee$ to the reader.

%%%%

\subsection{Nori's construction}\label{sect:nori}

Any category can be regarded as a directed graph (or diagram in Nori's
terminology) by forgetting the composition law. This forgetful functor
admits a left adjoint: given a directed graph $\Delta$, we can form a
category $Paths(\Delta)$, whose objects are vertices of $\Delta$ and
morphisms are finite (possibly empty) connected paths between
vertices. The adjointness amounts to the obvious fact that given a
graph $\Delta$ and a category $C$, there is a one to one
correspondence between graph morphisms $\Delta\to C$ and functors
$Paths(\Delta)\to C$. In view of this, we may apply category theoretic
terminology and results to directed graphs.

Let $H:\Delta\to F$-mod be a functor, i.e. a quiver.
We can define $End^\vee(H)$  by the formula \eqref{eq:Endv}, which
simplifies to 
\begin{equation}
  \label{eq:Endv2}
  coker [\bigoplus_{f:N\to P\in Mor\Delta}  Hom(H(P),H(N))\stackrel{S}{\longrightarrow} \bigoplus_{M\in
    Ob \Delta} End(H(M))]
\end{equation} 
where $S$ takes $\eta_f\in Hom(H(P),H(N))$ to the difference of 
$\eta_f\circ H(f)\in End(H(N))$ and  $H(f)\circ \eta_f\in
End(H(P))$.

We note the following, which is easily checked.

\begin{lemma}\label{lemma:endvee}
  \begin{enumerate}
  \item[]
\item  The collection of functors from graphs to $F$-mod forms a category
where the morphisms are commutative diagrams
$$
\xymatrix{
 \Delta\ar[r]^{H}\ar[d]^{\pi} & F\text{-mod} \\ 
 \Delta'\ar[ru]^{H'} & 
}
$$
\item  $End^\vee(H)$ is isomorphic to $End^\vee(\tilde H)$, where $\tilde H$ is the extension of
   $H$ to $Paths(\Delta)$.

\item The assignment $(\Delta, H)\mapsto End^\vee(H)$ is functorial.
In particular,  there is an induced map $End^\vee(H)\to End^\vee(H')$
of colagebras where $H$ and $H'$ are as in 1.
 
  \item If $\Delta$ is a category then  $End^\vee(H) \cong End^\vee(H')$,
  where  $H'$ is the induced functor on the category
   $H(\Delta)$  with the same objects as $\Delta$ but morphisms given by its
  image under $H$.

\item The functor $(\Delta,H)\mapsto End^\vee(H)$ preserves finite
  coproducts. In more explicit terms, if $\Delta$ decomposes into a
  disjoint union of $\Delta_1\coprod \Delta_2$, then $End^\vee(H) =
  End^\vee(H|_{\Delta_1})\times End^\vee(H|_{\Delta_2})$ (which is the
  coproduct of coalgebras).
  \end{enumerate}
\end{lemma}

\begin{proof}
  The first statement is clear.  

For the second, we have that $End^\vee(H)$  and $End^\vee(\tilde H)$ are the quotients of $\bigoplus End(H(M))$
by 
$$I_H= S(\bigoplus_{f\in Mor \Delta}  Hom(H(P),H(N)))$$
and 
$$I_{\tilde H}= S(\bigoplus_{f\in Mor Paths(\Delta)}  Hom(H(P),H(N)))$$
respectively. Clearly $I_H\subseteq I_{\tilde H}$. So we have to check the reverse inclusion.
We first note that $S(1)=0$, so it suffices to check that $S(\eta_{f_1\ldots f_n})\in I_H$ for $n\ge 2$.
For $n=2$, this follows from the identity
$$S(\eta_{f_1f_2}) = S(\eta_{f_1f_2}\circ H(f_1)) + S(H(f_2)\circ \eta_{f_1f_2})\in I_H$$
The general case is similar.

Although the third statement is similar
to lemma~\ref{lemma:EndveeColag}, the previous formalism will not
apply without modification. So it is easier to prove directly.  An element of
$End^\vee(H)$ is represented by a finite sum $\sum h_M$ of elements $h_M\in End(H(M))^*$.
Define $\pi(h_M) = h_{\pi(M)}\in End(H'(M))^*$. To see that this is compatible with comultiplication $\circ^\vee$,
observe that $\circ^\vee(h_M) = c_M^*(h_M)$, where $c_M^*$ is given in~\eqref{eq:cMvee}. Then
$$\pi(\circ^\vee(\sum_M h_M)) = \pi(\sum c_M^*(h_M)) = \sum c_{\pi(M)}^*(h_{\pi(M)})
= \circ^\vee(\pi (\sum_M h_M))$$

The fourth and fifth statement follows immediately from the formulas
(\ref{eq:Endv}) and \eqref{eq:Endv2}.

\end{proof}

We let $End^\vee(H)\text{-comod}$ denote the category of right comodules
over this coalgebra in $F$-mod. 

\begin{cor}
  A morphism $(\Delta, H)\to (\Delta',H')$ as above induces a faithful exact functor
 $End^\vee(H)\text{-comod}\to End^\vee(H')\text{-comod}$.
\end{cor}

\begin{proof}
  This isn't so much a corollary as a statement of the fact that both categories can be viewed
as subcategories of $F\text{-mod}$.
\end{proof}

We can therefore view $End^\vee(H)\text{-comod}$ as a subcategory of
$End^\vee(H')\text{-comod}$.
We will often apply this, without comment, when $\Delta\subset \Delta'$ is a subgraph
and $H$ is the restriction of $H'$.

\begin{cor}
  If $H':\Delta\to F\text{-mod}$ is another functor with a natural
  isomorphism $\Gamma:H\to H'$, then $End^\vee(H)$-comod and
  $End^\vee(H')$-comod are isomorphic.
\end{cor}

\begin{cor}\label{cor:tildeDelta}
Let $\pi:\tilde \Delta\to \Delta$ be a morphism of graphs such that it
is surjective on objects and such that  every fiber is connected.
Then $End^\vee(H)\cong End^\vee(H\circ \pi)$.
\end{cor}

\begin{proof}
The assumption guarantees that $H(Paths(\Delta))$ and $H\circ \pi(Paths(\tilde \Delta))$
are equivalent.
\end{proof}

 Given $M\in Ob \Delta$, $H(M)$ is naturally
a left $End(H(M))$-module, and hence by transpose an $End(H(M))^*$-comodule.
Via the map $End(H(M))^*\to End^\vee(H)$, $M$ becomes a
$End^\vee(H)$-comodule, which we usually denote by $h(M)$. 
This is a functor $\Delta\to End(H)$-comod. The structure of a general comodule
is  clarified by the following.

\begin{lemma}\label{lemma:presentation}
Any object $V$ of $End^\vee(H)$-comod  fits into an exact sequence
$$\bigoplus_{i=1}^m h(M_i)\to \bigoplus_{j=1}^n h(N_j)\to V\to 0$$
for some $M_i, N_j\in Ob\Delta$.
\end{lemma}

\begin{proof}
The lemma will follow from the claim that 
 any comodule is the image of finite
sum of the form $\bigoplus h(M_i)$.
Set $E^\vee(D)=End^\vee(H|_D)$ for any subgraph. When $D$ is finite, $E^\vee(D)$
is quotient of a finite sum of comodules  of the form $End(H(M)) \cong
H(M)^{\dim H(M)}$, so the claim follows when $V= E^\vee(D)$. In general, the matrix coefficients
of the $E^\vee(\Delta)$ coaction on $V$ lie in some $E^\vee(D)$ with $D$ finite.
 Thus $V$ has a quotient of   a finite sum of copies of
 $E^\vee(D)$. So the claim holds in general.
\end{proof}

\begin{remark}\label{rmk:2limit}
  There is a dual description of $End^\vee(H)\text{-comod}$ which is
  closer to what Nori originally used \cite{ levine, nori-mot}. As in the
  previous argument, we  can express $E\text{-comod}=\cup E^\vee(D)$ as $D\subset \Delta$ runs
over finite subgraphs.
Therefore as explained in appendix A, we have equivalences
  \begin{eqnarray*}
 End^\vee(H)\text{-comod} &\sim& \2lim_{D}E^\vee(D)\text{-comod}\\
 &\sim& \2lim_{D}End(H|_D)\text{-mod}
 \end{eqnarray*}

\end{remark}

\begin{thm}
  If $U:\A\to F$-mod is an exact faithful $F$-linear functor
on an $F$-linear abelian category, then $\A$ is equivalent to
$End^\vee(U)$-comod.
\end{thm}

\begin{proof}
The proof given in  \cite[\S 7, thm 3]{js} for the complex field
works in general.
\end{proof}

It is instructive to observe that the corresponding statement for
$End(U)\text{-mod}$  will usually fail.

\begin{ex}\label{ex:gradedmodule}
Let $\A$ be the category of finite dimensional $\Z$-graded $\C$-vector
spaces. This can be identified with the category of comodules over $End^\vee(U)=\C[T,T^{-1}]$
in the usual way. However,  the category of $End(U)=\prod_\Z\C$ modules
is much bigger. For example, $\C$ with the $End(U)$-action arising
from  a nontrivial ultraproduct
$End(U)\to (\prod\C/\mathcal{U})\cong \C$ gives an $End(U)$-module which does not arise
from a graded module.
\end{ex}

We can use this theorem to deduce a version of  Nori's Tannakian theorem. (The  original statement, which is
stronger, can be found in \cite[thm 1]{brug}, \cite[\S 3.3]{levine} or \cite{nori-mot}.) 

\begin{cor}[Nori]\label{cor:tannaka0}
  Suppose that $\A$ is
  an $F$-linear abelian category equipped with a faithful exact
  functor $U:\A\to F\text{-mod}$. If $G:\Delta \to \A$ is a morphism of
  directed graphs such that $H$ is equivalent to $U\circ G$, then
  there is a functor
$\tilde G:End^\vee(H)\text{-comod}\to \A$  (called the extension of $G$)
  rendering the diagram
$$
\xymatrix{
  \Delta\ar[r]^{G}\ar[d]_{h}\ar[dr]_<<<<{H} &\A\ar[d]^{U}\\
  End^\vee(H)\text{-comod}\ar@{-->}[r]\ar@{-->}[ur]_>>>>>{\tilde G}\ar[r]_>>>>>U & F\text{-mod}\\
  }
$$
commutative up to natural equivalence.
\end{cor}

It is convenient to prove a slightly stronger statement,
where $F\text{-mod}$ is replaced by, for example, the category of finite
dimensional modules over an $F$-algebra.

\begin{cor}\label{cor:tannaka}
  Let $\cR$ be an $F$-linear abelian category  with a faithful exact functor
  $\rho:\cR\to F$-mod. Suppose that $H:\Delta\to F$-mod
   factors as $H_1\circ\rho$ (up to natural equivalence).
  Suppose that $\A$ is
  an $F$-linear abelian category equipped with a faithful exact
  functor $U:\A\to \cR$. If $G:\Delta \to \A$ is a morphism of
  directed graphs such that $H_1$ is equivalent to $U\circ G$, then there are functors
  $End^\vee(H)\text{-comod}\to \cR$, $\tilde G:End^\vee(H)\text{-comod}\to \A$ 
  rendering the diagram
$$
\xymatrix{
  \Delta\ar[r]^{G}\ar[d]_{h}\ar[dr]_<<<<{H_1} &\A\ar[d]^{U}\\
  End^\vee(H)\text{-comod}\ar@{-->}[r]\ar@{-->}[ur]_>>>>>{\tilde G}\ar[dr]_U &\cR\ar[d]^\rho\\
  & F\text{-mod}
  }
$$
commutative up to natural equivalence.

\end{cor}

\begin{proof}
We obtain a commutative diagram of
coalgebras
$$
\xymatrix{
  End^\vee(H)\ar[r]\ar[d] &End^\vee(U\circ \rho)\ar[ld]\\
  End^\vee(\rho) &}
$$
Thus we get a functor between the categories of finite dimensional
comodules
$$
\xymatrix{
  End^\vee(H)\text{-comod}\ar[r]^{\tilde G}\ar[d] &End^\vee(U\circ\rho)\text{-comod}\sim \A\ar[ld] \\
  End^\vee(\rho)\text{-comod}\sim \cR & }
$$
The equivalences of categories, indicated by $\sim$,  follow from the
theorem and the above assumptions.
\end{proof}

There is also a naturally statement, which we give only in the
situation of corollary \ref{cor:tannaka0}.

\begin{cor}\label{cor:tannaka-nat}
  Given a diagram
$$
\xymatrix{
 \Delta\ar[dd]^{\pi}\ar[r]^{G}\ar[rrd]^{H} & \A\ar[rd]^{U}\ar[dd] &  \\ 
  &  & F\text{-mod} \\ 
 \Delta'\ar[r]_{G'}\ar[rru]^>>>>{H'} & \A'\ar[ru]_{U'} & 
}
$$
which commutes up to natural isomorphism,
the  diagram
$$
\xymatrix{
 End^\vee(H)\text{-comod}\ar[r]^>>>>{\tilde G}\ar[d]^{\pi} & \A\ar[d] \\ 
 End^\vee(H')\text{-comod}\ar[r]^>>>>{\tilde G'} & \A'
}
$$
 commutes up to natural isomorphism.
\end{cor}

The case of particular interest to us in corollary \ref{cor:tannaka} is the category
$\cR=(F\text{-mod})^n$ of  finite dimensional  vector spaces admitting
gradings of the form $V= V_1\oplus V_2\oplus \ldots V_n$. 
This can  be identified with the category of finite modules over the
ring $F^n$.
A natural example arises as follows.
Given functors $H_i:\Delta_i\to F\text{-mod}$, we can define a new
functor $H_1\times H_2\times\ldots H_n$ on  the
cartesian product $\Delta_1\times\Delta_2\times \ldots \Delta_n$ in the
category of graphs, to $(F\text{-mod})^n$ by
$$H_1\times H_2\times 
\ldots H_n(M_1,\ldots M_n) = H_1(M_1)\oplus \ldots H_n(M_n)$$
We have an induced functor
\begin{equation}
  \label{eq:EndveeH1H2}
 End^\vee(H_1\times H_2\times\ldots H_n)\text{-comod}\to \prod_i (End^\vee(H_i)\text{-comod} )
\end{equation}
where to be clear, on the left we really mean $End^\vee(\rho\circ
(H_1\times \ldots H_n))$, 
where $\rho: (F\text{-mod})^n\to F\text{-mod}$ is the
forgetful functor.
We will need a criterion for when this is an equivalence. It usually isn't.

\begin{ex}
  Let $\Delta$ be a graph consisting of a single vertex $pt$ and no
  morphisms. Let $H(pt)=F$. Since this is a finite graph, we can work
  with endomorphism rings rather than coalgebras. One has $End(H)= F$ and $End(H\times H)=M_2(F)$ 
  the ring of $2\times 2$ matrices. Therefore by Morita's theorem, $End(H\times
  H)\text-{mod}\sim F\text{-mod}\nsim (F\text{-mod})^2$. The natural
  map $End(H\times H)\text-{mod}\to (F\text{-mod})^2$ is the diagonal
  embedding $ F\text{-mod}\to (F\text{-mod})^2$.
\end{ex}

\begin{lemma}\label{lemma:productcomod}
Suppose that each object in $\Delta_i$ has maps to an object $\emptyset_i$
satisfying $H_i(\emptyset_i)=0$. Then \eqref{eq:EndveeH1H2} is an equivalence.
\end{lemma}

\begin{proof}
It is enough to check this for $n=2$ graphs. By taking limits, we can
reduce to the case where $\Delta_i$ are both finite. From equation \eqref{eq:HomGH},
the ring $End(H_1\times H_2)$ consists of families
$$(f_{P_1,P_2})\in \prod End(H_1(P_1)\oplus H_2(P_2))$$  
compatible with composition along morphisms. % of the product graph.
Choose maps $\tau_i:P_i\to \emptyset_i$.
By considering compatibility along the morphisms
$1\times \tau_2:(P_1,P_2)\to (P_1,\emptyset_2)$ and $\tau_1\times
1:(P_1,P_2)\to (\emptyset_1, P_2)$,
we see that $f_{P_1,P_2}$  must be of the form
$\begin{pmatrix} f_{P_1,\emptyset_2} & 0\\ 0 & f_{\emptyset_1,P_2} \end{pmatrix}$.
Thus
$$End(H_1\times H_2) = End(H_1)\times End(H_2)$$
\end{proof}

\subsection{Enriched model}\label{sect:enriched}

Let $H:\Delta\to F\text{-mod}$ be a functor on a graph as above.
Although functors on $End^\vee(H)\text{-comod}$ can be constructed with the help of corollary \ref{cor:tannaka},
it is sometimes difficult to apply.
It will be convenient to give an alternative description (up to equivalence)
of $End^\vee(H)\text{-comod}$ which allows us to incorporate  extra
structure.  Fix a finite dimensional  commutative $F$-algebra with an algebra
homomorphism $p:R\to F$ such that $F$ is flat over $R$. These rather
strong assumptions are valid in the case of principal interest  to us,
where
$R=F\times F$ with $p$ projection onto the first factor.
 Suppose that $H^\#:\Delta\to
R\text{-mod}$ is a functor. We can define the $R$-coalgebra
$End_R^\vee(H^\#)$ by replacing  $Hom$ by $Hom_R$ in
\eqref{eq:Endv2}. This is not the same as the  coalgebra $End^\vee(H^\#\circ
\rho)$ considered earlier.
Let $End_R^\vee(H^\#)$-comod denote the $R$-linear category of comodules in $R\text{-mod}$.
Then we wish to describe the relationship between $End^\vee(H^\#)\text{-comod}$ and
$End^\vee(H)\text{-comod}$.

First, we make a brief digression.
Given an $F$-linear abelian category $C$, an ideal $I$ is a collection of subspaces $I(c_1,c_2)\subseteq Hom(c_1,c_2)$ such that 
$$Hom(c_2,c_3)\circ I(c_1,c_2)\subseteq I(c_1,c_3),$$
$$  I(c_1,c_2)\circ Hom(c_0,c_1)\subseteq I(c_0,c_2)$$
Given an ideal $I$, $C/I$ is the category with the same objects as $C$ and 
$$Hom_{C/I}(c_1,c_2) = Hom_C(c_1,c_2)/I(c_1,c_2)$$
For example, if $G:C\to D$ is an exact functor, $\ker G = \{f\in Mor C\mid G(f)=0\}$ is an ideal.
Note, however, that the quotient $C/\ker G$ should not be confused with the quotient of $C$ by the thick
subcategory generated by $\{c \in ObC\mid G(c)=0\}$.

\begin{lemma}\label{lemma:surjcat}
  If $G:C\to D$ is an exact functor between essentially small
  abelian categories such that
  \begin{enumerate}
  \item[(a)] $G$ is essentially surjective.
  \item[(b)] $G$ is surjective on Homs.
  \end{enumerate}
Then  $D$ is equivalent to $C/\ker G$.
\end{lemma}

\begin{proof}
$G$ induces an equivalence   $C/\ker G\sim D$
since $Hom_{C}(c,c')/\ker G\cong Hom_{D}(G(c),G(c'))$
\end{proof}

Returning to the set up describe earlier. 
We have an isomorphism $p:End^\vee(H^\#)\otimes_R F\cong End^\vee(H)$
and hence an exact functor
$$p:End^\vee(H^\#)\text{-comod}\to End^\vee(H)\text{-comod}$$
 given by
$M\mapsto M\otimes_R F$. 
The conditions of the above lemma are easily verified in general.
 In the case, where $R=F^2$, this is  almost immediate.
$End^\vee(H^\#)$-comodule decomposes into a sum of two factors
corresponding to the idempotents of $R$, and $p$ is projection on the
first factor.
Thus:

\begin{cor}\label{cor:surjcat}
  $End^\vee(H)\text{-comod}$ is equivalent to
  $End^\vee(H^\#)\text{-comod}/\ker p$.
\end{cor}

\begin{cor}
  An $F$-linear functor on $End_R^\vee(H^\#)\text{-comod}$ such that $\ker(p)$ maps
to zero, induces a functor on $End^\vee(H)\text{-comod}$.
\end{cor}
\subsection{Products}
We need to incorporate tensor products into
our story.    The category of functors from graphs to $F$-mod  forms a category
with tensor product given as follows. Let $H:\Delta\to F$-mod and
$H':\Delta'\to F$-mod  be two such functors. Then
$H\otimes H':\Delta\times \Delta'\to F$-mod is given by
$(M,N)\mapsto H(M)\otimes H'(N)$.  The one point graph $\{*\}$ with $*\mapsto F$ gives
the unit making this into a tensor category, where for our purposes a tensor category over $F$
is an $F$-linear additive  category with a bilinear
symmetric monoidal structure. We have
$$End^\vee(H\otimes H')\cong End^\vee(H)\otimes End^\vee(H')$$
\cite[\S 8, prop 1]{js}. This yields a product
$$End^\vee(H)\text{-comod}\times End^\vee(H')\text{-comod}\to
End^\vee(H\otimes H')\text{-comod}$$
When $H=H'$ is equipped with a symmetric associative pairing $H\otimes
H\to H$ and a unit $*\in Ob\Delta, H(*)=F$. Then $End(H)$ becomes a commutative bialgebra.
Thus $End(H)$-comod becomes a tensor category with a tensor
preserving functor to $F$-mod  given by  the forgetful functor.
With minor modifications to the proof of corollary \ref{cor:tannaka0}, we have

\begin{cor}\label{cor:tannaka2}
Suppose that  $H$ has a product as  above. If in the hypothesis
of corollary \ref{cor:tannaka}, $\cR=F\text{-mod}$, $\A$ is an $F$-linear abelian
tensor category, and the functors $\rho, U, G$ are product
preserving. Then $\tilde G$ is also product preserving.
\end{cor}

Recall \cite[sect 2]{deligne} \cite[chap IV, sect 1]{levine-book}
that a dual of an object $M$ in a tensor category, with unit ${\mathbf 1}$,
is an object $M^\vee$ equipped with
morphisms  $\delta:{\mathbf 1}\to M^\vee\otimes M$ and $\epsilon:M\otimes M^\vee\to {\mathbf 1}$
such that the compositions
\begin{equation}\label{eq:dual1}
M \stackrel{id\otimes \delta}{\longrightarrow} M\otimes M^\vee\otimes M \stackrel{ \epsilon\otimes id}{\longrightarrow}M\tag{D1}
\end{equation}
\begin{equation}\label{eq:dual2}
M^\vee \stackrel{ \delta\otimes id}{\longrightarrow} M^\vee\otimes M\otimes M^\vee
 \stackrel{  id\otimes \epsilon}{\longrightarrow}M^\vee
 \tag{D2}
 \end{equation}
 yield the identities.  Alternatively, $M^\vee$ is characterized by the natural  isomorphisms
 \begin{equation}\label{eq:dual3}
Hom(X\otimes M, Y) \cong Hom(X, M^\vee\otimes Y)
\tag{D3}
\end{equation}
 \begin{equation}\label{eq:dual4}
Hom(X\otimes M^\vee, Y) \cong Hom(X,M\otimes Y)
\tag{D4}
\end{equation}
 
  In particular, the dual  is unique up to isomorphism if it exists. A map $f:M\to N$
 yields a dual or transpose map $f^\vee: N^\vee\to M^\vee$ if $M,N$ both possess duals.
  
 \begin{lemma}\label{lemma:duals}
Given an exact sequence 
$$M_1\to M_2\to M_3\to 0$$
if $M_i^\vee$ exists for $i=1,2$ then $M_3^\vee$ exists.
\end{lemma}

\begin{proof}
Set $M_3^\vee= \ker (M_2^\vee\to M_1^\vee)$. Condition \eqref{eq:dual3} is a consequence of
the diagram 
$$
\xymatrix{
 0\ar[r] & Hom(X\otimes M_3,Y)\ar[r]\ar@{-->}[d]^{\cong} & Hom(X\otimes M_2,Y)\ar[r]\ar[d]^{\cong} & Hom(X\otimes M_1,Y)\ar[d]^{\cong} \\ 
 0\ar[r] & Hom(X, M_3^\vee\otimes Y)\ar[r] & Hom(X,M_2^\vee\otimes Y)\ar[r] & Hom(X,M_1^\vee\otimes Y)
}
$$
and \eqref{eq:dual4} is similar.
\end{proof}

 A neutral Tannakian category over $F$ is an abelian tensor category over $F$, with a faithful
 exact tensor preserving functor to $F$-mod, such
 that every object possesses a dual. Such a category can be realized as the category of comodules
 over a commutative Hopf algebra.

 \begin{prop}\label{prop:tannaka3}
Suppose that $H:\Delta\to F$-mod is equipped with a symmetric associative product as above. Assume that for every object $M\in Ob \Delta$, $h(M)$ has a dual.
Then $End^\vee(H)$-comod is neutral Tannakian.
\end{prop}
 
 \begin{proof}
 The proposition follows from lemmas 
\ref{lemma:presentation} and \ref{lemma:duals}

\end{proof}

\section{Premotivic sheaves}

\subsection{Constructible sheaves}

We recall:

\begin{defn}
  If $X$ is a complex variety (defined over $k\subset \C$), a sheaf
  $\F$ on $X_{an}$ is called
  constructible (or $k$-constructible) if it has finite dimensional
  stalks and there exists a partition $\Sigma$ of $X$ into Zariski
  locally closed (defined over $k$) so that $\F|_\sigma$ are locally
  constant for each $\sigma\in \Sigma$. In this case, $\F$ is also
  called $\Sigma$-constructible. Let $Cons(X)$ or $Cons(X,\Sigma)$
  denote the category of these.
\end{defn}

The definition of constructibility for sheaves on the \'etale topology
$X_{et}$ is similar \cite[chap V]{milne}, \cite[exp IX]{sga4}.
Basic examples of constructible sheaves include the direct images $R^if_*F$ and more generally
direct images of constructible sheaves \cite[cor. 2.4.2]{verdier}.
We  give a slight refinement below  (theorem~\ref{thm:constrBC}).

\begin{defn}
Given a morphism $f:X\to S$ and  a sheaf $\F$ on $X_{an}$ or $X_{et}$, we  say that
$H_S^i(X,\F)=R^if_*\F$ commutes with base change if for any quasi-projective
morphism   $g:T\to S$
 the canonical base change map
$$g^*H^i_S(\F)\to H^i_T(X\times_S T, g^*\F)$$
is an isomorphism. 
\end{defn}

\begin{defn}
Given a morphism $f:X\to S$ and  a sheaf $\F$ on $X$,
 if $H_S^i(X,\F)$ commutes with base change for all
$i$, we will say that $\F$ has the base change property (with respect
to $f$).
\end{defn}

The condition implies that
$$H_S^i(X,\F)_s \to  H^i(X_s, \F|_{X_s})$$
is an isomorphism for every $s\in S$.
Open immersions $X\to S$ give examples where this will fail for $s\in
S-X$. We review some (known) criteria for this to hold.
 A morphism $f:X\to S$ will be called  {\em  locally trivial} if it is
 a topological (although not necessarily an analytic)
 locally trivial fibre bundle  with respect to the  the analytic
 topology. More generally:

 \begin{defn}
   Let say that the pair $(f:X\to S,\F\in Cons(X))$ is  locally trivial if
there exists an open cover $\{U_i\}$ of $S$ and a stratified space $\Phi$
with a constructible sheaf $\G$, such that there are stratified homeomorphisms
 $f^{-1}U_i\cong \Phi\times  U_i$ compatible with $f$ such that $\G$
 pulls back to the restriction of $\F$.
 \end{defn}

 \begin{thm}\label{thm:properbase}
\-
   \begin{enumerate}
\item Given a short exact sequence
$$0\to \F_1\to \F_2\to \F_3\to 0$$
if any two of the $\F_i$  have the  base change  property, then so does
the third.
\item   If $(f:X\to S,\F)$ is locally trivial, then $\F$ has the base
  change property. 
   \item(Proper base change)  If $f:X\to S$ is proper, then
    any sheaf has the base change property.
\item (Locally trivial base change) If $T\to S$ is locally trivial, then the
  base map for $f:X\to S$ with respect to $T$ is an isomorphism for any $\F$ and $i$.
   \end{enumerate}
 \end{thm}

 \begin{proof}
The first statement is obvious.
   The second is clear once we observe that it can be reduced to the case of a product
   $S\times \Phi\to S$, with $\F$ pulled back  from $\Phi$.  For the third,
when $f:X\to S$ is proper  and $T$ is a point,
   the base change property  follows from \cite[thm 6.2]{iversen}. Therefore 
$$g^*H^i_S(\F)\to H^i_T(X\times_S T, g^*\F)$$
is an isomorphism on stalks. 

For the fourth statement, we can reduce to
the case of product, and then apply the K\"unneth formula.
 \end{proof}

We can combine these criteria into one convenient notion.

\begin{defn}\label{defn:controlled}
Given a quasi-projective morphism $f:X\to S$ and a sheaf $\F\in
Cons(X)$, we will say that the pair $(f,\F)$ is  controlled, or that
$f$ is controlled with respect to $\F$,
if there exists a commutative diagram
$$
\xymatrix{
& & S\ar[d]^{=}\\
 X\ar[r]_{g}\ar[d]^{j}\ar[urr]^{f} & T\ar[r]_{h} &  S \\ 
 \bar X\ar[ru]_{\bar g} &  & 
}
$$
such that  %$S\to \bar S$ is a compactification, 
$h$ and $\bar g$ are projective, $j$ is an open immersion, and 
such that $(\bar X, j_!\F)$ is locally trivial over $T$.  
\end{defn}

 It is worth observing that the condition is automatic if $S$ is point
 because everything is locally trivial over a point. Also
note that in general the conditions imply that 
$(\bar  X,\bar X-X)$ is a relative fibre bundle over $T$.
 Such a diagram, which
need not be unique, will be  called a control diagram for the pair
$(f,\F)$.

\begin{lemma}\label{lemma:control}
  If $(f:X\to S, \F)$ is controlled, then $\F$
has   the  base change property with respect to $f$.
\end{lemma}

\begin{proof}
  Choose a control diagram as above. Let $q: S'\to S$ be a morphism,
  and consider the diagram
$$
\xymatrix{
 X\ar[r]^{g} & T\ar[r]^{h} &  S \\ 
 X'\ar[r]^{g'}\ar[u]^{q_2} & T'\ar[r]^{h'}\ar[u]^{q_1} &  S'\ar[u]^{q}
}
$$
where the squares are Cartesian. 
We have to  prove that $q^*R^i(h\circ g)_*\F\cong R^i(h\circ g)_*\F$. It
is enough to check isomorphism on the $E_2$ terms of Leray spectral
sequence. We have
$$q^*R^ah_*R^bg_*\F\cong R^ah_*' q_{1}^* R^bg_*\F$$
because $h$ is proper, and we have 
$$ R^ah'_* q_{1}^* R^bg_*\F\cong R^ah_*R^bg'_*q_2^*\F$$
because $g$ is locally trivial with respect to $\F$.
\end{proof}

\begin{prop}\label{prop:controlH}
  Suppose that $(f:X\to S,\F)$ is a controlled pair with a locally
  closed  $S$-embedding $X\to \PP^N\times S$. 
Given  a nonempty Zariski open set $P\subset
  \check{\PP}^N$, there exists a
  Zariski open cover $\{S_\alpha\}$ of $S$
and elements $H_\alpha\in P$, such that
 $$(H_\alpha\cap f^{-1}S_\alpha \to S_\alpha,\F|_{ H_\alpha})$$ 
$$(f^{-1}S_\alpha\to S_\alpha, q_{\alpha!}q_\alpha^*\F) $$
are controlled, where $q_\alpha:f^{-1}S_\alpha-H_\alpha\to f^{-1}S_\alpha$ is the inclusion.
\end{prop}

\begin{proof}
  Choose a control diagram
$$
\xymatrix{
 X\ar[r]_{g}\ar[d]^{j} & T\ar[r]_{h} &  S \\ 
 \bar X\ar[ru]_{\bar g} &  & 
}
$$
for $\F$.
Then by assumption, $j_!\F$ is construcible with respect to   a
stratification $\{\bar X_\dt\}$ of $\bar X$ which is locally trivial over
$S$. Given $s\in S$, we may choose $H\in P$  transverse to $\bar X_\dt\cap
\bar X_s$. It remains transverse to  $\bar X_\dt\cap \bar X_t$, for
$t$ in a neighbourhood $S_s$ of $s$.
 It follows that the stratification generated by
$\bar X_\dt$ and $H$ and $\bar X_\dt\cap H$ are locally trivial over
$S_s$.
\end{proof}

In order to proceed, we will need Whitney stratifications.
For our purposes a stratification of a variety $X$ is a finite
partition $\Sigma$ of $X$  into Zariski
locally closed sets such that the closure of any stratum $\sigma\in
\Sigma$ is a union  of strata. When $X$ is complex, we will say that this is Whitney
if any stratum is smooth and if the Whitney conditions hold for any
$x\in \sigma\subset \bar \tau$
\cite{lipman1, mather, teissier, verdier-w}. This
 means that given sequences
 $x_i\in \sigma$ and $y_i\in \tau$, both converging to
 $x$, the limit of the secants $\overline{x_iy_i} $ (with respect to a
 local embedding into $\C^N$) lies in the limit of tangent spaces
 $T_{\tau,y_i}$ when the limits exist.  These conditions may appear strange
  at first glance, but their importance comes from the fact
 that they imply  local  triviality of the topology of $X$ along each stratum $\sigma$. In more
 precise terms, there exists a
tubular  neighourhood  $\sigma\subset T_\sigma\subset X^{an}$
\cite{mather} with a retraction $\pi:T_\sigma\to \sigma$ which makes
it a locally trivial fibre bundle  with a contractible fibre.

A number of authors have
 observed that the Whitney conditions can be reformulated in more algebraic language; a
 simple description can be found in  the proof of \cite[thm 3.2]{lipman1}
 for instance. So in particular, given $k$-variety $X$, a stratification
 which is Whitney for one embedding $k\subset \C$ will be Whitney for
 all. Concerning existence in this generality, we have

 \begin{lemma}\label{lemma:kwhitney}
   If $X$ is a $k$-variety with a filtration $X=Y^0\supset
   Y^1\supset\ldots Y^n=\emptyset $ by closed sets, there exists a Whitney
   stratification defined over $k$ such that each $Y^i$ is a union of strata.
 \end{lemma}

 \begin{proof}[Sketch]
  This is a modification of Teissier's method for constructing canonical
  Whitney stratifications  \cite{teissier}. For simplicity, assume that
  $X$ is irreducible. Define $X^0= X$,  $X^1=X_{sing}\cup Y^1$ and
  inductively set
$$X^{i+1}=\{x\in X^i\mid\exists j<i, \text{ the
  Whitney conditions fail at } x\in X^i\subset X^j\}\cup Y^{i+1}$$
This gives a chain of closed sets decreasing to $\emptyset$ by \cite[p 477]{teissier}. The
partition $\{X^i-X^{i+1}\}$ can be seen to give a Whitney
stratification by arguing as in   \cite[pp 478-480]{teissier}.
 \end{proof}

The following gives a version of Deligne's generic base
change theorem \cite[thm 1.9, p 240]{deligne-sga} for complex varieties.

\begin{thm}\label{thm:constrBC}
Given a morphism   $f:X\to Y$ defined over $k$ and a $k$-constructible
sheaf $\F$ on $X$,  the sheaves $R^if_*\F$ are $k$-constructible.
There exists a dense Zariski open  $U\subset S$ such that the
restriction $\F$ has the base change property with respect to
$f^{-1}U\to U$.
\end{thm}

\begin{proof}
Let $j:X\hookrightarrow\bar X$ be an open immersion such that there is a
proper  map $\bar f:\bar X\to Y $ extending $f$.
Let $\Sigma$ be a Whitney stratification of $Y$ with connected strata, and let $\Lambda$ be
a Whitney stratification of $X$ refining $f^{-1}\Lambda$ such that
$j_!\F$  is $\Sigma$-constructible.  By lemma~\ref{lemma:kwhitney}, we
may assume that $\Sigma$ and $\Lambda$ are defined over $k$.
We may  also assume that $\bar X-X$ is a
union of strata, and that $\bar f$ is a  submersion on each stratum.
Each $\sigma\in \Sigma$ possesses a
tubular  neighourhood  $\sigma\subset T_\sigma\subset Y$
with a retraction $\pi:T_\sigma\to \sigma$ which makes
it a locally trivial fibre bundle  with a
contractible fibre $G$.  
The preimage $f^{-1}T_\sigma$ inherits a stratification from $X$, such
that $f^{-1}\sigma$ is a union of strata.
Thom's isotopy theorem \cite{mather, verdier-w} implies
that  $f^{-1}T_\sigma\to T_\sigma$ is a stratfied fibre bundle. That
is, there exists an open cover $\{V_i\}$ of $T_\sigma$ and stratfied
space $\Phi$ such that there are homeomorphisms $f^{-1}V_i\cong \Phi\times
V_i$ of stratified spaces compatible with projection. One can see that
there is no loss in generality in assuming that  each
$V_i=\pi^{-1}U_i$ for an open subset $U_i\subset \sigma$. %CHECK
We may assume that the $U_i$ are contractible.
It follows that $(f^{-1} T_\sigma, f^{-1}\sigma)$ is a (relative)
fibre bundle over $\sigma$ with fibre say $(\Phi\times G,\Phi')$.
We can see that $\Phi$ carries a constructible sheaf $\G$ which pulls back to
the restriction of $\F$ under the
 homeomorphisms $f^{-1}U_i\cong \Phi\times G\times U_i$.
This implies that $R^if_*\F$ is locally constant along $\sigma$,
and hence $k$-constructible.

Applying the above argument to a  Zariski dense stratum $\sigma$, shows
that $(X\to Y,\F)$ is a locally trivial over  $\sigma$. Therefore
the base change property holds over $\sigma$.
\end{proof}

\subsection{Cohomology of pairs}

Let $S$ be a $k$-variety.
Let $Var^2_S$ be the category whose
objects are pairs $(X\to S,Y)$ with $Y\subseteq X$ closed. A morphism from
$(X\to S,Y)\to (X'\to S,Y')$ is a morphism of $S$-schemes $X\to X'$ such that
$f(Y)\subseteq Y'$. For such an object and a sheaf $\F$ on $X_{an}$, set
$$H_S^i(X,Y;\F) = R^if_*j_{X,Y!}\F|_{X-Y}$$
where $f:X\to S$ is the projection and $j_{X,Y}:X-Y\to X$ is the
inclusion. We revert to writing this as $H^i_S(X,\F)$ when $Y$ is empty.
When $\F=F$ is constant, $H_S^i(X,Y;F) $
is $k$-constructible by the  theorem \ref{thm:constrBC}, and we can
describe this as the sheaf associated to
$$
U\mapsto  H^i(f^{-1}U, f^{-1}U\cap Y; F)
$$

The map
$(X,Y)\mapsto H_S^i(X,Y;F)$ is easily seen to give a contravariant
functor on $Var^2_S$. The morphisms 
$H^i_S(X,Y)\to H_S^i(X',Y')$ are induced by the homomorphisms
$$H^i({f'}^{-1}U, {f'}^{-1}U\cap Y';F)\to H^i(f^{-1}U, f^{-1}U\cap
Y;F)$$

\begin{defn}\label{defn:controlledpair}
  A pair $(f:X\to S,Y)$
in $Var_S^2$  is controlled with respect to $\F$ if $f$ is controlled with respect to the
sheaf  $j_{X,Y!}\F|_{X-Y}$. The pair is said to be controlled it if it
so with respect to the constant sheaf $F$.
\end{defn}
 
 The control condition for a pair, with respect to $F$, amounts to requiring that
both $(\bar  X,\bar X-X)$  and $(\bar Y, \bar Y-Y)$ are  relative fibre bundles over an intermediate
projective family $T\to S$, where $\bar Y$ is the closure of $Y$.

\begin{lemma}
  If $f:(X\to S, Y)$ is controlled with respect to $\F$, then $j_{X,Y!}\F|_{X-Y}$
has  the base change property with respect to $f$.
\end{lemma}

\begin{proof}
  This is a consequence of lemma \ref{lemma:control}.
\end{proof}

Therefore if $(X\to S,Y)$ is controlled then
\begin{equation}\label{eq:basechange}
H_S^i(X,Y;F)_s \cong H^i(X_s, Y_s; F)
\end{equation}
for every $s\in S$.

From proposition \ref{prop:controlH}, we obtain.

\begin{lemma}
  Suppose that $(X\to S, Y)$ is controlled. Then  for a general
  hyperplane $H$  (with respect to a locally closed embedding $X\subset
  \PP^N\times S$), $(H\to S, H\cap Y)$ is controlled.
\end{lemma}

Given a chain of closed sets $X\supset Y\supset Z$ and a sheaf $\F$ on
$X$, we get an exact
sequence
$$0\to j_{X,Y!}\F|_{X-Y}\to j_{X,Z!}\F|_{X-Z}\to
i_*j_{Y,Z!}\F|_{Y-Z}\to 0$$
where $i:Y\to X$ is the inclusion.
This induces a long exact sequence
$$\ldots  H^i_S(X,Y;\F)\to H^i_S(X,Z;\F)\to H^i_S(Y,Z;\F)\to H^{i+1}_S(X,Y;\F)\ldots $$
which reduces to the usual exact sequence for pairs, when $S$ is point
and $\F$ is constant.

\subsection{Premotivic sheaves}\label{section:motivicsheaves}

Let $S$ be a $k$-variety.
The category $\PM(S)$ of premotivic sheaves is constructed as a direct limit of categories $\PM(S,\Sigma)$.
Each  $\PM(S,\Sigma)$ is obtained by applying Nori's  construction 
 to an appropriate graph $\Delta(S,\Sigma)$ and functor $H_\Sigma$ given below.

Let $S\in ObVar_k$ be {\em connected}. Then we construct a graph $\Delta(S)$ as
follows. The objects (i.e. vertices) are quadruples $(X\to S, Y, i, w)$
consisting of
\begin{enumerate}
\item[(1)] a  quasi-projective morphism $X\to S$.
\item[(2)] a closed subvariety $Y$ such that the pair $(X\to S,Y)$ is
  controlled (definition \ref{defn:controlledpair}),
\item[(3)] a natural number $i\in \N$ and an integer $w$.
\end{enumerate}

The set of morphisms (edges) is the union of the three following sets:
\begin{enumerate}
\item[Type I:] Geometric morphisms 
$$(X\to S,Y,i, w)\to (X'\to S,Y',i,w)$$ 
where $(X\to S,Y)\to  (X'\to S,Y')$ is a
  morphism in $Var^2_S$.
\item[Type II:] Connecting morphisms 
$$(f:X\to S,Y,i+1,w)\to (f|_Y:Y\to S,Z,i,w)$$ 
for every chain $Z\subseteq Y\subseteq
  X$ of closed sets.
  \item[Type III:] Twisted projection morphisms 
  $$(X\times \PP^1, Y\times \PP^1\cup X\times \{0\}, i+2,w+1)\to (X, Y, i,w)$$
  for every $(X,Y,i,w)\in Ob\Gamma(S)$.
\end{enumerate}
\bigskip

For arbitrary $S\in Var_k$, set $\Delta(S)=\prod \Delta(S_i)$, where
$S_i$ are the connected components. Thus the parameters $i$ and $w$
are locally constant. 

By a good stratification or simply stratification of $S$, we mean a
finite partition $\Sigma$ into connected locally closed sets (defined
over $k$) such that
$\Sigma$ contains the closure of every element.
Given a stratification $\Sigma$, let $\Delta(S,\Sigma)\subset
\Delta(S)$ be the full subgraph consisting of objects such that
$H_S^i(X,Y; F) $ is constructible with respect to the stratification
$\Sigma$. %for all $i$.  
We have that $\bigcup \Delta(S,\Sigma)=\Delta(S)$ because the sheaves $H_S^i(X,Y; F) $ are
constructible.

Given $\Sigma$ as above,
 let $s = (s_{\sigma}\in \sigma(\bar k))$ denote a collection of base points,
one for each $\sigma\in\Sigma$.  Let $|\Sigma|$ be the cardinality of
$\Sigma$.  Define
$$H_{\Sigma, s, F}(X,Y,i,w) =
\prod_{\sigma\in \Sigma } [H_S^i(X,Y;F)]_{s_\sigma}=\prod_{\sigma\in \Sigma } H^i(X_{s_\sigma},Y_{s_\sigma};F)$$ 
to be the product
of stalks. We usually suppress the symbols $\Sigma,s,F$.

We want to extend $H=H_{\Sigma, s,F}$ to  a functor $\Delta(S,\Sigma)^{op}\to
F$-mod. We do this case by case.
\begin{enumerate}
\item[Type I:] 
A morphism $g:(f:X\to S,Y,i, w)\to (f':X'\to S,Y',i,w)$ of type I gives rise to the natural
homorphism 
$$H^i({f'}^{-1}U, {f'}^{-1}U\cap Y';F)\to H^i(f^{-1}U, f^{-1}U\cap Y;F)$$
for each $U\subset S$. Since this is clearly a morphism of presheaves,
 it induces a morphism of sheaves $H_S^i(X',Y')\to H_S^i(X,Y)$. Thus we get the desired
 map  $H(X',Y',i,w)\to H(X,Y,i,w)$ by taking the product of this sheaf map over stalks.
 We give a second description 
which is a bit more complicated, although better for comparing to the \'etale case.
We have a commutative diagram
$$
\xymatrix{
  j_{X'Y'!}F_{X'-Y'}\ar[rd]\ar@{-->}[dd] &  &  F_{Y'}\ar[ll]^{[1]}\ar[dd] \\ 
  &  F_{X'}\ar[ru]\ar[dd] &  \\ 
 \R g_*j_{XY!}F_{X-Y}\ar[rd] &  &  \R g_* g^* F_{Y'}\ar[ll]^{[1]} \\ 
  &  \R g_* g^* F_{X'}\ar[ru] & 
}
$$
where the triangles are  distinguished, and the solid vertical arrows
are the adjunction homomorphisms. Thus we get the dotted arrow above.
From which we obtain
$$\R f'_*j_{X'Y'!}F \to  \R f'_*\R g_*j_{XY!}F \cong \R f_*j_{XY!}F$$
So we get a map of sheaves
$$R^i f'_*j_{X'Y'!}F \to   R^i f_*j_{XY!}F$$
which is easily seen to coincide with the previous map.
\item[Type II:] 
A morphism $(X,Y,i+1,w)\to (Y,Z,i,w)$
of type II gives rise to a connecting homomorphism  $H^i_S(Y,Z)\to H^{i+1}_S(X,Y)$
induced from the exact sequence
$$0\to j_{X,Y_!}F\to j_{X,Z!}F\to j_{Y,Z!}F\to 0$$
Taking a product over stalks yields
$H(Y,Z,i,w)\to H(X,Y,i+1,w)$ .

\item[Type III:]

 Finally a morphism $(X\times \PP^1, Y\times \PP^1\cup X\times \{0\}, i+2,w+1)\to (X, Y, i,w)$
 corresponds to the isomorphism
$$H_S^i(X,Y;F)\to H_S^{i+2}(X\times \PP^1,Y\times \PP^1\cup X\times \{0\}; F)$$
given by exterior product with the fundamental cycle of $(\PP^1,\{0\})$.
This gives rise to 
$$H(X,Y, i,w)\to H(X\times \PP^1,Y\times \PP^1\cup X\times \{0\}, i+2, w+1)$$
\end{enumerate}

Thus we can apply the construction from the previous section to obtain:

\begin{defn}
  The category $\PM(S, \Sigma,s;F)$ of $\Sigma$-constructible premotivic
  sheaves of $F$-modules on $S$ is the category of finite dimensional right comodules
  over $End^\vee(H_{\Sigma,s})$.  % 
For any finite commutative $F$-algebra $R$,
  let $\PM(S,\Sigma,s; F)\otimes_F R$ denote the category with finitely
  generated right comodules over $End^\vee(H)\otimes_F R$.
\end{defn}

\subsection{Realizations}
By definition there is a faithful exact forgetful functor $U:\PM(S,\Sigma;F)\to F$-mod.
We can see immediately from the universal coefficient theorem that
$\PM(S, R) = \PM(S,F)\otimes_F R$, whenever $F\subseteq R$ is a field
extension.
The matrix coefficients of the $End^\vee(H)$-coaction of any object
$V$ of $\PM(S,\Sigma)$ lie in some $End^\vee(H|_D)$ for a finite
subgraph $D$. Thus $V$ can be regarded as an
$End^\vee(H|_D)$-comodule, or equivalently an $End(H|_D)$-module.  In
fact, we can describe $\PM(S,\Sigma,s)$ as the direct limit of the categories of finite
dimensional $End(H|_D)$-modules, as $D\subset \Delta(S)$ varies over
finite subgraphs (cf \cite{brug}). This dual description was employed
by Nori in his work, and it would appear that $\PM(Spec\, k,
Spec\, \C)$ is just Nori's category of cohomological motives tensored with $F$. 
 We write this as $\PM(k;F)$ or simply $\PM(k)$ from now on. 
 
 Given $M=(X\to S,Y,i,w)\in Ob \Delta(S)$, $H(M)$ is naturally an
$End(H(M))$-module, and hence by transpose an $End^\vee(H)$-comodule
 denoted by $\ph_{S}^i(X,Y)(w)$ or $\ph_S^i(X,Y)$ if $w=0$ (we will see shortly that
this independent of $\Sigma$ and $s$ in a suitable sense). When
$S=Spec\, k$, we omit the subscript.

By definition we have

\begin{prop}
  $\PM(S,\Sigma,s;F)$ is an $F$-linear abelian category with an
  exact faithful functor to $F$-mod.
\end{prop}

In view of the following, we may suppress base points.  

\begin{lemma}
  Suppose that $t_{\sigma}$ is another collection of base points, then
  $\PM(S,\Sigma,s)$ and $\PM(S,\Sigma, t)$ are isomorphic.
\end{lemma}

\begin{proof}
  Given a homotopy class of paths $\gamma_{\sigma}$ in $\sigma$
  joining $s_{\sigma}$ to $t_{\sigma}$, parallel transport along these
  curves yields an isomorphism of fiber functors $H_s\cong H_t$.
\end{proof}

\begin{remark}\label{rmk:PM}
This business of choosing base points and then suppressing them
is a bit clumsy. A more elegant approach is to simply redefine
$$H_{\Sigma,  F}(X,Y,i,w) =
\prod_{\sigma\in \Sigma } \Gamma(\tilde \sigma, \pi_\sigma^*H_S^i(X,Y;F))$$ 
where $\pi_\sigma: \tilde \sigma\to \sigma(\C)$ are the universal covers,
and then build $\PM(S,\Sigma; F)$ accordingly. However, the original
approach does make certain things more transparent, and will generally be preferred.
\end{remark}

We have the following consequences of corollary \ref{cor:tannaka}.

\begin{constr}
  Let $Cons(S_{\iota,an}, \Sigma;F)$ denote the category of sheaves of
  $F$-modules which are constructible with respect to the
  stratification $\Sigma$.  The  fibre  functor $\Phi:Cons(S_{an}, \Sigma)\to
  F$-mod given the product of stalks at the  base points provides a
  faithful exact  functor.  The discussion from the  previous section shows that $(X,Y,i,w)\mapsto H_S^i(X,Y;F)$
  is a functor on $\Delta(S,\Sigma)^{op}$ and that $H$ is a composition of this with $\Phi$.
  Thus corollary \ref{cor:tannaka} yields an extension functor
  $R_{\iota,B}=R_B:\PM(S,\Sigma)\to Cons(S_{\iota,an}, \Sigma)$ that we call Betti realization.
  $R_B$ coincides with the forgetful functor $U$ on $\PM(k)$. 
\end{constr}

\begin{constr}
The map
$$t^n(X,Y,i,w) =\ph^i_S(X,Y)(w+n)$$
extends to a functor $\Delta(S,\Sigma)^{op}\to \PM(S,\Sigma)$ satisfying
$t^nt^m = t^{n+m}$. When composed
with the forgetful functor to $F$-mod, we obtain  $H$. Thus this extends to
an endofunctor $T^n:\PM(S,\Sigma)\to \PM(S,\Sigma)$ satisfying 
$$T^n ( \ph^i_S(X,Y)(w)) = \ph^i_S(X,Y)(w+n)$$
and $T^nT^m=T^{n+m}$;
in particular, it is an automorphism.
\end{constr}

\begin{constr}
Let $F$ be finite or $\Q_\ell$. Let $\bar k\subseteq \C$ denote the  algebraic closure of $k$. 
  Consider the map
  $$(f:X\to S,Y,i, w)\mapsto R^i \bar f^{et}_*j^{et}_{  \bar X, \bar Y!}F_{ \bar X- \bar Y}(w)$$
   where the sheaves and operations are on the \'etale topology, $\bar
   f :\bar X\to S$ etc. are the base changes to $\bar k$,
   $j_{\bar X\bar Y}:\bar X-\bar Y\to \bar X $  is the inclusion, and
   $(w)$ represents the Tate  twist.
  This is easily seen to be a functor by modifying the above discussion.
  Thanks to the comparison theorem between \'etale and classical cohomology (appendix B).
$$(R^i \bar f^{et}_*j^{et}_{  \bar X, \bar Y!}F_{ \bar X- \bar Y}(w))_s \cong (R^i \bar f_*j_{  X, Y!}F_{ X- Y})_s\text{ in $F$-mod}$$
for $s\in S(\C)$. Thus we can obtain an extension which is the \'etale
  realization functor  $R_{et}$ from $\PM(S;F)$ to the
  category  $Cons(S_{et}, \Sigma,;F)$ of sheaves of $F$-modules on $ S_{et}$ constructible with respect
  to $\Sigma$.
\end{constr}

\begin{constr}
   Let $Cons\text{-}MHM(S_{\iota,an}, \Sigma;\Q)$ denote the heart
   of the classical $t$-structure on the category $MHM(S,\Sigma)$ of $\Sigma$-constructible
  mixed Hodge modules (appendix C). We have an embedding 
$$rat:Cons\text{-}MHM(S_{an}, \Sigma;\Q)\hookrightarrow Cons(S_{an}, \Sigma)$$
which can be composed with the above functor $\Phi$ to obtain a fibre
functor. 
Consider the map
  $$(f:X\to S,Y,i, w)\mapsto {}^cH^i\circ \bar \R f_*j_{  \bar X, \bar Y!}F_{ \bar X- \bar Y}(w)$$
   where the operations are in the derived categories of mixed Hodge
   modules, and $ {}^cH^i={}^c\tau_{\le i}{}^c\tau_{\ge i}$ is cohomology with respect to the classical
   $t$-structure. When composed with $rat$, we obtain $H_S^i(X,Y)$. 
Thus we obtain a Hodge realization functor 
$$R_{\iota,H}=R_H:\PM(S,\Sigma)\to Cons\text{-}MHM(S_{\iota,an}, \Sigma)$$
A special given in section \ref{section:localsystems} can be made more explicit.
\end{constr}

We fixed an embedding  $\iota:k\hookrightarrow \C$ at the outset. Let write  $\PM(S;F)^\iota$
for the resulting category. We now show that the category $\PM(S;F)$ is independent of this.

\begin{prop}
  For any two embeddings of  $\iota,\mu: k\subset \C$, the categories
$\PM(S,F)^\iota$ and $\PM(S,F)^\mu$ are equivalent.
\end{prop}

\begin{proof} 
It suffices to show that $\PM(S,\Sigma, F)^\iota$ and $\PM(S,\Sigma, F)^\mu$ are equivalent
for every stratificaion.
 We note that $\PM(S,\Sigma,;F) = \PM(S, \Sigma;F_0)\otimes_{F_0} F$ for any 
subfield. Thus 
it suffices to assume that $F$ is the prime field. Suppose that  $F=\Z/p\Z$.  Then, we can see this immediately by the  comparison
theorem \cite[exp XVI, thm 4.1]{sga4}
\begin{eqnarray*}
  H_{\Sigma, s}(X,Y,i,w;F) &=&
\prod_{\sigma\in \Sigma } [R^if_*j_!F]_{s_\sigma}\\
&\cong& \prod_{\sigma\in \Sigma } [R^i\bar f^{et}_*j^{et}_!F]_{s_\sigma}\\
 &=&H^{et}_{\Sigma, s}(X,Y,i,w;F)
\end{eqnarray*}
This description is independent of the embedding.  Therefore 
$\PM(S,\Sigma, F)^\iota$ and $\PM(S,\Sigma, F)^\mu$ are equivalent.

The remaining case $F=\Q$ follows Nori's argument \cite{nori-mot}
quite closely. Write $H^\iota$ and $H^\mu$ for the
functors corresponding to the embeddings.
Define the category $\T$  of triples $(A,B,h)$ $A,B\in Ob
F\text{-mod}$, $h:A\otimes \Q_\ell\cong B\otimes \Q_\ell$, where
morphisms are compatible pairs of linear maps.  If  $p$ denote the
first projection $(A,B,h)\mapsto A$, then $p$
is easily seen to be fully faithful and essentially surjective. Therefore it is an equivalence.
So there is functor $q:\Q\text{-mod}\to \T$  and natural isomorphisms $\gamma:q\circ p\cong 1_\T$
and $\eta: p\circ q\cong 1_{\Q\text{-mod}}$.
We get a functor $H^\T:\Delta(S,\Sigma)\to \T$ by taking  $H^\iota$
and $H^\mu$  as the first and second component.
 For the third, we use the composition of the comparison maps  
$$h:H^\iota\otimes\Q_\ell\cong H^{et}\cong H^\mu\otimes \Q_\ell$$ 
The map $p$ induces a homomorphism $End^\vee(H^\iota)\to End^\vee(H^\T)$.
We claim that this gives an isomorphism. Here we use the duality principle given in lemma~\ref{lemma:dualityPrinc}.
The  dual of $p$ is given by $p^*(f)= 1_{p}\diamond f$. The map  is injective as it has a left inverse given by
$$g\mapsto  (\gamma\diamond 1_H)\circ (1_q\diamond g)\circ (\gamma^{-1}\diamond 1_H)$$
The map $p^*$ is also surjective, because
$$p^* (1_q\diamond [(\eta\diamond 1)\circ g\circ (\eta^{-1}\diamond 1)])=g$$
By an identical argument $End^\vee(H^\mu) \cong End^\vee(H^\T)$

\end{proof}

\subsection{Base change}

\begin{lemma}
  Let $S=S_1\cup S_2\cup \ldots S_n$ be a decomposition into connected
  components. Choose stratifications $\Sigma_i$ of $S_i$ and let
  $\Sigma=\bigcup \Sigma_i$. Then 
$$\PM(S,\Sigma)= \PM(S_1,\Sigma_1)\times \PM(S_2,\Sigma_2)\times\ldots \PM(S_n,\Sigma_n)$$
\end{lemma}

\begin{proof}
  This is an immediate consequence of lemma~\ref{lemma:productcomod},
  which applies because
   the empty families $(\emptyset, \emptyset ,i,w)\in
   \Delta(S_i,\Sigma_i)^{op}$ map to $0$ under $H$.
\end{proof}

A morphism of (pointed) stratified varieties is a morphism of varieties such
that a nonempty preimage of any stratum is a union of strata (and base
points go to base points). If the
underlying morphism of varieties is the identity, we say that the
first stratification refines the second.  

\begin{constr}\label{constr:basech}
Suppose that
$f:(T,\Lambda,t)\to (S,\Sigma,s)$ is a morphism of pointed stratified
varieties. 
First, suppose (*) that the map on base points is surjective.
Applying corollary \ref{cor:tannaka} to 
$$(X\to S,Y,i,w)\mapsto \ph_T^i(X\times_S T, Y\times_S T)(w)$$
 yields an extension, which is
  the base change functor
  $f^*:\PM(S,\Sigma,s)\to \PM(T,\Lambda,t)$.
If $f(t)\not= s$, set $T' = T\coprod s'$
where $s' = s-f(t)$. Then the map 
$$f':(T', \Lambda'=\Lambda\cup s',  t'=t\cup s')\to (S,\Sigma, s),$$
which is $f$ on $T$ and identity on $s'$, satisfies
(*). We now define $f^*$ as the composite
$$\PM(S,\Sigma,s)\stackrel{f'^*}{\to} \PM(T',
\Lambda',t')=\PM(T,\Lambda,t)\times \PM(s',s',s')\stackrel{p}{\to} 
\PM(T,\Lambda,t)$$
where $p$ is the projection.
\end{constr}

We can always extend a morphism of stratified varieties to a
morphism of pointed stratified varieties, and this way obtain
 base change functor   $f^*:\PM(S,\Sigma)\to
 \PM(T,\Lambda)$. Alternately, we could define this directly in the
 spirit of remark~\ref{rmk:PM} without recourse to base points.
We mention two important special cases of this construction:
\begin{enumerate}
\item The construction applies when $S=T$ and $\Lambda$ refines $\Sigma$. This
  leads to faithful exact
   embeddings $\rho_{\Sigma,\Lambda}:\PM(S,\Sigma)\to \PM(S,\Lambda)$. 
  \item When $T=s$ is the of set of, say $n$, base points, we get an embedding
  $\PM(S,\Sigma)\to \PM(k)^n$.
\end{enumerate}

The last map is usually not an equivalence.

\begin{ex}\label{ex:Mnotproduct}
  Let $S=\mathbb{A}^1$ with  $\Sigma=\{0,\mathbb{A}^1-\{0\}\}$ and
  base points $s=\{0,1\}$.  Now consider the motive $\Q_S$ represented by
$(id:S\to S,\emptyset, 0,0)$. By passing to sheaves under $R_B$, we
see that this is not isomorphic to $\Q_0\oplus \Q_1$, so the base
point map $\PM(S)\to \PM(k)\times \PM(k)$ cannot be an equivalence.
\end{ex}

When $f:T\to S$ is an inclusion, we often denote $f^*M$ by $M|_T$.

\begin{lemma}\label{lemma:pullback}
  The assignment  $(S,\Sigma)\mapsto \PM(S,\Sigma)$, $f\mapsto f^*$ yields a contravariant pseudofunctor from the
category of stratified varieties to the $2$-category of abelian categories. 
This commutes with $R_B$.
\end{lemma}

\begin{proof}
The functor $id_S^*:\PM(S,\Sigma)\to \PM(S,\Sigma)$ is the extension
of $(X\to S,Y,i,w)\mapsto \ph_S^i(X, Y)(w)$. So clearly we have a
 natural isomorphism $id_S^*\cong id_{\PM(S)}$. 
Suppose that $f:(T,\Lambda)\to (S,\Sigma)$ and $g:(V,\Theta)\to
(T,\Lambda)$ are given, and assume the maps surject on base points.
 Then $(f\circ g)^*$ is the extension of 
$$(X\to S,Y,i,w)\mapsto \ph_V^i(X\times_S V, Y\times_S V)(w)\cong g^*\ph_T^i(X\times_S T, Y\times_S
T)(w)$$
So we obtain  a natural isomorphism $(f\circ g)^*\cong g^*\circ
f^*$. These natural transformations can be seen to have the necessary compatibilities (see
appendix A) to define a pseudofunctor.

The last statement follows from the isomorphism
$$R_B(\ph_T^i(X\times_S T, Y\times_S T)(w) )\cong H_T^i(X\times_S T, Y\times_S
T)$$ 
and corollary \ref{cor:tannaka-nat}.
\end{proof}

\begin{defn}
  The category of motivic sheaves of $F$-modules is given by the
  $2$-colimit 
$$\PM(S; F) = \text{2-}\varinjlim_{\Sigma} \PM(S,\Sigma;F)$$
(see appendix A).
\end{defn}

It follows from the discussion in appendix A that $\PM(S;F)$ is abelian, and
the natural maps $\PM(S,\Sigma;F)\to \PM(S;F)$ are exact.

Note that $\ph^i_S(X,Y)\in \PM(S,\Sigma)$ maps to the 
 same symbol $\ph^i_S(X,Y) $ under refinement. We denote the common
  value in the colimit by $\ph^i_S(X,Y)$ as well.
  Observe that $R_B$ provides a faithful exact embedding of $\PM(S,\Sigma)$ into
the category $Sh(S)$ of sheaves of $F$-vector spaces on $S$.  This is
compatible with refinement. So it passes to the limit. Therefore in more concrete terms, we can 
identify $\PM(S)$ (up to equivalence) with the directed union of subcategories
$$\bigcup_\Sigma R_B(\PM(S,\Sigma))\subset Sh(S)$$
Note that this lies in the subcategory $Cons(S)$ of constructible
sheaves.

As a corollary to lemma \ref{lemma:pullback}.

\begin{cor}
  $S\mapsto \PM(S)$ is a contravariant  pseudofunctor.
\end{cor}

There are a number of useful variations of this construction. Let
$\Delta_{c}(S,\Sigma)\subset \Delta(S,\Sigma)$ denote the full
subgraph consisting of tuples $(X\to S,Y,i, w)$ where $X\to S$ is projective. 
Then set $\PM_c(S,\Sigma)$  be the category of comodules over $End^\vee(H_{\Sigma}|_{\Delta_c(S,\Sigma)})$.
The category of compact motives
$$\PM_c(S) = 2\text{-}\varinjlim_{\Sigma} \PM_c(S,\Sigma;F)$$
This can be regarded as an abelian subcategory of $\PM(S)$. It contains
motives $h^i_S(X)$ of  projective families.

\section{Motivic Sheaves}

\subsection {Zariski Descent}\label{sect:extmotivicsheaves}

 In the next section, we will see that motivic sheaves in $\PM_c(S)$ can  be patched on a Zariski
open cover.  This is not yet evident for $\PM(S)$, and may be an
indication of a defect in the definition.
 Since this issue plays a relatively minor role (so far), we just sketch a
construction of a new category of  motivic sheaves
$\M(S)$  which removes this defect.

We can repackage $\M(S)$ in the language of  fibred categories \cite{giraud, vistoli}.
Given a functor $\pi:\F\to \calC$, the fibre $\pi^{-1}(A)$ over $A\in Ob(S)$ is the category
with objects $\pi^{-1}A$ and morphism $\pi^{-1} id_A$. Fibres need not behave as expected,
for example fibres over isomorphic objects
need not be equivalent, unless further conditions are imposed.
An arrow
$\phi\in Mor\F$ is {\em cartesian} if for any commutative diagram consisting of solid arrows
$$
\xymatrix{
 A\ar[rrd]^{\psi}\ar@{-->}[rd] &  &  \\ 
 \pi(A)\ar[rrd]^{\pi(\psi)}\ar[rd] & B\ar[r]_{\phi} & C \\ 
  & \pi(B)\ar[r]_{\pi(\phi)} & \pi(C)
}
$$
the dotted arrow can be filled in uniquely. 
The functor $\pi:\F\to \calC$ is  {\em fibred} if
any arrow of $\calC$ can be lifted to a  cartesian arrow of $\F$ with a
specified target. It is sometimes convenient
to fix a collection of specific cartesian lifts. Such a collection is called a cleavage.
Given a cleavage, there is a well defined way to define a pullback functor $f^*:\pi^{-1}(B)\to \pi^{-1}(A)$
for any $f:A\to B\in Mor\calC$. These form a pseudofunctor. Conversely, any pseudofunctor determines
a fibred category.
Define $\M$ to be the category whose objects are pairs $(S\in Var_k, M\in \M(S))$,
and  morphisms are pairs $(f:T\to S,  M\to f^*N\in Mor \M(T) )$. This is fibred over $Var_k$ via the natural projection $\pi:\M\to Var_k$.  The  categories $\M(S)$ are just the fibres $\pi^{-1}(S)$, and the original
pseudofunctor is determined by the cleavage $\{(f, f^*N=f^*N)\}$.

\begin{defn}
 Let $\pi: \M\to Var_k$ denote that
stack associated to $\PM\to Var_k$ for the Zariski site  \cite[chap 2 \S2]{giraud}.
The category of  motivic sheaves $\M(S)$  over $S\in Ob Var_k$  is the  fibre over $S$.
\end{defn}

Unravelling all of this, we see that:
\begin{enumerate}
\item $\M\to Var_k$ is a fibred category. Fix a cleavage for it, so that
 there are functors $f^*:\M(S)\to   \M(T)$ for each $f:T\to S$.
If $f$ is an inclusion, we denote $f^*M$ by $M|_S$.
\item For any $M,N\in  \M(S)$, $U\mapsto Hom_{\M(S)}(M|_U, N|_U)$ is a
  sheaf on the Zariski  topology,
i.e. $\ \M$ is a prestack.
\item Given an open cover $\{U_i\}$ of $S$, $M_i\in \PM(U_i)$ with isomorphisms $g_{ji}: M_i|_{U_{ij}}\cong M_j|_{U_{ij}} $
satisfying the usual cocycle condition $g_{ii}=id$,
$g_{ik}=g_{ij}g_{jk}$, there exists $M\in \M(S)$ such that $M_i\cong M|_{U_i}$.
\item There is a functor $\iota:\PM\to  \M$ of fibred categories. This means that it fits into
a commutative diagram 
$$
\xymatrix{
 \PM\ar[rr]^{\iota}\ar[rd] &  & \M\ar[ld] \\ 
  & Var_k & 
}
$$ 
and takes cartesian arrows to cartesian arrows. The functor $\iota$ is univeral among all such functors
from $\PM$ to stacks. So that
any functor $\M\to \M'$ to a fibred category satistfying conditions  (2) and (3) must factor through $ \M$ in an essentially unique way, or more precisely, there is an equivalence between the 
categories $Hom_{fibcat}(\M,\M')$ and $Hom_{fibcat}(\M,\M')$.
\end{enumerate}

A concrete construction of the associated stack is given in \cite[chap 3]{lm} (although stated for
 groupoids, it works in general). It is a two step process. First,  one constructs a prestack
$\PM^+$. The objects of $\PM^+(S)$ are the same as for $\PM(S)$. For $Hom_{\PM^+(S)}(M,N)$
we take the global sections of the sheaf associated to 
$U\mapsto Hom_{\PM(S)}(M|_U, N|_U)$. This forces condition (2) to hold.
Now  construct $ \M(S)$ as the category whose objects
consists of descent data, i.e., collections $(M_i\in \PM^+(U_i), g_{ij})$ as in (3).
Given two objects  $(M_i\in \PM^+(U_i), g_{ij}),(M_i'\in \PM^+(U_i'), g'_{ij})$,
after passing to a common refinement we can assume that $U_i=U'_i$.
A morphism is then given by a compatible collection of morphisms $M_i\to M_i'$.
We have obvious functors $\PM(S)\to \PM^+(S)\to \M(S)$. It is easy to see from
this description that $\M(S)$ is abelian, and $\PM(S)\to \M(S)$ is exact.
To summarize:

\begin{thm}
 $ \M$ is a stack of abelian categories  over $Var_k$. There is an exact functor
$\PM\to  \M$ which is universal among all  functors from $\PM$ to stacks.
\end{thm}

As corollary, the realization functors $R_B$ et cetera extend to $\M$.

We define  $h_S^i(X,Y)(w)$ as the image of $\ph_{S}^i(X,Y)(w)$ in $\M(S)$.
Since the collection of sheaves $Sh(S)$ forms a stack over $S_{zar}$, we can see that $R_B$
factors through $\M$. With the help of  the above explicit description, we get a
slightly sharper result.

\begin{cor}
  The functor $R_B:\M(S)\to Sh(S)$ extends to an exact faithful functor $R_B:\M(S)\to Sh(S)$.
In particular, $R_B(h_S^i(X,Y)(w)) = H^i_S(X,Y)$.
\end{cor}

By a very similar argument, we have:

\begin{cor}
  The functor $R_{et}:\M(S,F)\to Sh(S_{et},F)$ extends to an exact faithful functor $R_{et}:\M(S,F)\to Sh(S_{et},F)$.
\end{cor}

We want give an alternative concrete construction. 
It will be convenient to start with some generalities. It will be convenient to ``sheafify'' the discussion of section
\ref{sect:nori}.
 Let $\D=\{D_U\}$ be a collection of graphs indexed  Zariski
open subsets of $S$, such that for each  inclusion $\iota:V\subset U$ we have a restriction
morphism $\iota^{-1}:D_V\to  D_U$ satisfying $(\iota\circ\mu)^{-1}=
\mu^{-1}\circ \iota^{-1}$. Let $H:D_U\to F\text{-mod}$ be a collection
of morphisms compatible with restriction. We refer $(\D,H)$ as a
compatible collection. Given such a collection, let $\E^\vee (H)$ to be the sheaf associated to the  presheaf
$$U\mapsto End^\vee(H|_{D_U})$$
 Let $\E^\vee(H)\text{-comod}$ denote the
category of Zariski sheaves of finite dimensional $F$-vector spaces  with
a coaction by $\E^\vee (H)$.  When $\D$ consists of  finite graphs, this is isomorphic
to the category of modules over the sheaf of rings 
$$\E(H) = \mathcal{H}om(\E^\vee(H),F_S)$$
Let $Coh(\E(H))$ denote the category of coherent i.e. locally
finitely presented, modules. The colimit 
$$Coh(\E^\vee(H)) =\2lim_{\D'\subset \D\text{ finite}} Coh(\E(H|_{\D'}))$$
can be realized as a subcategory of $\E^\vee(H)\text{-comod}$.

Suppose that we are given compatible collections
$\tilde \D$ and $\D$ and  morphisms $\pi:\tilde D_U\to D_u$ compatible
with restriction. Then a functor $H$ on $\D$ induces a functor
$$Coh(\E^\vee(H)) \to Coh(\E^\vee(H\circ \pi))$$
An analogue of corollary \ref{cor:tildeDelta} is:

\begin{lemma}\label{lemma:tildeD}
This functor is an equivalence if for any two objects of the fiber
$\pi:\tilde D_U\to D_u$  are connected by a chain of morphisms on some
cover of $U$.
\end{lemma}

Given a functor $H^\#:\Delta(S,\Sigma)\to F\text{-mod}^2$ compatible
with restriction and such that $H= p_1\circ H^\#$, we can define a
category $\M^\#(S,\Sigma)$ by mimicing the  procedure used to define
$\M'$. We have refinement of corollary \ref{cor:surjcat} and which follows
by the same argument.

\begin{lemma}\label{lemma:surjcatstack}
  $\M^\#(S,\Sigma)/\ker p_1 \sim \M(S,\Sigma)$.
\end{lemma}

Returning to the initial set up of a stratified variety 
We now give the explicit construction. Fix a stratified variety
$(S,\Sigma)$ with a given set of base points $s$ for $\Sigma$. Given Zariski open sets $V\subset
U\subset S$, let $s'$
denote the intersection of $s$ with $U-V$.  As in construction
\ref{constr:basech}, we have a map given
by  composition
$$End^\vee(H|_{\Delta(U,\Sigma)})\to
End^\vee(H|_{\Delta(V,\Sigma)})\times End^\vee(H|_{\Delta(s')})\to End^\vee(H|_{\Delta(V,\Sigma)})$$
where the last map is projection.
This makes
$$U\mapsto End^\vee(H|_{\Delta(U,\Sigma)})$$
into a presheaf of coalgebras.
Let $\E^\vee (S,\Sigma)$ denote the associated sheaf on the Zariski
topology $S_{zar}$. 

Then $\D= \{\Delta(U,\Sigma)\}$ gives a compatible collection with
restrictions given by
$$(X\to V,Y,i,w)\mapsto (X\times_V U, Y\times_V U,i,w)$$
Set
$Coh(\E^\vee(S,\Sigma)) =Coh(\E^\vee(H))$ as above.
A premotivic sheaf  $M$ in $\PM(S,\Sigma)$ determines an object of
$U\mapsto M|_U$ of $\E^\vee(S,\Sigma)\text{-comod}$, which can be
seen to lie in some  $Coh(\E(H|_{\D'}))$, with $\D'\subset \D$ finite, and therefore in
$Coh(\E^\vee(S,\Sigma))$. This gives a fully faithful exact functor
$\PM(S,\Sigma)\to Coh(\E^\vee(S,\Sigma))$.
%--

\begin{lemma}
  $Coh(\E^\vee(S,\Sigma))\text{-comod}$ coincides with the full subcategory of
comodules $M$ which are locally isomorphic to objects of $\PM(-,\Sigma)$
\end{lemma}

\begin{proof}
  An object of  $Coh(\E^\vee(S,\Sigma))\text{-comod}$ lies in some
  $Coh(\E(\D))$. So it is locally of the form $coker(\E(\D)^n\to
  \E(\D)^m)$ which lies in $\PM(S,\Sigma)$.
\end{proof}

 Set $\M'(S) = \2lim_\Sigma Coh(\E^\vee(S,\Sigma))$. Then
 we have   a functor $\PM(S)\to \M'(S)$ induced from the one above.

\begin{lemma}
  $\M(S)$ is equivalent to $\M'(S)$.
\end{lemma}
\begin{proof}

The functor $\PM(S)\to \M'(S)$ induces a functor $\PM^+(S)\to \M'(S)$ which is fully faithful.
This extends to a functor $\M(S)\to \M'(S)$ which is again fully faithful.
It also essentially surjective by the previous  lemma.
\end{proof}

In view of this result, we will generally denote $\M'(\ldots)$ simply by
$\M(\ldots)$ from henceforth. We also define $\M(S,\Sigma) =
Coh(\E^\vee(S,\Sigma))\text{-comod}$.

\subsection{Extension by zero}

Let $j:S\hookrightarrow \bar S$ be an open immersion with boundary
$\partial \bar S= \bar S-S$. Suppose that $\Sigma$ is a stratification of $\bar S$ such that $S$ is union of strata.
Let $\Delta(\bar S,\partial \bar S,\Sigma)\subset \Delta(\bar
S,\Sigma)$ denote the full subgraph consisting of tuples $(\bar X\to
\bar S,\bar Y,i, w)$ where $\bar Y\supseteq f^{-1}\partial \bar S$. We
construct the category $\M(\bar S,\partial \bar S,\Sigma)$ as in \S\ref{section:motivicsheaves},
as the category of $End^\vee(H_\Sigma|_{\Delta(\bar S,\partial \bar
  S,\Sigma)})$-comodules, and  $\M(\bar S,\partial \bar S)$ is the
colimit of these over $\Sigma$. There is an exact faithful  forgetful functor
$\iota_{\bar S}:\M(\bar S,\partial \bar S)\to \M(\bar S)$.
We define $\Delta_{ex}(S,\Sigma)\subseteq \Delta(S,\Sigma)$ be the
full subgraph of tuples $( X\to  S, Y,i, w)$ which extend to
$\Delta(\bar S,\partial \bar S,\Sigma)$. Then we can  define the subcategory
 $\M_{ex}(S,\Sigma)\subseteq \M(S,\Sigma)$ by taking comodules over $H_\Sigma$
 restricted to $\Delta_{ex}$. Let
$$\M_{ex}(S)=\text{2-}\varinjlim \M_{ex}(S,\Sigma)\subset \M(S)$$

\begin{lemma}
 
 The category $:\M(\bar S, \partial \bar S)$ is equivalent to a
subcategory  $\M_{ex}(S)\subset \M(S)$ via $j^*$.

\end{lemma}

\begin{proof}

  We start by proving that $j^*:\M(\bar S, \partial \bar S,\Sigma)\to \M_{ex}(S,\Sigma)$ is an equivalence.
We have a morphism
$\pi: \Delta(\bar S,\partial \bar S,\Sigma)\to \Delta_{ex}$ given by restriction. This is
surjective by definition, and the fibres of
$\pi$ are clearly connected.
For $(\bar X\to \bar S,\bar Y,i, w)\in \Delta(\bar S,\partial \bar S,\Sigma)$,
$R^ij_{\bar X,\bar Y!}F=0$ outside of $S$. Thus $H(\bar X\to \bar S,\bar Y,i, w)$ coincides with the value of
$H$ on the restriction.
Therefore we have a morphism of pairs $(\Delta(\bar S,\partial \bar S,\Sigma) , H)\to (\Delta_{ex},H)$,
and so $j^*:\M(\bar S, \partial \bar S,\Sigma)\to \M_{ex}(S,\Sigma)$ is an equivalence by corollary~\ref{cor:tildeDelta}.
Passing to the limit yields  the lemma.

\end{proof}

\begin{lemma}
  $\M_c(S)\subseteq\M_{ex}(S)$.
\end{lemma}

\begin{proof}
Given a projective family $X\subset   \PP^N\times S$ over $S$, we get
an extnesion over $\bar S$ by taking the  closure $\bar X\subset
\PP^N\times \bar S$. Given $Y\subset X$, we may take $\bar Y\subset
\bar X$ to be the union of the closure of $Y$ with the preimage of
$\partial S$. Thus we see that $\Delta_c(S,\Sigma)\subset
  \Delta_{ex}(S,\Sigma)$.
\end{proof}

\begin{defn}
 Define $j_!:\M_{ex}(S)\to \M(\bar S)$ by $j_!=\iota_{\bar S'}\circ {j'}^{*-1}$.
\end{defn}

\begin{lemma}
  This is compatible with extension by zero for sheaves $j_!:Sh(S)\to
  Sh(\bar S)$ in the sense that $R_B(j_!\F)=j_!R_B(\F)$. We have
  $j_!\M_c(S)\subset \M_c(\bar S)$.
\end{lemma}

\begin{prop}\label{prop:jshriekjstar}
  There exists a natural transformation $j_!j^*\to 1$ on $\M(\bar S)$ compatible with
  the usual adjunction map for sheaves.
\end{prop}

\begin{proof}
  The proof is similar to the proof of proposition
  \ref{prop:adjointimmersion}.
Let $Mor'\subset Mor \M(\bar S)$ denote the subcategory of morphisms
$M_2\to M_1$ such that there is a commutative diagram
$$
\xymatrix{
 j_!j^*R_BM_1\ar[r]\ar[d]^{\cong} & R_BM_1\ar[d]^{=} \\ 
 R_BM_2\ar[r] & R_BM_2
}
$$
commutes. The functor $(M_2\to M_1)\mapsto M_1$ is clearly faithful
and exact. Therefore by  corollary~\ref{cor:tannaka}, we get a functor
$\M(\bar S)\to Mor'$ such that 
$$h^i_{\bar S}(X,Y)(w) \mapsto [h^i_{\bar S}(X,Y\cup 
f^{-1}\partial \bar S)(w) \to h^i_{\bar S}(X,Y)(w)]$$
for each $(f:X\to \bar S, Y,i,w)\in \Delta(S)$.
\end{proof}

\begin{prop}
  $\PM_c(S)$ forms a stack in the Zariski topology, i.e. objects and
  morphisms can be patched on Zariski open covers.
\end{prop}

\begin{proof}
By compactness and induction, it suffices to treat covers  $S= U_0\cup
U_1$ consisting of two open sets. Given $M_i\in Ob\PM_c(U_i)$ 
 with an isomorphism $f:M_0|_{U_0\cap U_1}\cong M_1|_{U_0\cap U_1}$.
Let $j_i:U_i\hookrightarrow S$ and $j_{01}:U_0\cap U1\to S$ denote the
inclusions. We define a morphism
$$ g:j_{01!}M_ 0\to j_{0!}M_0\oplus j_{1!}M_1$$
extending $1$ on the first factor and $-f$ on the second.
Then $M= coker(g)$ gives an object of $\PM_c(S)$ which restricts
to $M_i$.
\end{proof}

\subsection{Cellular decompositions}\label{section:cellular}

A number of constructions will be based on the
existence ``cellular'' decompositions.
The use of  such  decompositions plays a key role in Nori's work and also
 \cite{arapuraL}. This hinges on a result of Beilinson \cite{beilinson} that Nori
calls the basic lemma. We need a slight modification of this result.
Define a map $f:X\to S$ to be equidimensional (respectively uniform)
if  $\dim X_s$ is constant (respectively, if for every $s$,  all irreducible
components of $X_s$  have the same dimension).

\begin{prop}\label{prop:basic}
  Let  $X\to S$ be a uniform affine  morphism.
Suppose that  $\F$ is a constructible sheaf on $X$ such that $(X\to S,
\F)$ is controlled.
Then there exists a Zariski open cover $\{S^\beta\}$ of $S$, dense affine open subsets $g^\beta:U^\beta\hookrightarrow X^\beta=f^{-1}S^\beta$
such that 
\begin{enumerate}
\item[(1)] $ (U^\beta\to S^\beta,g^\beta_{!}g^{\beta*}\F)$ is controlled.
\item[(2)] For every $s\in S^\beta$,
 $$H^i_{S^\beta}(X^\beta,X^\beta-U^\beta; \F)_s=H^i_{S^\beta}(X^\beta, g^\beta_{!}g^{\beta*}\F)_s = 0$$
unless $i=\dim X_s$.
\end{enumerate}

\end{prop}

\begin{proof}  
Most of the  argument is pretty much identical
to the proof of \cite[lemma 3.3]{beilinson}. Nevertheless, we spell
this out since the result is central.
Let 
$$
\xymatrix{
 X\ar[r]^{g}\ar[d]_{k} & T\ar[r]^{h} &  S \\ 
 \bar X\ar[ru]^{\bar g}\ar[urr]_{\bar f}&  & 
}
$$
be a control diagram.  After replacing $\bar X$ by the blow up along
$\bar X-X$, we can assume that $k$ is affine.
We can choose a divisor $Z'\subset T$ such that $\F$ is constant on
the complement of $Z=g^{-1}Z'$.  Let $\ell: X-Z\hookrightarrow X$ be
the inclusion, then $M=\ell_!\F_{X-Z}$ is also controlled by the above
diagram. Set $\bar M=\R k_* M$. 
Note that $\F[\dim X_s],M[\dim X_s]$ and $\bar M[\dim X_s]$ restrict to 
 perverse sheaves on the fibres of $f$ or $\bar f$, because $\ell$ and
 $k$ are affine embeddings \cite[cor. 4.1.3]{bbd}.

Fix an embedding $\bar X\subset \PP^N\times S$ over $S$.
Given a hyperplane $H$, let
$$
\xymatrix{
 V=\bar X-H\ar[r]^{j} & \bar X  & H\ar[l]_i\\ 
 X-H=V\cap X\ar[u]^{k'}\ar[r]^>>>>{j'} & X\ar[u]^{k} & H\cap X\ar[l]_{i'}\ar[u]^{k''}
}
$$
denote the inclusions.  Let $j_s:\bar X_s-H\hookrightarrow \bar X_s$ etc. denote the restrictions of these inclusions
to the fibre over $s\in S$. 
We claim that
\begin{equation}
  \label{eq:beilinson1}
  j_!\R k'_* M|_{V\cap X} \cong \R k_* j'_! M|_{V\cap X}
\end{equation}
holds for a dense open set $P$ of hyperplanes $H$ in the dual projective
space $\check{\PP}=\check{\PP}^N$. The argument which we sketch is from
\cite[pp 35-36]{beilinson}. First note that  \eqref{eq:beilinson1} equivalent to
\begin{equation}
  \label{eq:beilinson1a}
  i^*\R k_* M \cong \R k''_* {i'}^* M
\end{equation}
by the argument indicated in [loc. cit.].
Let $H_P\subset \bar X_P=\bar X\times \check{\PP}$ denote the universal
hyperplane.  Let $M_P$ denote the pullback of $M$ to $X_P=X\times \check{\PP}$.
Let $i_P:H_P\to \bar  X_P$ denote the canonical morphism.
Similarly the other morphism $i',\ldots$ have obvious extensions denoted by
$(-)_P$ such that they can be recovered by taking  the fibre at $H\in
\check{\PP}$. With this notation, it suffices to prove
\begin{equation*}
  i_P^*\R k_{P*} M_P \cong \R k''_{P*} {i'_P}^* M_P
\end{equation*}
by virtue of theorem \ref{thm:constrBC} (2). But this is a consequence of the
 \ref{thm:constrBC} (3), because $i_P$ is locally trivial.

By combining this with proposition~\ref{prop:controlH}, we can find a
cover $\{S^\beta\}$ and
hyperplanes  $H^\beta\in P$ so that $j'_!M|_{V\cap X}$ is controlled
over $S^\beta$. In order to simplify notation, replace $S$ by
$S^\beta$ etc. below, for a  fixed $\beta$.
We set $U=X-H-Z$ with
inclusion $g$. Then $g_!g^*\F=j'_!M|_{V\cap X}$. So the first item
of the proposition holds, and in particular, this sheaf has the base
change property.

It remains to prove item (2).
Since each fibre $X_s$ is affine, we have $H^i_S(g_!g^*\F)=0$ for
$i>\dim X_s$ by
Artin's vanishing theorem. The remaining half also follows from
the affineness  of $X$. As $f_s$ is affine, $\R f_{s!}$ is left
$t$-exact for the perverse $t$-structure \cite[cor. 4.1.2]{bbd}.
Therefore 
$$ H_c^i(V\cap X_s, \bar M|_V[\dim X_s])= \calH^i(\R f_{s!}\bar
M|_{V}[\dim X_s])=0$$
 for $i<0$. Then from this, (\ref{eq:beilinson1}) and the proper
 base change theorem  we obtain,
\begin{eqnarray*}
  R^if_*(j'_!M|_{V\cap X})_s
&\cong& \calH^i(\R f_* j'_{!} M|_{V\cap X})_s\\
&\cong& \calH^i(\R \bar  f_* \R k_* j'_{!} M|_{V\cap X})_s\\
&\cong& \calH^i(\R \bar  f_* j_{!}\R k_*  M|_{V\cap X})_s\\
&\cong& \calH^i(\R \bar f_* j'_{!} M|_{V\cap X})_s\\
&\cong& \calH^i(\R \bar f_* j_{!} \bar M|_{V})_s\\
&\cong& (R^i(\bar f\circ j)_! \bar M|_{V})_s\\
&=& 0
\end{eqnarray*}
for $s\in S$ and  $i<\dim X_s$.

\end{proof}

\begin{cor}
  The schemes $U^\beta\to S^\beta$ can be assumed to be smooth.
\end{cor}

\begin{cor}
   If $X\to S$ is  equidimensional of relative dimension $n$. Then
 $$H^i_{S^\beta}(X, g^\beta_{!}g^{\beta*}\F) = 0$$
unless $i=n$.
\end{cor}

 We define $\Delta_{eq}(S,\Sigma)\subset \Delta(S,\Sigma)$ by 
 requiring
 $X\to S$ and $Y\to S$ to be equidimensional. The categories $\M_{eq}(S,\Sigma)$ and $\M_{eq}(S)$
 are defined by the same procedure as before by restricting $H$ to  $\Delta_{eq}$.
 These  can be viewed as subcategories of $\M(S)$.

\begin{lemma}\label{lemma:cell}
  Let $X\to S$ be an affine equidimensional morphism to a variety
  which is controlled with respect to a sheaf $\F$.
 Then  there  is a Zariski open cover $\{S^\beta\}$ of $S$ and filtrations 
$$X_{0}^\beta\subset X_{1}^\beta\subset\ldots X_{n}^\beta= X ^\beta=f^{-1}S ^\beta$$
 such that
\begin{enumerate}
\item[(1)]  $X_{i}^\beta\to S ^\beta$ is equidimensional of pure relative dimension $i$.
\item[(2)] The pairs $(X_a ^\beta\to S,X_{a-1}^\beta)$ are
  controlled with respect to $\F$, and 
$$H^i_{S^\beta }(X_a ^\beta,X_{a-1}^\beta;\F)=0$$
  for $i\not=a$. 
\end{enumerate}

If in addition,
$$X_0'\subset X_1'\subset\ldots X_n'= X$$
 is a given  chain of closed sets  each of pure relative dimension
 $i$.  Then  we can choose $X_i'\cap X ^\beta\subseteq X_{i}^\beta$.
\end{lemma}

\begin{proof}
  This follows from the previous proposition  and induction on $\dim X$.
\end{proof}

\begin{defn}\label{defn:QVar}
Suppose that we are given a morphism $\tilde X\to X$ of $S$-schemes,
a Zariski open cover $\{S^\beta\}$  of $S$,  filtrations by
closed sets of the preimage of each $S^\beta$
$$\tilde X^\beta=\tilde  X_d^\beta\supset \tilde X_{d-1}^\beta\supset\ldots \tilde X_0^\beta \supset
\tilde X_{-1}^\beta=\emptyset$$ We refer to  the collection $(\ldots, \tilde
X_\dt^\dt )$ as a quasi-filtration on
$X$, and the whole thing as a quasi-filtered $X$-variety. 
\end{defn}

 These objects form a
category $QVar_S$, where a morphism 
$$\phi:(\{S^\beta\}_{\beta\in B}, \tilde X'\to X, \tilde
{X'}_\dt^\beta)\to (\{T^\gamma\}_{\gamma\in G},\tilde X\to X, \tilde
X_\dt^\gamma)$$  
is given by a map $r:B\to  G$, such that $S^\beta\subseteq T^{r(\beta)}$ plus 
 commutative squares of $S$-schemes
$$
\xymatrix{
  \tilde {X'}\ar[r]\ar[d] & \tilde X\ar[d] \\
  {X'}\ar[r]^{f} & X }
$$
with $\tilde {X'}_\dt^\beta$ mapping to $\tilde X_\dt^{r(\beta)}$. We say that
$\phi$ covers $f$.  
Let say that the quasi-filtration is simple if the cover
$\{S^\beta\}$ consists of $\{S\}$ alone.
A filtered variety is the  special case  of a simple quasi-filtered
variety,  where $\tilde X\to X$ is the
identity. Let   $FVar_S$ be the full subcategory of filtered varieties.

We give   a relative version of Jouanolou's trick \cite[lemma
1.5]{jou} below. To simply the statement, let us say that $\tilde X\to X$ is
{\em bundle of affine spaces} if $\tilde X=T \times_{\text{Aff}(n)}\mathbb{A}_k^n$,
where $T\to X$ is a torsor  for the affine group in the Zariski topology

\begin{lemma}\label{lemma:jou}
If $f:X\to S$ is a quasi-projective morphism then there exists a commutative diagram
$$
\xymatrix{
 \tilde X\ar[r]^{\pi}\ar[rd]^{\tilde f} & X\ar[d]^{f} \\ 
  & S
}
$$
such that $\tilde f$ is affine, and $\pi$ is a 
bundle of affine spaces.
\end{lemma}

\begin{proof}
When $X=\PP^N\times S$, $\tilde X$ can be taken to be product of $S$ with
the complement of the incidence variety $St_N=\{(x,H)\in \PP^N\times \check{\PP}^N\mid x\notin H\}$.
% with the
% Stieffel variety $St_N = GL(N+1)/(GL(1)\times GL(N))$ as in the proof of \cite[lemma 1.5]{jou}.
For the general case,
let $X\subset \bar X\subset \PP^N\times S$ be a relative
compactification. 
After blowing up, we can assume that $\bar X-X$ is a divisor. 
Then the preimage $\tilde X$ of $ X$ in $St_N\times S$ will do the job.
\end{proof}

When $f:X\to S$ is projective, we see that the pullback of any
constructible sheaf $\pi^*\F$ is necessarily controlled.
We say that a quasi-filtration $(\pi:\tilde X\to X, \tilde X_\dt)$ is {\em cellular}
with respect to a controlled constructible sheaf $\F$ if 
\begin{enumerate}
\item $\pi$  is a bundle of affine spaces,
\item $\pi^*\F$ is controlled,
\item $\tilde X_\dt$  satisfies condition (2)  of   lemma
  \ref{lemma:cell}  with respect to $\pi^*\F$, i.e.
$$H^i_{S^\beta }(\tilde X_a ^\beta,\tilde X_{a-1}^\beta;\pi^*\F)=0$$
\end{enumerate}
Note that the first assumption implies that $\R \pi_*\pi^*\F= \F$. The
second assumption can be seen to be redundant, but there is no harm in
including it.

\begin{lemma}\label{lemma:affinespacebundle}
  A bundle $\pi:\tilde X\to X$ of affine spaces over an affine scheme
  admits a section.
\end{lemma}

\begin{proof}
  The bundle $\tilde X/X$ is a homogeneous space  associated to a torsor
for the affine group $\text{Aff}(n)$. Using the exact sequence
$$1\to G_a^n\to \text{Aff}(n)\to GL(n)\to 1$$
and the fact that $X$ is affine, we conclude that $H^1(X,G_a^n)=0$,
and therefore that  $\tilde X/X$ is a vector bundle. So it has a section.
\end{proof}

\begin{prop}\label{prop:cell}
\-
  \begin{enumerate}
  \item[(1)] Every equidimensional  quasiprojective morphism possesses
    a cellular quasi-filtration with respect to a given controlled
    sheaf $\F$.
\item[(2)] Every morphism of equidimensional quasiprojective schemes over $S$
  can be lifted to a   morphism of cellular quasi-filtered
  varieties with respect to  a controlled constructible sheaf on the
  target.

\item[(3)] The category of
  cellular quasi-filtrations of a fixed pair $(X,\F)$  is connected i.e. any
  two objects can be connected by a chain of morphisms.
  \end{enumerate}

\end{prop}

\begin{proof}

 The first two statements  follow immediately  from  lemmas
 \ref{lemma:cell} and  \ref{lemma:jou}. The remaining part, take a
 bit more work.

First we treat the special case of  (3) for cellular filtrations. Given two such filtrations
$X_\dt, X_\dt'$, lemma \ref{lemma:cell} shows that there is 
third cellular filtration $X_\dt''\supseteq X_\dt\cup X_\dt'$.
Now we prove it general. Suppose that we have  cellular
 quasi-filtrations $(S^\beta, \tilde X\to X, \tilde X_\dt^\beta)$ and
 $(T^\gamma, \tilde  Y\to
 X, \tilde Y_\dt^\gamma)$. Take the fibre product $\tilde Z= \tilde X\times_X
 \tilde Y$. By lemma \ref{lemma:affinespacebundle}, $\tilde Z\to \tilde X$
 and $\tilde Z\to \tilde Y$ admit sections $\sigma$ and $\tau$. 
 Lemma \ref{lemma:cell} shows that we can refine $\sigma(\tilde
 X_\dt^\beta)\cup \tau(\tilde Y_\dt^\gamma)$ to a cellular filtration
 $\tilde Z_\dt^\dt$ of
 $Z$. By refining $\tilde X_\dt^\dt$ and $\tilde Y_\dt^\dt$, we obtain
 a diagram of cellular quasi-filtrations
$$ 
(\tilde X, \tilde X_\dt^\dt)\to (\tilde X, \tilde
{X'}_\dt^\dt)\leftarrow (\tilde Z,\tilde Z^\dt_\dt)\to (\tilde Y',\tilde {Y'}_\dt^\dt)
\leftarrow (\tilde Y,\tilde Y_\dt^\dt)
$$

\end{proof}

Given a  filtration $X_\dt\subset X$ by closed
sets and a sheaf $\F$, we have a spectral sequence
$$E_1^{pq}= H_S^{p+q}(X_p,  X_{p-1};\F)\Rightarrow H_S^{p+q}(X, \F)$$
cf \cite[(10)]{arapuraL}.
When this is cellular, this reduces to an isomorphism at $E_2$. 
Then putting this remark together with the above results yields

\begin{lemma}\label{lemma:complex}
Suppose that $(\pi:\tilde X\to X, \tilde X_\dt)$ is cellular with
respect to $j_{XY!}F$. 
$H_S^i(X,Y;F)$ is 
isomorphic to the $i$th cohomology of the complex
$$
\ldots H_S^{i}(\tilde X_i,(\pi^{-1}Y\cap \tilde X_i)\cup \tilde X_{i-1})\to H_S^{i+1}(\tilde X_{i+1}, 
(\pi^{-1}Y\cap \tilde X_{i+1})\cup \tilde X_{i}) \ldots   
$$
\end{lemma}

\subsection{Tensor products}\label{section:tensor}

We have a product structure on $\Delta(S,\Sigma)$ (and $\Delta_{eq}(S,\Sigma)$)
 given by
$$(X\to S, Y, i,w)\times (X'\to S, Y', i',w')
= (X\times_S X'\to S, X\times_S Y'\cup X'\times_S Y, i+i',w+w')$$
which makes it into a monoid in the category of graphs with unit $(id_S,\emptyset,0,0)$.
Unfortunately, this does not immediately lead to a product on
$\M(S,\Sigma)$. The problem has to do  with the K\"unneth
formula. To remedy this, we
 define a full subgraph
 $$\Delta_{cell}(S,\Sigma) \subset \Delta_{eq}(S,\Sigma)$$
 The objects of $\Delta_{cell}$ consist of 
 quadruples $(X\to S, Y, i,w)$ such that $X\to S$ is affine and
such that $H^j_S(X,Y) = 0$ unless $j=i$, and such that  $X-Y\to S$ is smooth.
Thanks to K\"unneth's formula, we have a commutative diagram
$$
\xymatrix{
\Delta_{cell}(S,\{S\})\times \Delta_{cell}(S,\Sigma)\ar[r]\ar[d]^{H_{\{S\}}\times H_\Sigma}&\Delta_{cell}(S,\Sigma)\ar[d]^{H_\Sigma}\\
F\text{-mod}\times  F\text{-mod}
\ar[r]^>>>>{\otimes} & F\text{-mod}
}
$$
leading to a  product 
$$End^\vee(H|_{\Delta_{cell}(S,\{S\}}))\text{-comod}\times 
End^\vee(H|_{\Delta_{cell}(S,\Sigma)})\text{-comod}\to
End^\vee(H|_{\Delta_{cell}(S,\Sigma)})\text{-comod}$$
With this, $\PM_{cell}(S)=End^\vee(H|_{\Delta_{cell}(S,\{S\})})\text{-comod}$ becomes
a tensor category.  We form the associated stack $\M_{cell}(S,\Sigma)$
as before. The tensor product extends to this. To summarize

 \begin{lemma}\label{lemma:prodMcell}
There are tensor products
$$\M_{cell}(S,\{S\})\times \M_{cell}(S,\Sigma)\to \M_{cell}(S,\Sigma)$$
compatible, via the forgetful functor $U$, with the vector space tensor product.
 With this structure $\M_{cell}(S,\{S\})$ becomes a tensor category.
\end{lemma}

The key point is:

\begin{thm}\label{thm:cellveq}
The category
  $\M_{cell}(S)$ is equivalent
  to $\M_{eq}(S)$.
\end{thm}

Before giving the proof, we give a construction.
 Let $C^{[0,\infty)}(\M(S,\Sigma))$ be the
category of bounded complexes supported in nonnegative degrees.
Let $\calH^i: C^{[0,\infty)}(\M(S,\Sigma))\to\M(S,\Sigma) $ denote the $i$th cohomology
functor. Then composition gives a functor $R_B\circ \calH^*$ from
$C^{[0,\infty)}(\M(S,\Sigma))$ to the
category $Gr\, Sh(S((\C))$ of $[0,\infty)$-graded sheaves. % We also have a functor
% $H_S^*\circ R_B:\M(S,\Sigma)\to Gr\,Sh(S(\C))$.
Let $\calC(S,\Sigma)$ be
the so called comma category whose objects are triples
$$(K^\dt, M, \phi:  R_B(M)\to R_B\circ \calH^0(K^\dt))$$
where $K^\dt\in Ob C^{[0,\infty)}(\M(S,\Sigma))$ and $ M\in Ob\M(S,\Sigma)$.
Morphisms are pairs $K_1^\dt\to K_2^\dt$, $M_1\to M_2$ satisfying
obvious compatibilities. Let $\calC_{iso}(S,\Sigma)$ be the full subcategory
consisting of triples for which $\phi$ is an isomorphism.
We can identify $F^2\text{-mod}$ with $F\text{-mod}\times  F\text{-mod}$.
There is a faithful exact functor
$U_2:\calC(S,\Sigma)\to (F\text{-mod})^2$ given by $(K^\dt, M,\phi)\mapsto
(\prod_i U(K^i))\times U(M)$.

Given a  simple quasi-filtration $(T\to S, T_\dt)$ and a stratification $\Sigma$, choose 
base points $s$ for $(S,\Sigma)$ and $t_\dt\in T_\dt$. 
 We define a functor $H^\#:\Delta(S,\Sigma)^{op}\to \calC(S,\Sigma)_{iso}$ as follows.  On
  objects
  \begin{equation*}
    H^\#(X\to S,Y,i,w) =
  \end{equation*}
  \begin{equation}
    \label{eq:Hsharp}
((h_S^{0}(X_{T_0},Y_{T_0}\cup X_{T_{0-1}})\to
h_S^{1}(X_{T_1},Y_{T_1}\cup 
X_{T_{1-1}})\to \ldots)[i] ; h^i_S(X,Y);\phi)
  \end{equation}
where the differentials of the complex are
connecting maps and $\phi$ is given by lemma~\ref{lemma:complex}. % proposition \ref{prop:leray}.

\begin{defn}
  $\PM^\#(f,T\to S, T_\dt,\Sigma)= End_{F^2}^\vee(H^\#\circ U_2)$-comod.
\end{defn}

This carries an exact faithful embedding into $F^2$-mod.
The categories $\PM^\#(f, T\to S,   T_\dt,\Sigma)$ fibred over
$S$-schemes.
We can form the associated stack $\M^\#(f, T\to S   T_\dt,\Sigma)$ for
the Zariski topology. This can be constructed explicitly by following
the procedure outlined at the end of \S~\ref{sect:extmotivicsheaves}.
In order to simplify notation, we usually just write this as
$\M^\#(T_\dt)$, when the rest of the data is understood.
We let $h_{T_\dt}(X,Y)(w)$ denote the object of this category associated
to $(X,Y,i,w)$. We have a functor $\M^\#(T_\dt) \to \calC(S,\Sigma)_{iso}$, and a  functor
$p:\M^\#(T_\dt)\to  \M(S,\Sigma)$ 
given as a composition
of this with the functor $\calC(S,\Sigma)_{iso}\to \M(S,\Sigma)$ given by
projection onto the second factor. 
From lemma \ref{lemma:surjcatstack}, we obtain
%From corollary \ref{cor:surjcat}, we obtain

\begin{lemma}\label{lemma:surjcat}
  $\M(S,\Sigma)$ is equivalent to $\M^\#(T_\dt)/\ker p$.
\end{lemma}

\begin{proof}[Proof of theorem \ref{thm:cellveq}]

  Restriction gives a functor $\iota: \PM_{cell}(S,\Sigma)\to \PM_{eq (S,\Sigma)}$ which is
  necessarily exact and faithful. It suffices to show that this is an
  equivalence, because it will then induce an equivalence of the
  corresponding stacks $\M_{cell}(S,\Sigma)\sim \M_{eq (S,\Sigma)}$. 
We  show that $\iota$ is
essentially surjective and full, and for this it suffices to have a right
inverse  up to natural equivalence.  This is induced by the functor
$\PM^\#(T_\dt)\to \M_{cell}(S)$ given by 
$$ (X\to S,Y,i,w) \mapsto \calH^0(h_S^{0}(X_{T_0},Y_{T_0}\cup X_{T_{0-1}})\to
h_S^{1}(X_{T_1},Y_{T_1}\cup  X_{T_{1-1}})\to \ldots)[i] $$

\end{proof}

\begin{cor}
 There is a K\"unneth decomposition for motives associated to objects in $\Delta_{eq}(S,\{S\})$:
$$h_S^i(X\times_S X' ,X\times_S Y'\cup X'\times_S Y)\cong \bigoplus_{j+j'=i}
h_S^j(X,Y)\otimes h_S^{j'}(X',Y')$$  
\end{cor}

\begin{proof}
  This follows from the theorem and lemma \ref{lemma:prodMcell}.
\end{proof}

For objects in $\Delta_{eq}(S,\{S\})$, we get exterior products 
$$h_S^j(X,Y)\otimes h_S^{j'}(X',Y')\to h_S^{j+j'}(X\times_S X' ,X\times_S Y'\cup X'\times_S Y)$$
and cup products 
$$h_S^j(X,Y)\otimes h_S^{j'}(X,Y)\to h_S^{j+j'}(X , Y)$$
by composing this with the  restriction to  the  diagonal.
Corollary \ref{cor:tannaka2} shows that these products are compatible
with the standard tensor products on the categories of classical and \'etale local systems.

 \subsection{Ind objects}

Let  $\text{Ind-}\mathcal{A}$ denote the category of  Ind-objects of
a category $\mathcal{A}$ obtained by formally adjoining filtered
colimits \cite{ks}. This is abelian, when $\A$ is [loc. cit.]. This
can be given a concrete description in many cases.
For example, it is well known that $\text{Ind-}F\text{-mod}$ can be identified with
$F\text{-Mod}$. We extend this to sheaves.
Recall that an object $c$ of an additive category is compact or finitely presented if $Hom(c,-)$ commutes with
arbitrary small coproducts. For example, a finite dimensional vector space $V$
is seen to be compact in the category of all vector spaces, because an
element of $Hom(V,-)$ is determined by its value on a finite basis.  
For essentially the same reasons, we have:

\begin{lemma}
  A constructible sheaf is compact in the category $Sh(S)$ of sheaves
  of $F$-modules.
\end{lemma}

\begin{proof}
For any collection of sheaves $\F,\G_i$, we have  
canonical map 
$$\kappa:\bigoplus_{i\in I} Hom(\F,\G_i)\to Hom(\F, \bigoplus_{i\in I} \G_i)$$
Injectivity  can be checked on stalks. 
Consider the projection 
$$p_j:\bigoplus \G_i\to \prod \G_i\to \G_j$$
One checks that $\phi\in Hom(\F,\oplus \G_i)$ lies in $im(\kappa)$ precisely
when it has finite support in the sense that there exists a finite subset $J\subset I$ such that the germs $p_i(\phi_s)=0$ for
all $s\in S$ and $i\notin J$.

Suppose $\F$ is constructible with respect
to a necessarily finite Zariski stratification $\Sigma$.  Let $\pi_\sigma:\tilde \sigma \to
\sigma$ denote the universal cover of a stratum. 
Choose bases of cardinality say $n _\sigma$  for each $H^0(\pi_\sigma^* \F)$.
 Then $\phi\in Hom(\F,\oplus \G_i)$ is determined by its image
 $$r(\phi)= (\pi_\sigma^* \phi)\in \prod_\sigma H^0(\pi_\sigma^*\mathcal{H}om(\F,\oplus \G_i)),$$
i.e. the map $r$ is injective.  In
explicit terms, $r(\phi)$ is given
by a collection of  $n_\sigma$ sections of  $\oplus
\pi_\sigma^*\G_i$ for each $\sigma$. The projections $r(p_j(\phi))$ are given
by simply projecting these sections to $\G_j$.
 It should now be clear that $\phi$ has finite support.
\end{proof}

\begin{cor}
  There is a fully faithful exact embedding of $\text{Ind-}Cons(S)$ into $Sh(S)$.
\end{cor}

\begin{proof}
  This follows from \cite[prop 6.3.4]{ks} and the exactness of
  filtered colimits.
\end{proof}

Therefore we have an exact faithful functor
$$
 \text{Ind-}\M(S)\to  \text{Ind-}Cons(S)\to Sh(S)
$$
given by composition. This is also denoted by $R_B$.
\section{Direct Images}

\subsection{Direct Images (abstract construction)}\label{section:directimages0}

We  start by giving a general construction of direct images. 
Set $DM(S)= D(\text{Ind-}\M(S))$.
Fix a  morphism $f:S\to Q$.
Since the functor $f^*$ is exact, it extends to an exact functor on
$\text{Ind-}\M(Q)\to \text{Ind-}\M(S)$. 
Thus we have an extension $f^*:DM(Q)\to DM(S)$ as a triangulated functor.

\begin{thm}\label{thm:directimage0}
  If $f:S\to Q$ is a  morphism of quasiprojective varieties, 
then there is a triangulated functor $rf_*:DM(S)\to DM(Q)$  which is  right adjoint to $f^*$.
\end{thm}

\begin{proof}
  The  theorem will be deduced from a form of  Brown's  representability theorem
  due to Franke \cite{franke}. Note that the extension $f^*$ to
  $\text{Ind-}\M(-)$ commutes with filtered direct limits and therefore coproducts.
Since $\text{Ind-}\M(S)$ is a Grothendieck
  category by \cite[thm 8..6.5]{ks},  Franke's theorem \cite[thm 3.1]{franke} implies that
$$ M\mapsto Hom(f^* M, N)$$
is representable by an object $rf_*N$. The map $N\to rf_*N$ extends to
a functor which is necessarily the right adjoint,
cf. \cite[p223]{neeman}.
Moreover, this  is automatically
   triangulated  by \cite[prop 3.3.8]{lipman}. 
\end{proof}

The abstract construction is not terribly useful by itself. We would really like more:

\begin{defn}
Let us say that a morphism $f:S\to Q$ possesses a good direct
image if
\begin{enumerate}
\item
 $rf_*(D^b(\M(Q))\subseteq D^b(\M(S))$, where we identify $D^b\M$ with a triangulated
 subcategory of $DM$.
\item For each $M\in \M(S)$. The map
$$R_B rf_*M\to \R f_* R_B M$$
adjoint to the canonical map
$$f^*R_B rf_* M\cong R_Bf^*rf_*M\to R_B M$$
is an isomorphism.
\end{enumerate}
\end{defn}

The following lemma gives a criterion for checking this.

\begin{lemma}\label{lemma:adjointness}
Suppose that   $r'f_*:D^b(\M(S))\to D^b(\M(Q))$ is a functor equipped with natural transformations
$$\eta:1\to r' f_* f^*$$ and 
$$\epsilon:f^*r'f_*\to 1$$
such that: 
\begin{enumerate}
\item[(1)] The map
$$R_B r'f_*M\to \R f_* R_B M$$
adjoint to
$$R_B\epsilon:f^*R_B r'f_* M\to R_B M$$
is an isomorphism.
\item[(2)] The composition
$$R_BM\stackrel{R_B\epsilon}{\longrightarrow} R_B r'f_* f^* M\stackrel{(1)}{\longrightarrow}  \R f_* f^* R_B M$$  
coincides with the adjunction map $1\to \R f_* f^*$.
\end{enumerate}
Then $r'f_*$ is right adjoint to $f^*$. So, in particular, $f$ has a
good direct image.
\end{lemma}

\begin{proof}
 It is enough to check that the compositions
$$f^*\to f^* r'f_* f^*\to f^*$$
and
$$r'f_* \to r'f_* f^* r'f_* \to r'f_*$$
are both identity \cite[chap IV]{maclane}. Since the  realizations $R_B$ are embeddings,
this follows  from the compatibility of $\eta,\epsilon$ with 
the  usual adjunctions on  the categories of sheaves.
\end{proof}

As a prelude to a more general result proved later,
we show that $f:S\to Q$ has a good direct image when it is a closed immersion.
 By \ref{cor:tannaka}, the map $\Delta(S,\Sigma)\to \M(Q)$ given by 
$$(X\to S,Y,i,w)\mapsto h_Q^i(X, Y)(w)$$
induces an exact functor  $f_*:\M(S)\to \M(Q)$.

\begin{prop}\label{prop:adjointimmersion}
If $f$ is a closed immersion, then  $f_*$ satisfies the conditions of lemma~\ref{lemma:adjointness}.
Therefore, $f_*$ is right adjoint to $f^*$.
\end{prop}

\begin{proof}
We have to construct natural transformations
$\eta:1\to f_* f^*$ and $\epsilon:f^*f_*\to 1$ satisfying the conditions of  lemma~\ref{lemma:adjointness}.

We can see from the construction that $f^*f_* h_S^i(X,Y)(w)$
is equal to $h_S^i(X,Y)(w)$. Thus we have a canonical isomorphism, which  gives the required  map $\epsilon$.
This clearly satisfies  lemma~\ref{lemma:adjointness} (1).

Let $Mor\M(Q)$ denote the category whose objects are morphisms of $\M(Q)$, and whose 
morphisms are commutative squares. Let $Mor'\subset Mor(\M(Q))$ denote the subcategory
of morphisms $M_1\to M_2$ such that there is a commutative diagram
$$
\xymatrix{
 R_BM_1\ar[r]\ar[d]^{=} & f_*f^*R_BM_1\ar[d]^{\cong} \\ 
 R_BM_1\ar[r] & R_BM_2
}
$$
commutes. The morphisms are squares
$$
\xymatrix{
 M_1\ar[r]\ar[d]^{h_1} & M_2\ar[d]^{h_2} \\ 
 M_1'\ar[r] & M_2'
}
$$ 
such that $R_Bh_2$ is given by $f_*f^* h_1$. The functor $(M_1\to M_2)\mapsto M_1$ is clearly faithful
and exact. Therefore by  corollary~\ref{cor:tannaka}, we get a functor $\M(Q)\to Mor'$ such that 
$$h^i_Q(X,Y)(w) \mapsto [h^i_Q(X,Y)(w) \to h^i_Q(X_S,Y_S)(w)]$$
This gives the canonical  adjunction $\eta:1\to f_*f^*$.
\end{proof}

Combining this with proposition \ref{prop:jshriekjstar} yields:

\begin{lemma}\label{lemma:jshreikistar}
Let $j:S\to \bar S$ be an open immersion with complement $i:\bar
S-S\to S$. Then for 
  any $\F\in \M(\bar S)$, there is a canonical exact sequence
$$0\to j_!j^*\F\to \F\to i_*i^*\F\to 0$$
where $i:\partial \bar S\to \bar  S$ is the inclusion.
\end{lemma}

\subsection{Direct Images (conclusion)}\label{section:directimages}

We come to the main technical result of this paper.

\begin{thm}\label{thm:directimage}
A  morphism   $f:S\to Q$ possesses a good direct image if either $f$
is projective or $Q$ is a point.
\end{thm}

The proof, which will be broken into a series of lemmas, is quite
messy, although the basic idea is rather simple. The hypothesis of the
theorem is used in the following way:
 a controlled pair $(g:X\to S,Y)$  determines
a controlled pair $(f\circ g:X\to Q,Y)$ when either $f$ is projective
or trivially when $Q$ is a point. Let us say that $(X\to S,
Y, i,0)\in \Delta(S)$ is a $f$-cellular if $H_Q^j(H_S^i(X,Y))$ is zero for all but one
value of $j$, say $j=m$. Then the proof will show that
$h_Q^{m+i}(X,Y)[m]$ will give a model for $rf_*(h_S^i(X,Y))$. This
clearly maps to $\R f_* H_S^i(X,Y)$ under $R_B$, so ``goodness'' is
verified in this case.
In  general, we will realize  $rf_*M$ as an explicit complex of
$f$-cellular motives, which maps to $\R f_* R_BM$. Since
this construction depends on auxiliary  choices, it is necessary
work on a bigger category $\M^\#$, lying over $\M$, in order to get a functor temporarily
called $ q$. The final step is to show that $ q$ descends to a
functor on the derived categories, and that this is indeed the
adjoint to $f^*$.

By factoring $f$ through a closed immersion followed by a projection,
and applying proposition \ref{prop:adjointimmersion}, we can see that
 to prove theorem \ref{thm:directimage},
we can assume that $f:S\to Q$ is flat and therefore equidimensional.
 Let  $g:X\to S$ be a quasi-projective morphism with $Y\subset X$
 closed, such that $(X\to S, Y)$ is controlled.
 By proposition~\ref{prop:cell}, we can find a quasi-filtration
  $ (\{Q^\dt\}, T\to  S, T_\dt^\dt)$  which is cellular with respect to the sheaves
  $H^*_{S}( X,  Y)$.  To simplify the discussion, let us suppose that
  this is simple, i.e. that the cover $\{Q^\dt\}=\{Q\}$.
Consider the commutative diagram
$$
\xymatrix{
 X_{T_\dt}\ar[r]\ar[d] & X_T\ar[r]\ar[d] & X\ar[d]^{g}\ar[rd]^{f\circ g} &  \\ 
 T_\dt\ar[r] & T\ar[r] & S\ar[r]^{f} & Q
}
$$
where the squares are cartesian.
 Then we can form a complex of sheaves
 $$\calK_i^\dt= H_Q^{i}( X_{ T_0}, Y_{ T_0}\cup X_{ T_{0-1}})
 \to H_Q^{i+1}(  X_{  T_1},  Y_{  T_1}\cup   X_{  T_{1-1}})\to \ldots  
 $$
where the differentials are the connecting maps.

\begin{prop}\label{prop:leray}
With the previous assumptions, there is a canonical isomorphism
$$
  H_{Q}^j(S, H_S^i(X,Y))\cong 
\calH^j(\calK_i^\dt)
 $$
where $\calH^j$ stands for the $j$th cohomology sheaf.
\end{prop}

\begin{proof} %CHECK
When $Q$ is a point and $Y=\emptyset$, this was originally proved in \cite[thm 3.1]{arapuraL}. The general case can be proved by the same
method. However,
a slightly cleaner alternative is to deduce it from \cite{cm}. Since
the model case ($Q=pt,Y=\emptyset$) is spelled out in  detail in \cite{cataldo},
 we will be content to give the broad outline.
Since both sides of the purported isomorphism
 are stable under base change to $T$, we can assume that $  T=  S$.
Let  $ L = \R f_*\R g_*j_{XY!}F$.
We consider two filtrations on $L$. The first  is defined by truncations $P^\dt(L)=\R f_*\tau_{\le-\dt}\R g_*j_{XY!}F$,
so that $Gr_P^\dt L = \R f_* H^{-\dt}_S(X,Y)$. The second $\F$ is  the filtration on $L$ associated to 
$T_\dt$, 
$$\F^\dt L= \R f_* j_{S T_\dt !}j_{S T_\dt}^*\R g_*j_{XY!}F$$
Then 
$$ Gr_\F^\dt L\cong \R f_*  j_{T_\dt T_{\dt-1} !}j_{T_\dt T_{\dt-1}}^*\R g_*j_{XY!}F
\cong\R  f_*\R g_*j_{ X_{\T_\dt}, Y_{ T_\dt}\cup X_{T_{\dt-1}} !}F$$
holds.
 Then the cellularity of $T_\dt$ implies that the assumptions of \cite[prop 5.6.1]{cm} are satisfied.
Therefore we have a natural  filtered quasi-isomorphism
$$(L, P)\cong (L, Dec(\F))$$
where $Dec(\F)$ is the shifted filtration associated to $\F$. This has the property
that $E_1(Dec(\F)) = E_2(\F) $ (c.f.   \cite{deligne-hodge}).
Thus it follows that there is an isomorphism of the spectral sequences associated to $P$ and
$Dec(\F)$. The spectral sequence for $P$ is Leray with a shift in indices, and in particular
$E_1(P)=H_{Q}^j(S, H_S^i(X,Y))$. On the other hand,
we can identify
$$ E_1(\F)=H_Q^{\dt}( X_{ T_\dt}, Y_{T_\dt}\cup X_{T_{\dt-1}}),$$ 
using the fact that  $H_Q^{\dt}(X_{T_\dt},\ldots)$
commutes with base change because the maps are controlled.
Hence
$E_1(Dec(\F)) = E_2(\F)=\calH^j(\calK_i)$.
   \end{proof}

We now construct an auxiliary category $\M^\#$ by a variation of the
method used in section \ref{section:tensor}.
A pair of stratifications $\Sigma$ of $S$ and
$\Lambda$ of $Q$ will be called $f$-admissible if each $\sigma\in\Sigma$
is a fibre bundle over some $\lambda\in \Lambda$.
Given a stratification $\Sigma'$ of $S$, there exists an admissible
pair $(\Sigma,\Lambda)$ for which $\Sigma$ refines $\Sigma'$. 
Any pair of stratifications can always be refined so that admissibility holds. 
Choose admissible stratifications $\Sigma$ and $\Lambda$.
Then composition gives a functor $R_B\circ \calH^*$ from
$C^{[0,\infty)}(\M(Q,\Lambda))$ to the
category $Gr\, Sh(Q((\C))$ of $[0,\infty)$-graded sheaves. We also have a functor
$H_Q^*\circ R_B:\M(S,\Sigma)\to Gr\,Sh(Q(\C))$. Let $\calC(S,\Sigma)$ be
the so called comma category whose objects are triples
$$(K^\dt, M, \phi: H_Q^*\circ R_B(M)\to R_B\circ \calH^*(K^\dt))$$
where $K^\dt\in Ob C^{[0,\infty)}(\M(Q,\Lambda))$ and $ M\in Ob\M(S,\Sigma)$.
Morphisms are pairs $K_1^\dt\to K_2^\dt$, $M_1\to M_2$ satisfying
obvious compatibilities. Let $\calC_{iso}(S,\Sigma)$ be the full subcategory
consisting of triples for which $\phi$ is an isomorphism.
We can identify $F^2\text{-mod}$ with $F\text{-mod}\times  F\text{-mod}$.
There is a faithful exact functor
$U_2:\calC(S,\Sigma)\to (F\text{-mod})^2$ given by $(K^\dt, M,\phi)\mapsto
(\prod_i U(K^i))\times U(M)$.

Given a simple cellular
 quasi-filtration $(T\to S, T_\dt)$ and a stratification $\Sigma$, choose 
base points $s$ for $(S,\Sigma)$ and $t_\dt\in T_\dt$. 
 We define a functor $H^\#:\Delta(S,\Sigma)^{op}\to \calC(S,\Sigma)_{iso}$ as follows.  On
  objects
  \begin{equation*}
    H^\#(X\to S,Y,i,w) =
  \end{equation*}
  \begin{equation}
    \label{eq:Hsharp}
(h_Q^{i}(X_{T_0},Y_{T_0}\cup X_{T_{0-1}})(w)\to h_Q^{i+1}(X_{T_1},Y_{T_1}\cup
X_{T_{1-1}})(w)\to \ldots ; h^i_S(X,Y)(w);\phi)
  \end{equation}
where the differentials of the complex are
connecting maps and $\phi$ is given by proposition \ref{prop:leray}.
We can extend this to nonsimple cellular 
 quasi-filtrations $(\{Q^\dt\}, T\to S, T_\dt^\dt)$ as follows. For notational
 simplicity, we assume that the cover consists of two sets $\{Q^0,Q^1\}$
 and that $T_\dt^0,T_\dt^1$  can be refined to $T_\dt^{01}$ on
 $Q^{01}=Q^0\cap Q^1$. Let $j_{0},
 j_1, j_{01}$ denote the inclusions of $Q^0, Q^1$ and $Q^{01}$ into
 $Q$ respectively.  Since the varieties $T_\dt^\dt$ can be extended
 over $Q$ by taking closures, the extensions by zero
 $$M^j_\alpha=j_{\alpha!}h_{Q^\alpha}^{i+j}(X_{T_j^\alpha},Y_{T_j^\alpha}\cup X_{T_{j-1}^\alpha})(w)$$
 are defined.  We can now define $ H^\#(X\to S,Y,i,w)$  by taking
 $h^i_S(X,Y)(w)$ as the second component as above.
For the first component, we use  the complex
$$\ker[M^\dt_0\oplus M^\dt_1\to M^\dt_{01}] $$
where the map is given the difference of restrictions. This complex is
quasi-isomorphic to $M^\dt_\alpha$  on $Q^\alpha$. The
quasi-isomorphism $\phi$ can be thus extended, so that $ H^\#(X\to
S,Y,i,w)\in \calC_{iso}$

We let 
  $\PM^\#(f,\{Q^\dt\}, T\to S, T_\dt^\dt,\Sigma)= End_{F^2}^\vee(H^\#\circ U_2)$-comod,
and let  $\M^\#(f, \{Q^\dt\}, T\to S ,  T_\dt^\dt,\Sigma)$ denote the associated
stack. This carries an exact faithful embedding into $F^2$-mod.
In order to simplify notation, we usually just write these as
$\PM^\#(T_\dt^\dt)$ and  $\M^\#(T_\dt^\dt)$.
We let $h_{T_\dt}(X,Y)(w)$ denote the object of this category associated
to $(X,Y,i,w)$. We have a functor $\M^\#(T_\dt^\dt) \to
\calC(S,\Sigma)_{iso}$. We can compose this with the projections to get functors $p:\PM^\#(T_\dt^\dt)\to
  \M(S,\Sigma)$ and $q:\M^\#(T_\dt^\dt)\to C^{[0,\infty)}(\M(Q,\Lambda))$.
These extend to functors on $\M^\#(T_\dt^\dt)$ denoted by the same symbols.
%From lemma \ref{lemma:surjcatstack}, we obtain
From corollary \ref{cor:surjcat}, we obtain

\begin{lemma}\label{lemma:surjcat}
  $\M(S,\Sigma)$ is equivalent to $\M^\#(T_\dt^\dt)/\ker p$.
\end{lemma}

Let 
$$ \bar q:\M^\#(T_\dt^\dt)\to D^b(\M(Q,\Lambda))$$
 denote the composition of $q$ with the canonical map.
 
\begin{lemma}\label{lemma:dirimage1}
There is a functor $r'f_*$ fitting into the commutative diagram
$$
\xymatrix{
 \M^\#(T_\dt^\dt)\ar[d]\ar[rd]^{\bar q} &  \\ 
 D^b(\M(S,\Sigma))\ar[r]^{r'f_*} & D^b(\M(Q,\Lambda))
}
$$
such that there is a natural isomorphism $R_Br'f_* \cong \R f_* R_B$.
This functor is independent of the choice of quasi-filtration.
\end{lemma}

\begin{proof}
We note that $q$ and $\bar q$ can be extended to $C^b(\M^\#(T_\dt^\dt))$ by taking the total complex
associated to the double complex induced by these functors.
 By lemma \ref{lemma:surjcat}, $\M(S)$ is equivalent to
  $\M^\#(T_\dt^\dt)/\ker p$.  This extends to an equivalence $C^b(\M(S))\sim  C^b(\M^\#(T_\dt^\dt)/\ker p)$.
From the definition of $\calC_{iso}(S)$, it follows  that 
  \begin{equation}\label{eq:hiM}
    R_B(\calH^i(q(M)))\cong H_Q^i(S, R_B(p((M)))
  \end{equation}
  as functors in $M\in C^b(\M^\#(T_\dt^\dt))$. Since $R_B$ is faithful, this isomorphism implies
  that $\calH^i\circ q$ factors through $D^b(\M^\#(T_\dt^\dt)/\ker p)$. We can summarize all of
  this  by the  commutative diagram
$$
\xymatrix{
  C^b(\M^\#(T_\dt^\dt))\ar[rr]^{\bar q}\ar[dd]^{p}\ar[rd] &  & D^b(\M(Q,\Lambda))\ar[dd]^{R_B} \\
  & D^b(\M^\#(T_\dt^\dt)/\ker p)\ar[ru]_{h}\ar[ld]_{\sim}^{e} &  \\
  D^b(\M(S,\Sigma))\ar[rr]^{\R f_*\circ R_B} & & D^b(Sh(Q))}
$$
Since $e$ is an equivalence, it has an inverse.
The desired functor $r'f_*$
would be given by $h\circ e^{-1}$, but since it depends, a priori, on the the choice of the
quasi-filtration, we temporarily denote it by $r'f_{*,Q,T_\dt}$. Given a
second cellular quasi-filtration $(T'_\dt\to S, T'_\dt)$, we wish to show
that there is a canonical isomorphism $r'f_{*,Q,T_\dt}\cong
r'f_{*,Q,T'_\dt}$. By proposition~\ref{prop:cell}, we can assume that there is a morphism $(T'_\dt\to S,
T'_\dt)\to (T_\dt\to S, T_\dt^\dt)$ in $QVar_Q$ covering the identity $id:S\to
S$.  Therefore the complexes defining $r'f_{*,Q,T_\dt}$ and $r'f_{*,Q,T'_\dt}$
become quasi-isomorphic since the constructions factor through
$\calC(id)_{iso}$.  So we can now omit the second subscript.
\end{proof}

To finish the proof of the theorem, we will verify the conditions of
lemma~\ref{lemma:adjointness}  which entails constructing adjunctions.

\begin{lemma}
Given maps  $g:X\to S$ and $f:S\to Q$ of topological spaces,
  consider a commutative diagram
$$
\xymatrix{
 X\times_Q S\ar[r]^{\pi}\ar[d]^{p} & X\ar[d]^{f\circ g}\ar@/^/[l]^{\Gamma} \\ 
 S\ar[r]^{f} & Q
}
$$
where $\Gamma$ is the inclusion of the graph of $g$. Then the
adjunction map $f^*\R f_* \R g_*\to \R  g_*$ is the composition of
the base change map
$$f^*\R (f\circ g)_*\to \R p_*\pi^*$$
and the adjunction
$$\R p_* \pi^* \to \R p_*\R \Gamma_* \Gamma^*\pi^* = \R  g_*$$
%A similar statement holds for pairs $(g:X\to Q,Y)$.
\end{lemma}

\begin{proof}
  This follows by applying \cite[prop 3.7.2ii]{lipman} to the diagram
$$
\xymatrix{
 X\ar[r]^{id}\ar[d]^{\Gamma} & X\ar[d]^{id} \\ 
 X\times_Q S\ar[r]^{\pi}\ar[d]^{p} & X\ar[d]^{f\circ g} \\ 
 S\ar[r]^{f} & Q
}
$$
\end{proof}

\begin{cor}
  If the first base change map is an isomorphism, $f^*\R f_* \R g_*\to \R
  g_*$ can be identified with the map $\R p_* \pi^* \to \R  g_*$
  induced by $\Gamma$.
\end{cor}

\begin{lemma}\label{lemma:frfto1}
There is a morphism $\epsilon:f^*r'f_*\to 1$   compatible with the
adjunction  $\epsilon:f^*\R f_*\to 1$.
\end{lemma}

\begin{proof}
The previous corollary implies that
$$\Gamma^*:H^i_S(X\times_Q S,Y\times_Q S)\to H^i_S(X,Y)$$
is precisely the adjunction map $\epsilon$. We have to lift this to a morphism $\epsilon:f^*r'f_*\to 1$. 
Let $\Sigma$ and $\Lambda$ be an $f$-admissible pair of stratifications. 
Choose a quasi-filtration $T_\dt \to S$. 
Define a functor 
$$H^{\spadesuit}:\Delta(S,\Sigma)\to F\text{-mod}^3$$
by sending  $(X\to S, Y, i, w)$ to the direct sum of the three vertices of
the  diagram
$$
\xymatrix{
\mathcal{H}^0( H_S^{i}(X_{T_0}\times_Q S,Y_{T_0}\times_QS \cup
X_{T_{0-1}}\times_QS)\to\ldots )& H^i_S(X\times_Q
S,Y\times_Q S)\ar_<<<{\sim}[l]\ar^{\Gamma^*}[d]\\
& H^i_S(X,Y)
}
$$
%<<
We build a category $\M^\spadesuit(S)=
End{F^3}^\vee(H^\spadesuit)\text{-comod}$.  From the universal property, we
have a functor $h^\spadesuit$ which assigns to $(X\to S, Y, i, w)$ the diagram
of motives
$$
\xymatrix{
\mathcal{H}^0( h_S^{i}(X_{T_0}\times_Q S,Y_{T_0}\times_QS \cup
X_{T_{0-1}}\times_QS)\to\ldots )& h^i_S(X\times_Q
S,Y\times_Q S)\ar_<<<{\sim}[l]\ar^{\Gamma^*}[d]\\
& h^i_S(X,Y)
}
$$
%<<<<<<<
In particular, we obtain  a projection  $\M^\spadesuit(S)\to  \M(S)$. 
As above one can argue that
 $\M^\spadesuit(S)/\ker(p)\sim \M(S)$. We can see that $h^\spadesuit$
 factors through $\ker p$. This determines  a functor
$\epsilon:\M(S)\to Mor\M(S)$ compatible with adjunction. 
\end{proof}

\begin{lemma}
  There is a morphism  $\eta:1\to r' f_* f^*$ compatible with the
  adjunction $\eta:1\to \R f_* f^*$ 
\end{lemma}

\begin{proof}
The strategy is similar to the previous argument.
Let $\Sigma$ and $\Lambda$ and $T_\dt \to S$  be as in the above argument.
Define 
$$H^{\clubsuit}:\Delta(Q,\Lambda)\to F\text{-mod}^2$$
by sending $(X\to Q, Y, i, w)$
to the sum of vertices of the diagram
$$
\xymatrix{
 H_Q^{i}(X,Y)\ar[d] & &\\
 H_Q^{i}(X_{T_0},Y_{T_0}\cup X_{T_{0-1}})\ar[r] &
 H_Q^{i+1}(X_{T_1},Y_{T_1}\cup X_{T_{1-1}})\ar[r] &\ldots  
}
$$
partitioned so that $ H_Q^{i}(X,Y)$ corresponds to the first
component.
Let $\M^\clubsuit(S)=
End_{F^3}^\vee(H^\clubsuit)\text{-comod}$. As above, we have a functor
$h^\clubsuit$ sending  $(X\to Q, Y, i, w)$ to
$$
\xymatrix{
 h_Q^{i}(X,Y)\ar[d] & &\\
 h_Q^{i}(X_{T_0},Y_{T_0}\cup X_{T_{0-1}})\ar[r] &
 h_Q^{i+1}(X_{T_1},Y_{T_1}\cup X_{T_{1-1}})\ar[r] &\ldots  
}
$$
This yields the map $\eta$ once we observe that
 $\M^\clubsuit(S)/\ker(p)\sim \M(S)$ where $p$ is the natural
 projection. 
\end{proof}

\begin{proof}[Proof of theorem]
  To finish the proof of the theorem, it is enough to observe that by
  the previous lemmas, we can apply lemma~\ref{lemma:adjointness} to
  conclude that $rf_*=r'f_*$.
\end{proof}

\subsection{Direct image with compact support}

\begin{thm}\label{thm:dirimage}
  If $f:S\to Q$ is an morphism of quasiprojective varieties, 
then there is a functor $h^i_{c,Q}=r^if_!:\M_c(S;F)\to
  \M(Q;F)$, such that $R_B(r^if_!(M)) \cong R^if_!(S, R_B(M))$.  If $g:S'\to
  S$ is a morphism, there is an isomorphism $g^*r^if_!(M)\cong r^if_!(g^*M)$
  compatible with the base change isomorphism on realizations.
\end{thm}

\begin{proof}
Choose a relative compactification 
$$
\xymatrix{
 S\ar[r]^{j}\ar[d]^{f} & \bar S\ar[ld]^{\bar f} \\ 
 Q & 
}
$$
Then we can define $r^if_! = (r\bar f^i_*)\circ (j_!)$,
provided that we can show that it is well defined. First observe that  we have 
$$R_B(r\bar f^i_*\circ j_!M) =R\bar f^i_*\circ j_! R_B(M) =R^if_!R_B(M)$$
as required.
To see that it is independent of the choice,
first observe that given a second compactification $S\to \tilde S$, we can find
a third  compactification $\overline{\overline{S}}$ which dominates both $\bar S$ and $\tilde S$. For example,
we can take  $\overline{\overline{S}}$ equal the closure of diagonal in $\bar S\times_Q \tilde S$.
Thus we can assume that $\overline{\overline{S}}=\tilde S$, and so we  can assume that
we have a commutative diagram
$$
\xymatrix{
 S\ar[r]^{j}\ar[dd]^{f}\ar[rd]_{\tilde j} & \bar S\ar[ldd]^{\bar f} \\ 
  & \tilde S\ar[u]_{\pi}\ar[ld]^{\tilde f} \\ 
 Q & 
}
$$
Therefore, we get a morphism
$$r^i\bar f_*j_!M\to r^i\tilde f \pi^*j_!M\cong r^i\tilde f\tilde j_!M$$
which induces identity on realization. So it must be an isomorphism in
$\M(Q)$.

\end{proof}

\begin{remark}
 The above construction can be lifted to a functor on  derived categories,
$rf_!: D^b\M_c(S;F)\to D^b\M(Q;F)$, by defining
$rf_! = (r\bar f_*)\circ (j_!)$ using the notation of the proof.
\end{remark}

\section{Motivic Local Systems}

\subsection{Local Systems}\label{section:localsystems}

Call an object of $\Delta(S)$  {\em tame} if it is of the form
  $(\bar X-D\to S, E\cap (\bar X-D), i, w)$ such that $\bar X\to S$ is smooth
and projective, and $D+E$ is a  divisor  such
that any intersection of components is smooth over $S$ (we will refer to this condition
as having relative normal crossings). We usually just write $E$
instead of $E\cap (\bar X-D)$ above.
Such a pair $(\bar X-D\to S, E)$ is a fibre bundle,
so it is controlled.

It is easy to see
that $\Delta_{tame}(S)\subset \Delta_{eq}(S,\{S\})$. 

\begin{defn}
    The category of   premotivic local systems $\PM_{ls}(S;F)= $\\
    $End^\vee(H|_{\Delta_{tame}(S)})\text{-comod}$.
The category of motivic local systems $\M_{ls}(S;F)$ is obtained by
forming the associated stack as in section \ref{sect:extmotivicsheaves}.
\end{defn}

We note the following properties which are either
 immediate consequences of what has been said or easily checked.
 
 \begin{enumerate}
\item $\M_{ls}(S)\subset \M_{eq}(S,\{S\})$ is an abelian subcategory.
\item The realizations $R_B$ and $R_{et}$ take $\M_{ls}(S)$ to the categories
of locally constant sheaves for the classical and \'etale topologies, and they
factor through $\M_{ls}(S)$.
\item The tensor product given earlier restricts to a product
$\M_{ls}(S)\times \M_{ls}(S)\to \M_{ls}(S)$. (The key point is that
$\M_{ls}$ is equivalent to comodules over the restriction of $End^\vee$ to
$\Delta_{cell}\cap \Delta_{tame}$.) This induces a product on the stacks
$\M_{ls}(S)\times \M_{ls}(S)\to \M_{ls}(S)$
\end{enumerate}

By item 2 above, we see that $\M_{ls}(S)$ is strictly contained in  $\M(S)$ in
general. However, we do observe the following:

\begin{thm}
When $S=Spec\, k$,
  $\M(S; F)$ and $\M_{ls}(S;F)$ are equivalent.
\end{thm}

\begin{proof}
By theorem~\ref{thm:cellveq}, $\M(Spec\,k; F)$ is equivalent to
$\M_{cell}(S;F)$.  Given  $(X,Y,i,w)\in \Delta_{cell}(S)$, by
resolution of singularities we can find  a tame object such that
$(\tilde X, E,i,w)$ and a map $\pi:\tilde X\to X$ which is an isomorphism over $X-Y$
and such that $E=\pi^{-1}Y$. Therefore $h^i(X,Y)(w)\cong h^i(\tilde X,E)(w)\in
\M_{ls}$. Again by resolution of singularities, any morphism in
$\Delta_{cell}$ can be lifted to a morphism in $\Delta_{tame}$.
This together with lemma~\ref{lemma:presentation} implies the theorem.
\end{proof}

We outline the construction of Gysin maps, which will be needed later.
Given a smooth subscheme $\bar Y\to S$ 
of $\bar X$ transverse to $D+E$ with relative dimension $m$. Set $c= n-m$. Then
the Gysin homomorphism on cohomology
$$H^i_S(\bar Y-D,E)\to H^{i+2c}_S(\bar X-D,E)$$
can be defined simply by dualizing the restriction under Poincar\'e duality. 
However, this description is not very convenient.  A better alternative is to define this
via  a deformation to the normal bundle as in \cite{bfm}. Let $\tilde X$ be the blow up of
$\bar X\times \mathbb{A}^1$ along $\bar Y\times \{0\}$. Let $\tilde Y$ be the strict transform
of $\bar Y\times \mathbb{A}^1$.  Let $\tilde D,\tilde E$ be the preimages of $D,E$ in
$\tilde X$. The fibre of the natural map $\pi:\tilde X\to \mathbb{A}^1$ over $t\not=0$ is
$X$. While the fibre  $\pi^{-1}0$ is
the union of  the projectivized normal bundle $p:\PP(N\oplus \mathcal{O}_{\bar Y})\to \bar Y$  and the blow up
$B$ of $\bar X$ along $\bar Y$.  Let $\tau=c_1(\mathcal{O}_{\PP(N\oplus \mathcal{O})}(1))^c \in H^{2c}(\PP(N\oplus \mathcal{O}),\PP(N))$. The  Gysin map can then be realized as the composition of
the given maps
\begin{eqnarray*}
H^i_S(\bar Y-D, E) &\stackrel{p^*}{\longrightarrow} & H_S^i(p^{-1}Y-p^{-1}D, p^{-1}E)\\
 &\stackrel{\cup \tau}{\longrightarrow} & H_S^{i+2c}(p^{-1}Y-p^{-1}D, p^{-1}E)\\
 &\stackrel{\cong}{\longleftarrow} & H_S^{i+2c}(\pi^{-1}(0)-\tilde D, \pi^{-1}(0)\cap \tilde E\cup B)\\
 &\stackrel{\cong}{\longleftarrow} & H_S^{i+2c}(\tilde X-\tilde D, \tilde E\cup B)\\
  &\stackrel{\pi_t^* }{\longrightarrow} & H_S^{i+2c}(\bar X-D,  E)
\end{eqnarray*}
The second description yields a motivic Gysin map
\begin{equation}
  \label{eq:gysin}
h^i_S(\bar Y-D,E)\to h^{i+2c}_S(\bar X-D,E)(c)  
\end{equation}
We can define the Gysin morphism 
$$h^i_S(\bar Y-f^{-1}D,f^{-1}E)\to h^{i+2c}_S(\bar X-D,E)(c)$$
for an arbitrary map $f:\bar Y\to \bar X$ as the composition of the Gysin
 morphism associated to the inclusion of  graph of $\Gamma_f\subset \bar Y\times \bar X$ 
 followed by a K\"unneth  projection.

When, $S$ is smooth let $VMHS(S_{\iota,an})$ 
denote the category of rational variations
of mixed Hodge structures on $S_{\iota,an}$, 
which are admissible in the sense of 
Steenbrink and Zucker \cite{sz} and  Kashiwara \cite{kashiwara}.
In a nutshell, an object of this category consists of a filtered local system $(V,W)$
together with a compatible bifiltered vector bundle with connection
$(\V\cong V\otimes \mathcal{O}_S,W,F,\nabla)$ subject to the appropriate axioms (Griffith's transversality...).  For the precise conditions, see \cite[sect. 14.4.1]{ps} or the above references.
Given  $(X=\bar X-D\to S, E, i, 0)\in \Delta_{tame}$,
we can construct an admissible variation as follows:
$$
\begin{cases}
V = H_S^i(X,E\cap X;\Q) &\\
 \V = \R f_*\Omega^\dt_{\bar X/S}(\log D+E)(-E) &\\
 F^p = \im \R f_*\Omega^{\ge p}_{\bar X/S}(\log D+E)(-E) &\\
W_q = \im \R f_* W_q\Omega^{\dt}_{\bar X/S}(\log D+E)(-E) &\\
\nabla=\text{ Gauss-Manin connection} 
\end{cases}
$$
This is given in \cite{sz}, when $E=\emptyset$. The general case
is easily reduced to this via the resolution
$$j_{X, E!}\Q_{X-E}\to \Q_X\to \bigoplus \Q_{ E_i}\to  \bigoplus \Q_{ E_i\cap E_j}\ldots$$
where $E=\cup E_i$ is the decomposition into irreducible components.
We can extend this to arbitrary objects $(X=\bar X-D\to S, E, i, w)\in \Delta_{tame}$
by tensoring the above variation  with $\Q(w)$.
This construction is easily checked to yield a functor 
$\Delta_{tame}(S)^{op}\to VMHS(S)$. Thus we get

\begin{ex}
  an exact faithful Hodge realization functor 
$$R_{\iota,H}=R_H: \M_{ls}(S; \Q)\to VMHS( S_{\iota,an})$$
This functor is compatible with  tensor product. This coincides with
the Hodge realization constructed earlier,  restricted $\M_{ls}(S;
\Q)$, once we identify $VMHS(S)\subset Cons\text{-}MHM(S)$.
\end{ex}

One of the consequences of the admissibility conditions mentioned
above is the following 
removable singularities theorem:
An admissible variation extends from a Zariski open to the whole
variety if the underlying local system extends.
Using this, it is possible to prove a stronger statement that
$R_H$ extends to all of $\M_{ls}(S)$.  

We can define a system of realizations on $S$ by following the usual
pattern \cite{deligne-3, jannsen-mm}.  Here we outline the construction.
A ``locally  constant'' or more correctly lisse $\ell$-adic sheaf $V$ on $S_{et}$ corresponds to a representation
of the algebraic fundamental group $\pi_1^{et}(S)\to GL_N(\Q_\ell)$.
Composing this with the canonical map  from
the topological fundamental group $\kappa:\pi_1(S_{\iota,an})\to \pi_1^{et}(S)$
results in a local system $\kappa_\iota^*V$ of $\Q_\ell$-modules on $S_{\iota, an}$.
By a system of realizations we will mean 
\begin{enumerate}
\item  A collection of locally 
constant $\ell$-adic sheaves $V_\ell$ on $S_{et}$, for each prime $\ell$.
Each $V_\ell$ should be mixed in the sense that they carry weight filtrations.
\item A collection of variations of mixed Hodge structures $V_\iota$ on $S_{\iota, an}$
indexed by embeddings of $\iota:k \hookrightarrow  \C$.
\item Compatibility isomorphisms $\kappa_\iota^*V_\ell\cong V_\iota\otimes \Q_\ell$
respecting weight filtrations.
\end{enumerate}

These form a $\Q$-linear abelian category $SR(S)$.  An appeal to corollary \ref{cor:tannaka}
and the comparison theorem (appendix B) yields a realization functor $R_{SR}:\M_{ls}(S,\Q)\to SR(S)$
which combine all of the previous realizations into one. Thus $\M_{ls}$ gives a finer theory
than motives built from systems of realizations.

\subsection{Duality}

The goal  of this section is to prove:
           
\begin{thm}
 $\M_{ls}(S;F)$ is a neutral Tannakian
category over $F$
\end{thm}

\begin{cor}
  $\M_{ls}(S,\Q)$ is equivalent to the category of representations of a
  proalgebraic group (which we refer to as the Tannakian dual of this category).
\end{cor}

To be more explicit, after choosing a base point $s\in S(\bar k)$, we
obtain a so  called fibre functor $F_s:\M_{ls}(S,\Q)\to \Q\text{-mod}$
given as the composition of $R_B$ with the stalk at $s$. Setting
$\pi_1^{mot}(X,s)$ to the  group of tensor automorphisms 
of $F_s$, we have that $\pi_1^{mot}(X,s)$ is proalgebbraic and that  $\M_{ls}(S,\Q)$ is equivalent to the category of
representations of it. The methods of \cite{arapura2} show that this
carries more structure, but the details will be spelled out elsewhere.

As to the theorem's proof, we 
 know that  $\M_{ls}(S;F)$  is a tensor category over $F$ with a tensor preserving
fibre functor. What remains to be proven is that every object has a dual. By proposition \ref{prop:tannaka3}
it is enough to construct duals for objects of the graph $\Delta_{tame}(S)$. We will  show that
\begin{equation}\label{eq:dualXDE}
h_S^i(\bar X-D, E)( w)^\vee = h_S^{2n-i}(\bar X-E, D)( -w+n)\tag{Dual}
\end{equation}
 where $n$ is the relative dimension of $\bar X\to S$.
 As first step, we note the following form of  Poincar\'e duality.

\begin{lemma}
There is a  pairing 
$$H_S^i(\bar X-D,E)\otimes  H_S^{2n-i}(\bar X-E,D)\to F_S$$
which is perfect in the sense that it induces an isomorphism
of local systems
$$H_S^i(\bar X-D,E)\cong H_S^{2n-i}(\bar X-E,D)^*$$
\end{lemma}
  
   \begin{proof}
This follows from Verdier duality \cite{iversen}
$$H_S^i(\bar X,  \R j_{(\bar X, D)*} j_{(\bar X, E)!} F) \cong
H_S^{-i}(\bar X,  D \R j_{(\bar X, D)*} j_{(\bar X, E)!} F )^*$$
$$ \cong
H_S^{-i}( j_{(\bar X,D)!}\R j_{(\bar X,E)*}F[2n])^*\cong 
H_S^{2n-i}(\bar X-E,D)^*$$
\end{proof}

The next task is to realize the above pairing  geometrically
by  a morphism of $\M_{ls}$.
When $D=E=\emptyset$, we can take the cup product pairing which is induced by the
diagonal embedding
into the product. In general, we need to blow up the product to get a well defined diagonal. 
Set $Y= \bar X\times \bar X$, $D_1= D\times \bar X$, $D_2= \bar X\times D$,
$E_1= E\times \bar X$ and $E_2= \bar X\times E$.  Let $\tilde Y$ be
obtained by blowing up $Y$ along $D_1\cap D_2$ and
then along the intersection of the strict transforms of $E_1$ and $E_2$.
Let $G$ be the exceptional divisor of $\tilde Y\to Y$.
Denote the strict transforms of $D_i, E_j$ by $\tilde D_i, \tilde E_j$.
The diagonal
embedding $\bar X\to Y$ extends to an embedding of $d:\bar  X\to \tilde Y$ (it is not
necessary to blow up $X$ since $D$ and $E$ are already divisors).
The image of $d$ is disjoint from $\tilde D_i, \tilde E_j$  and $d^{-1}G\subseteq D\cup E$.
We define 
$$\epsilon:h_S^i(\bar X-D,E)\otimes  h_S^{2n-i}(\bar X-E,D)(n)\to F_S $$
by the composition of
\begin{eqnarray*}
h_S^i(\bar X-D,E)\otimes  h_S^{2n-i}(\bar X-E,D)&\to& h_S^{2n}(Y-(D_1\cup E_2), E_1\cup D_2)\\
&\to&  h_S^{2n}(\tilde Y-(\tilde D_1\cup \tilde E_2\cup G), \tilde E_1\cup \tilde D_2)\\
&\stackrel{\cong}{\leftarrow}& 
 h_S^{2n}(\tilde Y-(\tilde D_1\cup \tilde E_2), \tilde E_1\cup \tilde D_2\cup G)\\
 &\to& h^{2n}_S(\bar X, E\cup D) \\
 &\stackrel{\cong}{\to}& h_S^{2n}(\bar X)\cong  F(-n)
\end{eqnarray*}
after twisting by $F(n)$.  The middle isomorphism is excision.
For  the last isomorphism, by projection  we can reduce to the
case $X=\PP^n_S$ and then to $X=(\PP_S^1)^n$, where it follows from K\"unneth.

To construct $\delta$, we dualize the above description using  Gysin maps in place
of pull backs:
\begin{eqnarray*}
F_S= h^0_S(\bar X, E\cup D) &\to &  
h_S^{2n}(\tilde Y-(\tilde D_1\cup \tilde E_2), \tilde E_1\cup \tilde D_2\cup G)(n)\\
&\stackrel{\cong}{\leftarrow}& 
 h_S^{2n}(\tilde Y-(\tilde D_1\cup \tilde E_2\cup G), \tilde E_1\cup \tilde D_2)(n)\\
&\to & h_S^{2n}(Y-(D_1\cup E_2), E_1\cup D_2)(n)\\
&\to& h_S^i(\bar X-D,E)\otimes  h_S^{2n-i}(\bar X-E,D)(n)
\end{eqnarray*}

To prove \eqref{eq:dualXDE}, we have to establish equations \eqref{eq:dual1} and \eqref{eq:dual2}.
It is enough to verify  these on the corresponding
vector spaces $H_S^i(\bar X-D,E)_s, H_S^{2n-i}(\bar X-E,D)_s$, 
and this becomes an exercise in linear algebra.
If $e_j$ is a basis of the first space, and $e^j$ the dual basis of the second, then
$$\delta(1) = \sum_\ell e^\ell\otimes e_\ell$$
$$\epsilon(\sum a_{j\ell} e_j\otimes e^\ell) = \sum a_{jj}$$
Therefore
$$(\epsilon\otimes id)\circ (id\otimes \delta)(\sum a_je_j)
= (\epsilon\otimes id)(\sum_{j\ell} a_j e_j\otimes e^\ell\otimes e_\ell)
=  \sum a_\ell e_\ell$$
proves \eqref{eq:dual1}.
The remaining equation is similar.

\subsection{Pure Objects and Weights}

We work in $\M(S,\Q)$ throughout this section.
Let $f:X\to S$ be a smooth projective map of relative dimension $n$.
 Fix an embedding $X\subset \PP^N_S$. The standard generator  
 $c_1({\mathcal O}(1))\in H^2(\PP^N)$ induces an isomorphism
 $\Q_S(0)\cong h_S^2(\PP^N_S)(1)$. This yields a map $\Q_S(0)\to h^2_S(X)(1)$ by restriction.
 Cupping with this induces the Lefschetz operator $\ell:h_S^i(X)\to
h_S^{i+2}(X)(1)$. The isomorphism
$$\ell^{i}:h^{n-i}_S(X)\stackrel{\sim}{\longrightarrow} h^{n+i}_S(X)(i)$$
follows from the usual hard Lefschetz theorem on the corresponding
sheaves. Therefore we get, as usual, the Lefschetz decomposition
$h_S^i(X) = \oplus  \ell^k p^{i-2k}(X)(-k)$, where $p^i(X) = h_S^i(X)\cap
\ker \ell^{n-i+1}$. This allows us to define the Hodge involution $*=*_H$
on $h_S^*(X)=\oplus h^i_S(X)$ by the formula in \cite[pp 10-11]{andre}. Note that
  the induced involution on the cohomology of a fiber $H^*(X_s,\C)$  coincides with the
Hodge star operator with  respect to the Fubini-Study metric (up to a factor and 
complex conjugation) [loc. cit.]

\begin{prop}
  The algebra $End(h_S^*(X))$ is semisimple.
\end{prop}

\begin{proof}
Set $a' = *a^t*$, where $a^t$ is the transpose
(c.f. \cite[1.3]{kleiman}). 
With the help of  the Hodge index theorem, we see that  the
  bilinear form $trace(ab')$ is positive definite (compare \cite[p. 381]{kleiman}). 
  Then the  criterion of  \cite[3.13]{kleiman} shows that the algebra
  is semisimple.
\end{proof}

\begin{defn}
Call an object of $\M_{ls}(S)$  pure (of weight $i$) if it  is a finite sum of
summands of motives $h^*_S(X)$ (or $h_S^i(X)$)  with $X\to S$ smooth and projective.
Let $\M_{pure}(S)\subset \M_{ls}(S)$ ($\M_{pure,i}(S)\subset
\M_{ls}(S)$) be the full subcategory of
pure objects.  
\end{defn}

\begin{thm}
  $\M_{pure}(S,\Q)$ and $\M_{pure,i}(S,\Q)$  are semisimple abelian
    subcategories of $\M_{ls}(S,\Q)$.  The Hodge realization $R_H$
    takes $\M_{pure}(S,\Q)$ (respectively $\M_{pure,i}(S,\Q)$).
to the category of pure polarizable variations of Hodge structure
$HS(S)$ (of weight $i$ respectively)
 There is a direct sum decomposition $\M_{pure}(S,\Q)=\bigoplus_i
 \M_{pure,i}(S,\Q)$, i.e. every object and morphism on the left decomposes into a
 sum as indicated.  Furthermore  $\M_{pure}(S,\Q)$ is a Tannakian subcategory.
\end{thm}

\begin{proof}
These  are abelian and semisimple  by  \cite[lemma 2]{jannsen} and the previous
proposition. The second statement is clear. The third statement
follows immediately from the previous two.
It is easy to see from the constructions that  $\M_{pure}(S,\Q)$ is closed under tensor product and duals.
\end{proof}

\begin{cor}
  The Tannakian dual of $\M_{pure}(S,\Q)$ is proreductive.
\end{cor}

\begin{thm}\label{thm:MtoMpure}
  There  are exact  functors $gr_j:\M_{ls}(S,\Q)\to \M_{pure,j}(S,\Q)$ which splits the
  inclusions $\M_{pure,j}(S,\Q)\subset\M_{ls}(S,\Q)$. These are compatible with the
   Hodge realizations in the sense that $R_Hgr_j= 
  Gr^W_j R_H$.
\end{thm}

\begin{proof}
It is enough to define $gr=\bigoplus_j gr_j$, and then set $gr_j$ to the
composition of this with the projection $\M_{pure}(S)\to
\M_{pure,j}(S)$.

Fix a smooth projective map $\bar X\to S$ with a relative normal
crossing divisor $D+E$. 
Then $H^*_S(\bar X-D,E,\C)$ can be computed by taking the direct image of
the  double complex
$$\Omega^\dt_{\bar X/S}(\log D)\to \bigoplus_i \Omega^\dt_{ E_i/S}(\log
D)\to \bigoplus_{i<j} \Omega^\dt_{ E_i\cap E_j/S}(\log D)\to \ldots $$
When restricted to  the fibres, this complex forms part of a differential graded cohomological
mixed complex \cite{deligne-hodge}. From this we can deduce a spectral
sequence  associated to the diagonal filtration (c.f. \cite[7.1.6, 8.1.19.1]{deligne-hodge})
\begin{equation}
    \label{eq:descentSS1}
    E_1^{-a,b}= \bigoplus_{p+2r=b,q-r=-a} H_S^p(Y_q^{(r)})\Rightarrow
 H_S^{b-a}(\bar X-D, E,\Q)
  \end{equation} 
where 
$$
Y_q^{(r)}=
\begin{cases}
 D^{(r)} &\text{if $q=0$}\\
E^{(q-1)}\cap D^{(r)} &\text{if $q>0$} 
\end{cases}
$$
and $D^{(r)}$ and $E^{(r)}$ are disjoint unions of $r+1$-fold
intersections of components of $D$ and $E$.
We see also that \eqref{eq:descentSS1} degenerates at $E_2$, and the induced filtration on the abutment
is  the weight filtration. For later use, we record the precise
formula for $ W_{i+k}H_S^i(\bar X-D,E,\C)$
\begin{equation}
  \label{eq:WHiDE}
\im H^i(W_k \Omega_{\bar X/S}^\dt(\log D)\to \bigoplus W_{k+1}
  \Omega_{E_i/S}^\dt(\log D)\to \ldots)
\end{equation}

The differentials of \eqref{eq:descentSS1} are sums of restrictions
and Gysin maps. So we can regard this as a spectral sequence
of variations of Hodge structures
\begin{equation}
  \label{eq:HodgedescentSS1}
  E_1^{-a,b}= \bigoplus_{p+2r=b,q-r=-a} H_S^p(Y_q^{(r)})(-r) \Rightarrow
 Gr^WH_S^{b-a}(X,\Q)
\end{equation}

As noted above, the differentials are sums of restrictions and Gysin
maps. The motivic versions of Gysin maps were defined in
\eqref{eq:gysin} of section~\ref{section:localsystems}.
Thus we can form a graded complex of motives in $\M_{pure}$
$$e^{-a,b}=\bigoplus_{p+2r=b,q-r=-a} h_S^p(Y_q^{(r)})(-r) $$
which maps to the left side of \eqref{eq:HodgedescentSS1} under
$R_H$.

We  would like to take $gr( h^i(\bar X-D,E))$ to be  the sum $\oplus
h^{-a}(e^{\dt, i+\dt})$, but at the moment this not well defined.
It depends on the choice of compactification,
 so we denote it by $G^i(\bar X,D,E)$.
We build a graph $\tilde \Delta(S)$, with a forgetful functor
$\pi:\tilde\Delta(S)\to \Delta(S)$, whose objects consist of such
compactifications, along with labels $i,w\in \Z$.
Any two compactifications are dominated by a third. Therefore
the fibres of $\pi$ are connected.  From corollary
\ref{cor:tildeDelta}, it follows that $End^\vee(H\circ\pi)\cong End^\vee(H)$.
Let $C$ be the category of triples $(A,B,\phi)$, where $A\in \M_{ls}(S)$,
$B\in \M_{pure}(S)$, and $\phi:Gr^W_*R_HA\cong R_HB$. There is an
exact faithful functor $U_2:C\to \Q^2\text{-mod}$ taking $(A,B,\phi)$
to $U(A)\times U(B)$, where $U:\M_{ls}(S)\to\Q\text{-mod}$ is the
forgetful functor.
We have a functor from $H^\#:\tilde\Delta(S)\to C$ sending a labelled compactification
$(\bar X,i,w)$ to $(h^i(\bar X-D,E)(w), G^i(\bar X,D,E)(w),\phi)$, where $\phi$ is the
natural isomorphism of Hodge realizations from \eqref{eq:HodgedescentSS1}. We can form the category
$\PM^\#$ of comodules over $End^\vee(U_2\circ H^\#)$. This has a
natural projection $p:\PM^\#\to End^\vee(H\circ \pi)\text{-comod}$.
There is an equivalence $\PM^\#/\ker p\sim \PM(S)$.
We have a functor $G:\PM^\#\to \M_{pure}(S)$ which factors through this
equivalence, and this yields $gr$.
\end{proof}

This leads to a theory of weights in $\M_{ls}$.
Let us say that an object $M\in \M_{ls}(S,\Q)$ has weight(s) in $ I\subset \Z$ if
$gr_jM=0$ for $j\notin I$. For $M\in
\M_{ls}(S,\Q)$. Define $W_kM$ to be the maximal subobject of $M$ with
weights $\le k$. Note that this exists because $\M_{ls}(S,\Q)$ embeds
into $\Q\text{-mod}$, so it is noetherian.

\begin{thm}
For all $k$, one has
\begin{enumerate}
\item $W_kM/W_{k-1}M\cong gr_kM$,
\item $W_k$ is strictly preserved by morphisms,
\item $R_H(W_kM)= W_kR_H(M)$.
\end{enumerate}  
\end{thm}

We first need:

\begin{prop}
  Given $M\in \M(S)$ and $j\in \Z$, there exists $N\subseteq M$ and $N'\subseteq M$ so that
  $N'$ has weights $< j$ and $N/N'\cong gr_jM$.
\end{prop}

  \begin{proof}
  By lemma~\ref{lemma:presentation}, we can assume that $M=h^i_S(\bar
  X-D,E)$ with $\bar X$ smooth, and $D+E$ a divisor with relative
  normal crossings.  In principle, the proof amounts to realizing the
  formula for $W$ given in \eqref{eq:WHiDE} by a motive. When $E=0$, this is easy to do
  directly. Let $D_{(r)}$ denote union of $(\dim (D/S)-r+1)$-fold
  intersections of components of $D$, with $D_{(-1)}=\emptyset$. The
  point is  that  $\dim (D_{(r)}/S)=r$. Then \eqref{eq:WHiDE}  reduces to
$$W_{i+k} H^i(\bar X-D,\C)=\im H^i( W_k\Omega_X^\dt(\log D)) =\im
H^i(X-D_{(k)},\C)$$
Therefore  $N$ (respectively $N'$) may be taken as
$$
 \im [h^i(\bar X-D_{k-1})\to h^i(\bar X-D)]
$$
with $k=j-i$ (respectively $k=j-i-1$). 

The general case, while feasible, is rather messy to write explicitly.
So instead, we finish the proof by induction on
$d=\dim (\bar X/S)$. Once we have
  established this for a given $d$, it follows the proposition holds
  for all motives generated as an abelian category by 
  varieties of dimension at most $d$.  So now consider the sequence
$$h^{i-1}(E-D)\to h^i(\bar X-D,E-D)\to h^i(\bar X-D)$$
By induction, we can find  $N_{E},N_{E}'\subset h^{i-1}(E-D)$
satisfying the proposition for this motive . Then 
$$N= \im N_{E}+ \im [h^i(\bar X-D_{(j-i-1)},E)\to h^i(\bar X-D,E)]$$
$$N'= \im N_{E}'+ \im [h^i(\bar X-D_{(j-i-2)},E)\to h^i(\bar X-D,E)]$$
will satisfy the proposition for $M$.

 \end{proof}

 \begin{proof}[Proof of theorem]
Let $k$ be the least weight of $M$. The canonical map $\iota:gr_k
(W_kM)\to gr_kM$ is a monomorphism, since $gr_k$ is exact.
   The previous proposition shows that $\iota$ is also an epimorphism
   and hence an isomorphism.
Applying the same argument to the quotients $M/W_jM$ establishes part
1. The filtration $W_k$ is functorial by construction. Strictness
follows by what has just been proved (cf \cite[\S1]{deligne-hodge}).
Finally part 3 follows immediately from theorem~\ref{thm:MtoMpure}.
 \end{proof}

From the construction, we can deduce the following:

\begin{prop}
The total functor
  $gr:\M_{ls}(S)\to \M_{pure}(S)$ is an exact tensor functor. 
\end{prop}

\begin{cor}
  The Tannkian dual of $\M_{ls}(S)$ is a semidirect product of the
  Tannakian dual of $\M_{pure}(S)$ with another group.
\end{cor}

\subsection{Andr\'e's category of motives}

Andr\'e  \cite{andre} has given an entirely different construction of pure motives
over a field  $k$ that we recall. Given a smooth projective variety
$X\in Var_k$, a
class in $H^{2n}(X,\Q)$ is called {\em motivated} cycle of degree $n$ if it can be
expressed as $p_*(\alpha\cup * \beta)$, where $\alpha,\beta$ are
algebraic cycles on a product $X\times Y$, with $Y$ smooth and
projective, and $p:X\times Y\to X$ is the projection. Let $A_{mot}^n(X)$ denote the
set of these classes. It contains the space of algebraic cycles and
would coincide with it assuming Grothendieck's standard conjectures.
 Andr\'e's category  of motives $M_A(k)$ is built by taking as objects triples $(X,n,p)$
with $X$ smooth projective, $n\in \Z$, and $p$ an idempotent in the
ring of motivated cycles on $X\times X$.
Morphisms are given by
$$Hom_{M_A}((X,n,p), (Y,m,q))
= pA_{mot}^{n-m}(X\times Y)q$$
Andr\'e proved that this category is semisimple Tannakian. The 
construction of  $M_A$ and this result was extended to more general
smooth bases in \cite{ad}. For simplicity, we concentrate on the case
of $S=Spec\, k$,
 Given a smooth projective variety $X$,
let $h_A(X)\in X$ denote the object represented by the triple
$(X,0,id)$ and $h_A^i(X)$ by $(X,n,\pi_i)$, where $\pi_i:H^{2\dim
  X-i}(X)\otimes H^{i} (X)$ is K\"unneth component of the diagonal.
Then $h_A(X)=\oplus h^i_A(X)$.
By construction, we have an exact faithful
embedding $R_B:M_A(k)\to \Q\text{-mod}$, sending $(X,n,p)$ to
$pH^*(X)$. In particular, $R_B(h^i(X)) = H^i(X)$.
The category $M_A$ can be also be constructed by first forming the
category $MotCor$ of smooth projective varieties and motivated
correspondences, and then taking the pseudo-abelian (or idempotent)
completion, and then formally inverting Tate.

\begin{thm}
  For any smooth $S$, the categories $M_A(S)$ and $M_{pure}(S,\Q)$ are equivalent.
\end{thm}

We will give the proof when $S=Spec\, k$ for simplicity, even though
the arguments works in general. We will refer to a cohomology class in
$H^i(X,\Q)$ as a {\em Nori cycle} of weight $2j$ if it lies in the image of
$Hom_{M(k)}(\Q(-j), h^i(X))$.

\begin{proof}

This  is broken into a series of steps.

\begin{enumerate}
\item Motivated  cycles are Nori cycles.

  \begin{proof}
Algebraic cycles are certainly Nori cycles. In general,
    motivated cycles are built from algebraic cycles by applying $*, \cup, p_*$. Each of these
operations preserves the space of Nori cycles.
  \end{proof}

 \item There is an functor $\iota: M_A\to M_{pure}$ taking $h^i_A(X)$
   to $h^i(X)$. This functor commutes with $R_B$.

 \begin{proof}
By 1, the map on objects $X\to \oplus h^i(X)$ gives  a functor
$\iota':MotCor\to M_{pure}$. Since $M(k)$ is abelian  and the Tate motive is invertible, $\iota'$
extends uniquely to a functor $\iota : M_A\to M_{pure}$. The final
statement follow  more or less automatically from $R_B(h^i_A(X))=H^i(X)=R_B(h^i(X))$.
\end{proof}

\item There is a functor
$gr_A:M\to M_A$ taking $h^i(X)\to h^i_A(X)$. This satisfies $gr_A\circ
\iota= id$ and it commutes with $R_B$.

 \begin{proof}
The functor $gr_A$ is constructed by substituting $M_A$ for $M_{pure}$ in
the proof of theorem \ref{thm:MtoMpure}.
\end{proof}

\item  The functor $\iota$ takes simple objects to simple objects.

 \begin{proof} Suppose that $M\in M_A$ is simple, but $\iota(M)$ is not.  We may write
$\iota(M)= N_1\oplus N_2$ with $N_i\not= 0$. Then $M=gr_A(N_1)\oplus
gr_A(N_2)$ which leads to a contradiction.
\end{proof}

\item The set of simple objects of $M_A$ and $M_{pure}$ are in one to
  one correspondence via
$\iota$ and $gr_A$.

 \begin{proof} Write $h_A^i(X)=\oplus M_j$, with $M_j$ simple. Then
$h^i(X)=\oplus \iota(M_j)$ gives a decomposition into simple objects
by 4. Since every simple object of $M_{pure}$ is a summand of some $h^i(X)$,
with $X$ smooth and projective, this proves the claim.
\end{proof}

\item If $M\in M_A$ is simple, then $End(M)\cong End(\iota(M))$.

 \begin{proof}
 The map $g:End(\iota(M))\to End(M)$, induced by $gr_A$, is a surjective homomorphism
because $\iota$ gives a splitting. Since $End(\iota(M))$ is a division
ring, $g$ is necessarily an injection as well.
\end{proof}

 \item Given $M,N\in M_A$, $Hom(M,N)\cong Hom(\iota(M), \iota(N))$.

 \begin{proof}
Decompose $M=\bigoplus M_j^{m_j}\oplus M'$ and $N=\bigoplus
 M_j^{n_j}\oplus N'$ , such that $M_j$ are distinct simple
 objects and $M',N'$ have no simple factors in common. Let $D_j=End(M_j)$. Then 
 $$Hom(M,N) = \prod Mat_{n_jm_j}(D_j) = End(\iota(M))$$
\end{proof}

\end{enumerate}
\end{proof}

\section{Nori's Hodge conjecture}

\subsection{Conjecture over $\C$}
As is well known,
the usual form of the Hodge conjecture is equivalent to the statement
that the Hodge realization of homological pure motives is a full faithful
embedding \cite[p 405]{saavedra}.  In a nutshell, this comes down to the observation that 
given smooth projective varieties $X$ and $Y$, a morphism $Hom_{MHS}(H^i(X),
H^i(Y))$ is a Hodge cycle on $X\times Y$ and therefore a
correspondence, assuming the conjecture.  The analogous statement in
the  present setting is due to Nori.

\begin{conj}[Nori]
  The Hodge realization $R_H:\M(\C,\Q)\to MHS$ is a full faithful embedding.
\end{conj}

This would imply that the canonical mixed Hodge structure is ``Galois
invariant'' in the following sense: if $H^i(X)\cong  H^i(Y)$ in MHS,
then $H^i(X^\sigma)\cong H^i(Y^\sigma)$ for any $\sigma\in Aut(\C)$.
Since $\M(\C)=\M_{ls}(\C)$
and $MHS$ are
Tannakian, we can rewrite $Hom(A,B)= Hom(\Q,A^*\otimes B)$, and
reformulate the last conjecture as

\begin{conj}\label{conj:nori2}
The map
$$Hom_{\M(\C)}(\Q, M)\to Hom_{MHS}(\Q, R_H(M))$$ 
is surjective for each $M\in \M(\C)$. In particular,
a Hodge cycle on any complex algebraic variety $X$ is a Nori cycle.
\end{conj}

When $M$ lies in $\M_{pure}(\C)$, this is implied by the usual Hodge
conjecture, but in general, it
 is neither weaker nor stronger
than the Hodge conjecture. It should be viewed as refinement of  Deligne's conjecture that
Hodge cycles are absolute \cite{dmos}.  
To understand this from a different  perspective, let us recall that the original form of Beilinson's
Hodge conjecture \cite{beilinson1} would imply that the regulator map on the higher Chow group
$$reg:CH^a(X,b)\otimes \Q\to Hom_{MHS}(\Q(-a),H^{2a-b}(X))$$
is surjective for all $a,b$.  
The conjecture is known to be overly
optimistic in general  (cf \cite{jannsen}), but it is expected for instance
when $X$ is defined over $\bar \Q$.
The map can be made explicit as follows.
An element $\alpha$ on the left is given a cycle in Bloch's complex
\cite{bloch}, and so it possesses a fundamental class in
$$reg(\alpha)\in H^{2a}(X\times \mathbb{A}^b,X\times \partial \mathbb{A}^b)(a)\cong H^{2a-b}(X)(a)$$
where $\mathbb{A}^b$ is thought of as a simplex with boundary
$\partial \mathbb{A}^b$. It is clear from this, that
we may factor $reg$ through
$$  Hom_{\M(\C)}(\Q(-a),H^{2a-b}(X))\to
Hom_{MHS}(\Q(-a),H^{2a-b}(X))$$
Thus the truth of Beilinson's conjecture, in cases where it is
expected, would imply the truth of Nori's.

\begin{thm}[Andr\'e \cite{andre}]
  Hodge cycles on abelian varieties are motivated.
\end{thm}

So we deduce:

\begin{cor}
 Conjecture \ref{conj:nori2} holds for any variety whose motive lies
  in the tensor category generated by abelian varieties.
\end{cor}

\subsection{Conjecture over general bases}

We can formulate an ostensibly stronger form of Nori's conjecture.

\begin{conj}\label{conj:nori3}
Given a smooth complex variety $S$,
  \begin{enumerate}
  \item[(a)] the Hodge realization $R_H:\M_{ls}(S,\Q)\to VMHM(S)$ is a full
    faithful embedding.  
\item[(b)] if $M\in \M_{ls}(S)$, the map
$$Hom_{\M_{ls}(S)}(\Q_S,M)\to Hom_{VMHS}(\Q_S,R_H(M))$$
is surjective.
  \end{enumerate}
\end{conj}

As above, we note that  (a) and (b) are equivalent because $\M_{ls}(S)$ is Tannakian.
We will prove that the conjectures for a general base follows from the
earlier ones. 
As a first step, let us suppose that we have motive $M\in \M_{ls}(S)$,
 then by theorem~\ref{thm:directimage},
  we have a good direct image $p_*M=r^0p_*M\in \M(\C)$, where
$p:S\to Spec\, \C$ is the canonical map. 
Restricting the adjunction map $p^*p_*M\to M$ to $s\in S$, yields map
$p_*M\to M_s$.
Under Betti realization
$R_B(p_*M)\to R_B(M_s)$ can be identified with the inclusion of the subspace of
$\pi_1(S,s)$-invariants of the fibre of the local system $R_B(M)$. As
an aside, we observe that this
 leads to a direct construction for basic examples:

\begin{lemma}
  If  $M=h_S^i(X,E)$, then 
$$p_*M\cong \im [h^i(X,E)\to h^i(X_s,E_s)]$$
\end{lemma}

\begin{proof}
Let $I$ denote the right hand side. Consider the inclusion $X\to
X\times S$ of $S$-schemes given by the graph of $X\to S$.
This induces a map $h^i(X,E)= p_*h_S^i(X\times S,E\times S)\to p_*M$,
which is a morphism $I\to p_*M$. It is suffices to prove that this is
an isomorphism of Betti realizations. For this, apply 
  the global invariant cycle theorem \cite[6.2.8]{bbd}:
$$R_B(I)= \im [H^i(X,E)\to H^i(X_s,E_s)]= H^i(X_s,E_s)^{\pi_1(S,s)}$$
\end{proof}

\begin{thm}
Suppose that $S$ is smooth and connected with a point $s\in S$.
Given $M\in \M_{ls}(S)$, if conjecture \ref{conj:nori2} holds for  $p_*M$ then conjecture
\ref{conj:nori3}(b) holds for $M$.

\end{thm}

\begin{proof}
Let us suppose that conjecture
\ref{conj:nori2} holds for $I=p_*M$
Any morphism
$\gamma\in Hom(\Q_S,R_H(M))$ gives a $\pi_1(S,s)$ invariant weight  Hodge
cycle on $H^i(X_s,E_s)$. Thus $\gamma$ lies in
$Hom(\Q_S,R_H(I))$. So it must   come from a morphism $\Q\to I$ by our
assumption. This induces
 a  map of pullbacks $\Q_S\to p^*I$, which when composed 
 with the adjunction map $\epsilon:p^*I\to M$ gives the desired lift of $\gamma$
to $\gamma'\in Hom_{\M_{ls}(S)}(\Q_S,M)$.
\end{proof}

\begin{cor}
  Conjecture \ref{conj:nori2} implies conjecture \ref{conj:nori3}.
\end{cor}

\begin{cor}
  Conjecture \ref{conj:nori3}(b) for any motive that lies in the
  tensor category generated by relative smooth curves over $S$.
\end{cor}

By a similar argument we obtain an analogue of  Deligne's
``principle B'' \cite{dmos} in the theory of absolute Hodge cycles.

\begin{prop}
  Given a tame family $(f:X\to S,E)$, a $\pi_1(S,s)$-invariant Nori
 cycle in $H^i(X_s,E_s)$ yields, under parallel transport, a Nori cycle
in every fibre $H^i(X_t,E_t)$.
\end{prop}

\begin{proof}
  A $\pi_1(S,s)$-invariant Nori cycle on $H^i(X_s,E_s)$ induces a
  morphism $\Q(j)\to p_*h^i(X,E)$.  This can specialized to any fibre.
\end{proof}

%\newpage
\bigskip
\begin{center}
  \large{{\bf Appendices}} 
\end{center}
\begin{appendix}
  \section{$2$-categories}

We will generally take the view that a category {\em is} the same as any other category equivalent to it. 
This needs some elaboration.
The category $Cat$ of all small categories is a $2$-category \cite{maclane}.  Among other things,
this means that the set of functors $Hom_{Cat}(C,D)$ is itself a category, where
the morphisms  are natural transformations. In this setting, functors are usually called
$1$-morphisms and natural transformations $2$-morphisms. 
Each kind of morphism can be composed as usual. This is denoted by $\circ$.
There are identities denoted by $1_X$ etc. There is another kind of composition for
adjacent $2$-morphisms, denoted here by $\diamond$.
Given objects $A,B,C$, $1$-morphisms $F,F',G,G'$ and $2$-morphisms $\alpha,\beta$ as
indicated below
$$
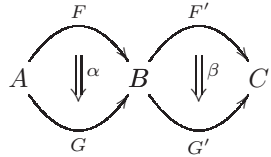
\xy 
(-16,0)*+{A}="4"; 
(0,0)*+{B}="6"; 
{\ar@/^1.65pc/^{F} "4";"6"}; 
{\ar@/_1.65pc/_{G} "4";"6"}; 
{\ar@{=>}^<<<{\scriptstyle \alpha} (-8,3)*{};(-8,-3)*{}} ; 
(0,0)*+{B}="4"; 
(16,0)*+{C}="6"; 
{\ar@/^1.65pc/^{F'} "4";"6"}; 
{\ar@/_1.65pc/_{G'} "4";"6"}; 
{\ar@{=>}^<<<{\scriptstyle \beta} (8,3)*{};(8,-3)*{}} ; 
\endxy 
$$
The composition $\beta\diamond\alpha$ sits as follows
$$
\xy 
(-8,0)*+{A}="4"; 
(8,0)*+{C}="6"; 
{\ar@/^1.65pc/^{F'\circ F} "4";"6"}; 
{\ar@/_1.65pc/_{G'\circ G} "4";"6"}; 
{\ar@{=>}^<<<{\scriptstyle \beta\diamond\alpha} (0,3)*{};(0,-3)*{}} ; 
\endxy 
$$
It is simply given as the composition
$$
F'(F(x))\stackrel{F'(\alpha_x)}{\longrightarrow} F'( G(x)) \stackrel{\beta_{G(x)}}{\longrightarrow} G'(G(x))
$$
This operation is associative, and the additional identities
$$\alpha\diamond 1_{1_A}= 1_{1_A}\diamond \alpha=\alpha,\quad
1_F\diamond 1_G = 1_{F\circ G}$$
$$(\beta'\circ \beta)\diamond(\alpha'\circ \alpha) = (\beta'\diamond\alpha')\circ (\beta\diamond\alpha)$$
are satisfied.

In $Cat$, we can either require that equations (or diagrams) hold (or commute) strictly, i.e. on the nose,
or only up to a natural isomorphism. The latter is frequently the more usual occurence.
Recall, for example, that categories $C$ and $D$ are equivalent (respectively isomorphic) if there are 
functors $F:C\to D$ and $G:D\to C$ such that $F\circ G \cong 1_D$ and $G\circ F\cong 1_G$
(respectively $F\circ G =1_D$ and $G\circ F=1_G$). Given a category $C$,
a pseudofunctor $F:C\to Cat$ is an assigment of objects to objects, and morphisms to
$1$-morphisms, together with natural isomorphisms, $\epsilon_c:F(1_c)\cong 1_{F(c)}$ and
$\eta_{f,g}:F(f)\circ F(g)\cong F(f\circ g)$.
These are required to satisfy certain commutivities that
 ensure that any two isomorphisms
$$F(f_1)\circ F(f_2)\circ\ldots F(f_n)\cong F(f_1\circ f_2 \circ\ldots f_n)$$
built from $\eta,\epsilon$ coincide. It suffices to check this for
$n\le 3$.
Contravariant pseudofunctors are simply pseudofunctors on $C^{op}$.

There are number of related notions of colimit (= direct limit) of categories.
We single out the notion that is most useful for us and refer to it as a
$2$-colimit, although ``pseudo-colimit'' or something like that may 
conform better to current usage. To simplify matters,
we discuss $2$-colimits in the filtered setting where things are
easier (cf. \cite[exp VI \S 6]{sga4}); this reference also gives a more  general construction.
A category $D$ is filtered if  for any two objects $d_1,d_2$, there exists an object $d_3$ and morphisms 
$d_1\to d_3, d_2\to d_3$, and
for any two parallel morphisms $f,g:d_1\to d_2$, there exists a morphism $h:d_2\to d_3$ such that $hf=hg$.
For example, a partially ordered  set is filtered precisely when it is directed. Given a pseudofunctor
$F:D\to Cat$ with $D$  filtered, the $2$-colimit 
$L= \text{2-}\varinjlim_{d} F(d) $ is the category  whose set of objects is the disjoint union
$$Ob L=\coprod_d Ob F(d)$$
The set of morphisms from $A\in Ob F(d_1)$ to $B\in ObF(d_2)$
is given by the filtered colimit
$$\varinjlim_{f:d_1\to d_3,g:d_2\to d_3} Hom_{F(d_3)}(F(f)(A), F(g)
(B))$$
We can see that there is a family of  $1$-morphisms $\{F(d)\to L\}$ such that for $f\in Hom(d,d')$
$$
 \xymatrix{
 F(d)\ar[r]\ar[d]_{F(f)} & L \\ 
 F(d')\ar[ru] & 
}
$$
commutes up to natural isomorphism. Moreover $L$ would be universal in the sense that for any 
category $L'$ with a family $\{F(d)\to L'\}$ as above, there is a unique
$1$-morphism $L\to L'$ such that the appropriate diagrams
nonstrictly commute. We really want to consider
$\text{2-}\varinjlim_{d} F(d) $ only  up to equivalence of
categories. In practice, there may be other representations of the
colimit which are more natural than the original construction.

\begin{lemma}
  Suppose that $C_i\subset C$ is a directed family of subcategories of
  a given category. Then $\2lim C_i$ is equivalent to the directed
  union $\bigcup C_i$ which the category having $\bigcup ObC_i$ and
$\bigcup MorC_i$ as its set of objects and morphisms.
\end{lemma}

\begin{ex}
  Suppose $E$ is a coalgebra over a field $F$. We can  express it as
  directed union, and therefore a $2$-colimit, of finite
  dimensional coalgebras $E=\varinjlim E_i$ . 
\end{ex}

From the earlier  descripition, it is not difficult to deduce the following:

\begin{prop}
Let $F$ be a pseudofunctor from a filtered category $D$ to the
$2$-category of abelian categories.  Suppose that that $F(f)$ is exact
for each $f\in MorD$. Then $\text{2-}\varinjlim_{d} F(d) $ is abelian
and the functors $ F(d')\to \text{2-}\varinjlim_{d} F(d) $ are exact.
\end{prop}

\begin{proof}[Sketch]
  It is clear that $L=\text{2-}\varinjlim_{d} F(d) $
and the functors $ F(d')\to L$ are
additive. Given a morphism in $L$ represented by
$f:A\to B$ in $F(d)$, $\ker(f)$, $coker(f)$ and $A\rightarrowtail im(f)\twoheadrightarrow B$
represents the kernel, cokernel and image factorization in $L$.
\end{proof}

In a similar vein:

\begin{prop}
Let $F$ be a pseudofunctor from a filtered category $D$ to the
$2$-category of triangulated categories (with $t$-structure) and
(exact) triangulated functors.
Then $\text{2-}\varinjlim_{d} F(d) $ carries the structure of a
triangulated category (with $t$-structure)
so that the  functors $F(d')\to \text{2-}\varinjlim_{d} F(d) $ are
triangulated (and exact).
\end{prop}

\section{Comparison theorem}

Let $X$ be a $\C$-variety. Define the site $X_{cl}$ with objects given by local homeomorphisms
$U\to X_{an}$ and coverings are surjective families $\{U_i\to U\}$. Then 
there is an obvious map of sites $X_{cl}\to X_{an}$ , which induces an equivalence of the categories of
sheaves \cite[exp XI \S 4]{sga4}. In 
particular, the cohomologies are the same. There is a canonical morphism of sites
$\epsilon:X_{cl}\to X_{et}$ which induces a map from \'etale to classical cohomology.

Since \'etale cohomology does not work properly for nontorsion coefficients, we start
with finite coefficients, and then take the limit.
Choose $N>0$.
A sheaf of $\Z/N\Z$-modules is constructible for either topology
if there is a decomposition of $X$ into Zariski locally closed sets, for which the restrictions
are locally constant. The pullback $\epsilon^*$ preserves constuctibility.
The following comparison theorem is given in  \cite[ exp XVI, thm 4.1; exp XVII, thm 5.3.3]{sga4}:

\begin{thm}
  Suppose that $f:X\to Y$ is a morphism of $\C$ varieties, and that
$\F$ is a constructible sheaf of $\Z/N\Z$-modules 
there  are isomorphisms
$$\epsilon^*R f^{et,i}_*\F \cong R f^{an,i}_* \epsilon^*\F$$
$$\epsilon^*R f^{et,i}_!\F \cong R f^{an,i}_! \epsilon^*\F$$
\end{thm}

Fix a prime $\ell$. A constructible $\ell$-adic sheaf is a system $\ldots \F_n\to \F_{n-1}\ldots$
of  sheaves on $X_{et}$, such that
each $\F_n$ is a constructible $\Z/\ell^n\Z$-module 
and the maps induce isomorphisms $\F_n\otimes\Z/\ell^{n-1}\Z\cong \F_{n-1}$.
Standard sheaf theoretic operations can  essentially  be defined componentwise, and they work as
expected \cite{sga5,deligne-sga}.
Given a constructible sheaf $\F=\{\F_n\}$ on $X_{et}$, define 
$$\epsilon^*\F = \varprojlim_n \epsilon^*\F_n$$
$$ \epsilon^*(\F \text{``}\otimes\text{''} \Q_\ell) = (\varprojlim_n \epsilon^*\F_n)\otimes \Q_\ell$$
on $X_{cl}$. Then with this notation, the above theorem extends
to  the case where $\F$ is an $\ell$-adic sheaf  ($\otimes \Q_\ell)$
\cite[\S 6.1]{bbd}.

\section{Classical $t$-structure for mixed Hodge modules}

Saito \cite{saito1, saito} has introduced a category of mixed Hodge
modules\footnote{To avoid confusion, we note that we are following the
  conventions of \S 4 of \cite{saito}.} $MHM(S)$.  
When $S$ is nonsingular, an object $\M$ of this category consists of a filtered perverse sheaf
$(K,W)$ of $\Q$-vector spaces on $S^{an}$
together with compatible bifiltered regular holonomic $D_S$-module
$(M,W,F)$. These are  subject to a rather delicate set of conditions
that we will not attempt to spell out. The definition is inductive.
In particular, when $S$ is a point, $MHM(S)$ is nothing but the  category of
polarizable mixed Hodge structures. One has a forgetful functor
$rat:MHM(S)\to Perv(S^{an})$ to the category of perverse sheaves
given by $\M\mapsto K$.
Saito has established the following properties:
\begin{enumerate}
\item There is an exact faithful functor $rat:MHM(S)\to Perv(S^{an})$
  for any $S$. It extends to a
triangulated functor $D^bMHM(S)\to D^bCons(S^{an})$.
\item $MHM(S)$ contains the category of admissible variations of mixed
Hodge structure. In fact, $M$ is a variation if and only if $rat(M)$
is a local system up to shift.
\item Standard sheaf theoretic operations extend 
to  $D^bMHM(S)$ including Grothendieck's ``six operations'' and
vanishing cycles functors. These are
compatible with the corresponding operations on $D^bCons(S^{an})$
via $rat$.
\end{enumerate}

The most natural $t$-structure on $D^bMHM(S)$ has $MHM(S)$ as its
heart. This  corresponds to the
perverse $t$-structure on the constructible derived category, so we
refer to this as the perverse $t$-structure on $MHM$. 
Saito \cite[remark 4.6]{saito} has pointed out that there this a second
$t$-structure that we call the classical $t$-structure which lifts the
standard $t$-structure on  $D^bCons(S^{an})$ with $Cons(S^{an})$
as its heart.

\begin{thm}
There exists a nondegenerate $t$-structure $({}^cD^{\le 0}, {}^cD^{\ge 0})$ on
$D^bMHM(S)$ which is compatible with the
standard $t$-structure on  $D^bCons(S^{an})$.
\end{thm}
 
\begin{proof}
Let $i_x:x\to X$ denote the inclusion of a point.
Define $M\in Ob  ({}^cD^{\le 0})$ (respectively $\in Ob({}^cD^{\ge 0})$) if
$i_x^*M= 0$ for all $x\in X$ and $k>0$ (respectively $k<0$). 
To see that this is a $t$-structure, note that 
it is enough to check this on each step of the filtered union
$D^bMHM(S)=\bigcup D^bMHM(S,\Sigma)$,
where  $MHM(S,\Sigma)\subset MHM(S)$
denotes  the full subcategory consisting of mixed Hodge modules such
that $rat(\M)$ is $\Sigma$-construcible.
One can now use induction on the cardinality $|\Sigma|$.
If $|\Sigma|=1$, the purported $t$-structure is in fact what it is
claimed to be since it  is the perverse $t$-structure up to shift.
When $|\Sigma|>1$, let $T$ be a closed stratum and $U= S-T$.
By induction, $({}^cD^{\le 0}, {}^cD^{\ge 0})$ determine
$t$-structures on $T$ and $U$. For $S$ this follows by
 verifying the conditions of \cite[thm
1.4.10]{bbd} using \cite[4.4.1]{saito}.

$({}^cD^{\le 0}, {}^cD^{\ge 0})$
is clearly compatible with the
standard $t$-structure on  $D^bCons(S^{an})$.
\end{proof}

Let $Cons\text{-}MHM(S)$ denote the heart  ${}^cD^{\le 0}MHM(S)\cap
{}^cD^{\ge 0}MHM(S)$, and likewise for $Cons\text{-}MHM(S,\Sigma)$.

\begin{lemma}
     The functor 
$$rat:Cons\text{-}MHM(S,\Sigma)\to  Cons(S^{an},\Sigma)$$
yields an exact faithful embedding. 
  \end{lemma}

  \begin{proof}
Exactness is already clear from the theorem, so
    the only issue is faithfulness. This can be proved by induction on
    $|\Sigma|$. This holds when    $|\Sigma|=1$ because
$Cons\text{-}MHM(S,\Sigma)$ is the category of variations of Hodge structures.
In general, let $i:T\to S$ be a closed stratum and $j:U\to S$ be the
complement. Suppose that $f\in Hom(M,N)$ is morphism in
$Cons\text{-}MHM(S,\Sigma)$  such that $rat(f)=0$. We need to prove
that $f=0$. By induction
$f|_T=0$ and $f|_U=0$.
From the  distinguished triangle
$$j_!j^*M\to M\to i_*i^*M\to j_!j^*M[1]$$
 and  adjointness  we obtain an exact  sequence
$$
\xymatrix{
Hom(i_*i^*M,N)\ar[r]\ar^{||}[d]& Hom(M,N)\ar[r]& Hom(j_!j^*M,N)\ar^{||}[d]\\
Hom(i^*M, i^*N) & & Hom(j^*M,j^*N)}
$$
Therefore $f=0$.
  \end{proof}

\end{appendix}

%%%

\noindent  Department of Mathematics,
   Purdue University,   West Lafayette IN 47907,
   U.S.A.\\ arapura@math.purdue.edu

\end{document}